\documentclass[10pt]{amsart}
\usepackage{etex}
\usepackage{mathtools,tensor}
\usepackage{latexsym,amssymb,amsthm,amsmath,enumerate,epsfig,setspace,bbding,color,graphicx,caption}
\usepackage[noadjust]{cite}
\usepackage{stmaryrd}
\usepackage[all]{xy}
\usepackage[retainorgcmds]{IEEEtrantools}
\usepackage{hyperref}
\usepackage{subfiles}
\usepackage[margin=10pt,font=small,labelfont=bf]{caption}
\usepackage{tikz-cd}
\usepackage{stmaryrd}
\usetikzlibrary{arrows}

\newtheorem{thm}{Theorem}[section]

\newtheorem{cor}[thm]{Corollary}
\newtheorem{prop}[thm]{Proposition}
\newtheorem{lemma}[thm]{Lemma}
\newtheorem*{Q}{Question}

\theoremstyle{definition}
\newtheorem{defi}[thm]{Definition}

\newtheorem{Ex}{Example}
\newtheorem{ex}[thm]{Example}
\newtheorem{rem}[thm]{Remark}
\newtheorem*{Rem}{Remark}

\pgfmathsetmacro{\shift}{0.65ex}
\renewcommand{\L}{\mathcal{L}}

\newcommand{\G}{\mathcal{G}}

\newcommand{\g}{\mathfrak{g}}
\newcommand{\h}{\mathfrak{h}}
\renewcommand{\t}{\mathfrak{t}}

\newcommand{\X}{\mathcal{X}}
\renewcommand{\H}{\mathcal{H}}

\newcommand{\T}{\mathbb{T}}

\newcommand{\No}{\mathcal{N}}

\newcommand{\R}{\mathbb{R}}
\newcommand{\C}{\mathbb{C}}

\newcommand{\Z}{\mathbb{Z}}

\renewcommand{\dim}{\text{dim}}
\renewcommand{\ker}{\textrm{Ker}}
\newcommand{\im}{\textrm{Im}}
\newcommand{\coker}{\text{CoKer}}

\renewcommand{\O}{\mathcal{O}}

\newcommand{\m}{\mathfrak{m}}
\renewcommand\S{\mathcal{S}}
\renewcommand\phi{\varphi}

\renewcommand{\d}{\textrm{d}}

\renewcommand{\P}{\mathcal{P}}
\newcommand{\J}{\mathcal{J}}

\newcommand{\p}{\mathfrak{p}}
\renewcommand{\c}{\mathfrak{c}}

\begin{document}


\newcommand{\Addresses}{{
  \bigskip
  \footnotesize

  \textsc{Mathematisch Instituut, Universiteit Utrecht, The Netherlands}\par\nopagebreak
  \textit{E-mail address}: \texttt{maarten.mol.math@gmail.com}

}}

\title{Stratification of the transverse momentum map}
\author{Maarten Mol\\
Universiteit Utrecht}
\date{}
\maketitle
\begin{abstract} Given a Hamiltonian action of a proper symplectic groupoid (for instance, a Hamiltonian action of a compact Lie group), we show that the transverse momentum map admits a natural constant rank stratification. To this end, we construct a refinement of the canonical stratification associated to the Lie groupoid action (the orbit type stratification, in the case of a Hamiltonian Lie group action) that seems not to have appeared before, even in the literature on Hamiltonian Lie group actions. This refinement turns out to be compatible with the Poisson geometry of the Hamiltonian action: it is a Poisson stratification of the orbit space, each stratum of which is a regular Poisson manifold that admits a natural proper symplectic groupoid integrating it. The main tools in our proofs (which we believe could be of independent interest) are a version of the Marle-Guillemin-Sternberg normal form theorem for Hamiltonian actions of proper symplectic groupoids and a notion of equivalence between Hamiltonian actions of symplectic groupoids, closely related to Morita equivalence between symplectic groupoids. 
\end{abstract}
\tableofcontents

\section*{Introduction}
Traditionally, a Hamiltonian action is an action of a Lie group $G$ on a symplectic manifold $(S,\omega)$, equipped with an equivariant momentum map: 
\begin{equation*} J:(S,\omega)\to \g^*,
\end{equation*} taking values in the dual of the Lie algebra of $G$. Throughout the years, variations on this notion have been explored, many of which have the common feature that the momentum map:
\begin{equation}\label{mommapintro0} J:(S,\omega)\to (M,\pi),
\end{equation} is a Poisson map taking values in a specified Poisson manifold (see for instance \cite{McD,Lu,MeWo,We6}). In \cite{WeMi}, such momentum map theories were unified by introducing the notion of Hamiltonian actions for symplectic groupoids, in which the momentum map takes values in the Poisson manifold integrated by a given symplectic groupoid. In this paper, we show that the transverse momentum map of such Hamiltonian actions admits a natural stratification, provided the given symplectic groupoid is proper. To be more precise, let $(\G,\Omega)\rightrightarrows (M,\pi)$ be a proper symplectic groupoid with a Hamiltonian action along a momentum map (\ref{mommapintro0}). The symplectic groupoid generates a partition of $(M,\pi)$ into symplectic manifolds, called the symplectic leaves of $(\G,\Omega)$. On the other hand, the $(\G,\Omega)$-action generates the partition of $(S,\omega)$ into orbits. We denote the spaces of orbits and leaves as: 
\begin{equation*} \underline{S}:=S/\G \quad\&\quad \underline{M}:=M/\G.
\end{equation*}
The momentum map (\ref{mommapintro0}) descends to a map:
\begin{center}
\begin{tikzcd} (S,\omega)\arrow[r,"J"] \arrow[d] & (M,\pi) \arrow[d]\\
\underline{S}\arrow[r,"\underline{J}"] & \underline{M}
\end{tikzcd}
\end{center}
that we call the \textbf{transverse momentum map}. Because we assume $\G$ to be proper, by the results of \cite{PfPoTa,CrMe} (which we recall in Section \ref{stratsec}) both the orbit space $\underline{S}$ and the leaf space $\underline{M}$ admit a canonical Whitney stratification: $\S_\textrm{Gp}(\underline{S})$ and $\S_\textrm{Gp}(\underline{M})$, induced by the proper Lie groupoids $\G\ltimes S$ (the action groupoid) and $\G$. These, however, do not form a stratification of the transverse momentum map, in the sense that $\underline{J}$ need not send strata of $\S_\textrm{Gp}(\underline{S})$ into strata of $\S_\textrm{Gp}(\underline{M})$ (see Example \ref{hamGspex} below). Our first main result is Theorem \ref{canhamstratthm}, which shows that there is a natural refinement $\S_\textrm{Ham}(\underline{S})$ of $\S_\textrm{Gp}(\underline{S})$ that, together with the stratification $\S_\textrm{Gp}(\underline{M})$, forms a constant rank stratification of $\underline{J}$. This means that: 
\begin{itemize}\item $\underline{J}$ sends strata of $\S_\textrm{Ham}(\underline{S})$ into strata of $\S_\textrm{Gp}(\underline{M})$,
\item the restriction of $\underline{J}$ to each pair of strata is a smooth map of constant rank.
\end{itemize} Theorem \ref{canhamstratthm} further shows that $\S_\textrm{Ham}(\underline{S})$ is in fact a Whitney stratification of the orbit space. We call $\S_\textrm{Ham}(\underline{S})$ the \textbf{canonical Hamiltonian stratification} of $\underline{S}$. 
\begin{Ex}\label{hamGspex} Let $G$ be a compact Lie group with Lie algebra $\g$ and let $J:(S,\omega)\to \g^*$ be a Hamiltonian $G$-space with equivariant momentum map. In this case, $(\G,\Omega)=(T^*G,-\d\lambda_\textrm{can})$ (cf. Example \ref{exhamGsp}), $\underline{S}=S/G$, $\underline{M}=\g^*/G$, and $\S_\textrm{Gp}(\underline{S})$ and $\S_\textrm{Gp}(\underline{M})$ are the stratifications by connected components of the orbit types of the $G$-actions. The stratification $\S_\textrm{Ham}(\underline{S})$ can be described as follows. Let us call a pair $(K,H)$ of subgroups $H\subset K\subset G$ conjugate in $G$ to another such pair $(K',H')$ if there is a $g\in G$ such that $gKg^{-1}=K'$ and $gHg^{-1}=H'$. Consider the partition of $\underline{S}$ defined by the equivalence relation: 
\begin{equation}\label{eqrelorbleaftyp} \O_p\sim \O_q\iff (G_{J(p)},G_p) \text{ is conjugate in $G$ to }(G_{J(q)},G_q),
\end{equation} where $G_p$ and $G_q$ denote the isotropy groups of the action on $S$, whereas $G_{J(p)}$ and $G_{J(q)}$ denote the isotropy groups of the coadjoint action on $\g^*$. The connected components of the members of this partition form the stratification $\S_\textrm{Ham}(\underline{S})$. When $G$ is abelian, $\S_\textrm{Ham}(\underline{S})$ and $\S_\textrm{Gp}(\underline{S})$ coincide, but in general they need not (consider, for example, the cotangent lift of the action by left translation of a non-abelian compact Lie group $G$ on itself).
\end{Ex}

Our second main result is Theorem \ref{poisstratthm}$b$, which states that $\S_\textrm{Ham}(\underline{S})$ is in fact a constant rank Poisson stratification of the orbit space and gives a description of the symplectic leaves in terms of the fibers of the transverse momentum map. To elaborate, let us first provide some further context. The singular space $\underline{S}$ has a natural algebra of smooth functions $C^\infty(\underline{S})$: the algebra consisting of $\G$-invariant smooth functions on $S$. This is a Poisson subalgebra of: 
\begin{equation*} (C^\infty(S),\{\cdot,\cdot\}_\omega).
\end{equation*} Hence, it inherits a Poisson bracket, known as the reduced Poisson bracket. Geometrically, this is reflected by the fact that $\S_\textrm{Gp}(\underline{S})$ is a Poisson stratification of the orbit space (see Definition \ref{poisstratdef} and Theorem \ref{poisstratthm}$a$). In particular, each stratum of $\S_\textrm{Gp}(\underline{S})$ admits a natural Poisson structure, induced by the Poisson bracket on $C^\infty(\underline{S})$. Closely related to this is the singular symplectic reduction procedure of Lerman-Sjamaar \cite{LeSj}, which states that for each symplectic leaf $\L$ of $(\G,\Omega)$ in $M$, the symplectic reduced space: 
\begin{equation}\label{sympredspintro} \underline{S}_\L:=J^{-1}(\L)/\G
\end{equation} admits a natural symplectic Whitney stratification. Let us call this the Lerman-Sjamaar stratification of (\ref{sympredspintro}). This is related to the Poisson stratification $\S_\textrm{Gp}(\underline{S})$ by the fact that each symplectic stratum of such a reduced space $(\ref{sympredspintro})$ coincides with a symplectic leaf of a stratum of $\S_\textrm{Gp}(\underline{S})$.
\begin{Rem} The facts mentioned above are stated more precisely in Theorems \ref{poisstratthm}$a$, \ref{redspstratthm} and \ref{poisstratthm}$c$. Although these theorems should be known to experts, in the literature we could not find a written proof (that is, not in the generality of Hamiltonian actions for symplectic groupoids; see e.g. \cite{FeOrRa,LeSj} for the case of Lie group actions). Therefore, we have included proofs of these. 
\end{Rem}

Returning to our second main result: Theorem \ref{poisstratthm}$b$ states first of all that, like $\S_\textrm{Gp}(\underline{S})$, the canonical Hamiltonian stratification $\S_\textrm{Ham}(\underline{S})$ is a Poisson stratification of the orbit space, the leaves of which coincide with symplectic strata of the Lerman-Sjamaar stratification of the reduced spaces (\ref{sympredspintro}). In addition, it has the following properties:
\begin{itemize} 
\item in contrast to $\S_\textrm{Gp}(\underline{S})$, the Poisson structure on each stratum of $\S_\textrm{Ham}(\underline{S})$ is regular (meaning that the symplectic leaves have constant dimension),
\item the symplectic foliation on each stratum $\underline{\Sigma}\in \S_\textrm{Ham}(\underline{S})$ coincides, as a foliation, with that by the connected components of the fibers of the constant rank map $\underline{J}\vert_{\underline{\Sigma}}$.
\end{itemize}
The reduced spaces $(\ref{sympredspintro})$ are, as topological spaces, the fibers of $\underline{J}$. As stratified spaces (equipped with the Lerman-Sjamaar stratification), these can now be seen as the fibers of the  stratified map: \begin{equation*} \underline{J}:(\underline{S},\S_\textrm{Ham}(\underline{S}))\to (\underline{M},\S_\textrm{Gp}(\underline{M})).
\end{equation*} 
Our third main result is Theorem \ref{poisstratintgrthm}, which says that, besides the fact that the Poisson structure on each stratum of $\S_\textrm{Ham}(\underline{S})$ is regular, these Poisson manifolds admit natural proper symplectic groupoids integrating them. 

\begin{Ex}\label{identityactionex} Let $(\G,\Omega)\rightrightarrows (M,\pi)$ be a proper symplectic groupoid. Then $(\G,\Omega)$ has a canonical (left) Hamiltonian action on itself along the target map $t:(\G,\Omega)\to M$. In this case, $(S,\omega)=(\G,\Omega)$ and the orbit space $\underline{S}$ is $M$, with orbit projection the source map of $\G$. The stratification $\S_\textrm{Ham}(\underline{S})$ is the canonical stratification $\S_\textrm{Gp}(M)$ induced by the proper Lie groupoid $\G$ (as in Example \ref{exmortyp}). So, Theorem \ref{poisstratthm} and \ref{poisstratintgrthm} imply that each stratum of $\S_\textrm{Gp}(M)$ is a regular, saturated Poisson submanifold of $(M,\pi)$, that admits a natural proper symplectic groupoid integrating it. This is a result in \cite{CrFeTo2}.
\end{Ex}

Regular proper symplectic groupoids have been studied extensively in \cite{CrFeTo} and have been shown to admit a transverse integral affine structure. In particular, the proper symplectic groupoids over the strata of the canonical Hamiltonian stratification admit transverse integral affine structures. As it turns out, the leaf space of the proper symplectic groupoid over any stratum of $\S_\textrm{Ham}(\underline{S})$ is smooth, and the transverse momentum map descends to an integral affine immersion into the corresponding stratum of $\S_\textrm{Gp}(\underline{M})$. This is reminiscent of the findings of \cite{CoDaMo,Zu}. \\

\textbf{\underline{Brief outline:}} In Part 1 we generalize the Marle-Guillemin-Sternberg normal form for Hamiltonian actions of Lie groups, to those of symplectic groupoids (Theorem \ref{normhamthm}). From this we derive a simpler normal form for the transverse momentum map (Example \ref{locmodmoreq}), using a notion of equivalence for Hamiltonian actions that is analogous to Morita equivalence for Lie groupoids (Definition \ref{moreqdefHam}). Part 1 provides the main tools for the proofs in Part 2, where we introduce the canonical Hamiltonian stratification and prove the main theorems mentioned above (Theorems \ref{canhamstratthm}, \ref{redspstratthm}, \ref{poisstratthm} and \ref{poisstratintgrthm}). A more detailed outline is given at the start of each of these parts.\\ 


\textbf{\underline{Acknowledgements:}}
I would like to thank my PhD supervisor Marius Crainic for his guidance. Marius suggested to me to try to prove the aforementioned normal form theorem by means of Theorem \ref{sympmoreqasrep}. Moreover, he commented on an earlier version of this paper, which most certainly helped me to improve the presentation. I would further like to thank him, Rui Loja Fernandes and David Mart\'{i}nez Torres for sharing some of their unpublished work with me, and I am grateful to David and Rui for their lectures at the summer school of Poisson 2018; all of this has been an important source of inspiration for Theorem \ref{poisstratintgrthm}. This work was supported by NWO (Vici Grant no. 639.033.312).\\

\textbf{\underline{Conventions:}} Throughout, we require smooth manifolds to be both Hausdorff and second countable and we require the same for both the base and the space of arrows of a Lie groupoid. 

\section{The normal form theorem}
In this part we prove a version of the Marle-Guillemin-Sternberg normal form theorem for Hamiltonian actions of symplectic groupoids. 
\begin{thm}\label{normhamthm} Let $(\G,\Omega)\rightrightarrows M$ be a symplectic groupoid and suppose that we are given a Hamiltonian $(\G,\Omega)$-action along $J:(S,\omega)\to M$. Let $\O$ be the orbit of the action through some $p\in S$ and $\L$ the leaf of $\G$ through $x:=J(p)$. If $\G$ is proper at $x$ (in the sense of Definition \ref{propatxdefi}), then the Hamiltonian action is neighbourhood-equivalent (in the sense of Definition \ref{nhoodeq1defi}) to its local model around $\O$ (as constructed in Section \ref{locmodsec}).
\end{thm} 
Both the local model and the proof of this theorem are inspired on those of two existing normal form theorems: the MGS-normal form \cite{Ma1,GS4} by Marle, and Guillemin and Sternberg on one hand, and the normal form for proper Lie groupoids \cite{We4,Zu,CrStr,FeHo} and symplectic groupoids \cite{Zu,CrMar1,CrFeTo1,CrFeTo2} on the other.\\

We split the proof of this theorem into a rigidity theorem (Theorem \ref{righamthm}) and the construction of a local model out of a certain collection of data that can be associated to any orbit $\O$ of a Hamiltonian action. In Section \ref{backhamactsec} and Section \ref{normrephamsec} we introduce the reader to this data and in Section \ref{locmodsec} we construct the local model. 
To prove Theorem \ref{normhamthm}, we are then left to prove the rigidity theorem, which is the content of Section \ref{normformpfsubsec}. 
Lastly, in Section \ref{translocmodsec} we introduce a notion of Morita equivalence between Hamiltonian actions that allows us to make sense of a simpler normal form for the transverse momentum map. We then study some elementary invariants for this notion of equivalence, analogous to those for Morita equivalence between Lie groupoids, which will lead to further insight into the proof of Theorem \ref{normhamthm}. This will also be important later in our definition of the canonical Hamiltonian stratification and our proof of Theorem \ref{canhamstratthm} and Theorem \ref{redspstratthm}.

\subsection{Background on Hamiltonian groupoid actions}\label{backhamactsec}
\subsubsection{Poisson structures and symplectic groupoids}\label{backhamactsubsec1} 
Recall that a \textbf{symplectic groupoid} is a pair $(\G,\Omega)$ consisting of a Lie groupoid $\G$ and a symplectic form $\Omega$ on $\G$ which is \textbf{multiplicative}. That is, it is compatible with the groupoid structure in the sense that:
\begin{equation*} (\textrm{pr}_1)^*\Omega=m^*\Omega-(\textrm{pr}_2)^*\Omega,
\end{equation*} where we denote by:
\begin{equation*} m,\textrm{pr}_1,\textrm{pr}_2:\G^{(2)}\to \G
\end{equation*} the groupoid multiplication and the projections from the space of composable arrows $\G^{(2)}$ to $\G$. Given a symplectic groupoid $(\G,\Omega)\rightrightarrows M$, there is a unique Poisson structure $\pi$ on $M$ with the property that the target map $t:(\G,\Omega)\to (M,\pi)$ is a Poisson map. The Lie algebroid of $\G$ is canonically isomorphic to the Lie algebroid $T^*_\pi M$ of the Poisson structure $\pi$ on $M$, via:
\begin{equation}\label{imsymp} \rho_\Omega:T^*_\pi M\to \text{Lie}(\G), \quad \iota_{\rho_\Omega(\alpha)}\Omega_{1_x}=(\d t_{1_x})^*\alpha, \quad \forall\alpha\in T^*_xM,\text{ } x\in M.
\end{equation} The symplectic groupoid $(\G,\Omega)$ it said to integrate the Poisson structure $\pi$ on $M$.
\begin{ex}\label{liealgintex} The dual of a Lie algebra $\g$ is naturally a Poisson manifold $(\g^*,\pi_{\textrm{lin}})$, equipped with the so-called Lie-Poisson structure. Given a Lie group $G$ with Lie algebra $\g$, the cotangent groupoid $(T^*G,-\d\lambda_{\textrm{can}})$ is a symplectic groupoid integrating $(\g^*,\pi_{\textrm{lin}})$. The groupoid structure on $T^*G$ is determined by that fact that, via left-multiplication on $G$, it is isomorphic to the action groupoid $G\ltimes \g^*$ of the coadjoint action. 
\end{ex} 
\subsubsection{Momentum maps and Hamiltonian actions} To begin with, recall:
\begin{defi}[\cite{WeMi}]\label{hamactdef} Let $(S,\omega)$ be a symplectic manifold. A left action of a symplectic groupoid $(\G,\Omega)\rightrightarrows M$ along a map $J:(S,\omega)\to M$ is called \textbf{Hamiltonian} if it satisfies the multiplicativity condition:
\begin{equation}\label{hammultcond} (\textrm{pr}_\G)^*\Omega=(m_S)^*\omega-(\textrm{pr}_S)^*\omega, 
\end{equation} where we denote by:
\begin{equation*} m_S,\textrm{pr}_S:\G\ltimes S\to S,
\quad \textrm{pr}_\G:\G\ltimes S\to \G,
\end{equation*} the map defining the action and the projections from the action groupoid to $S$ and $\G$. Right Hamiltonian actions are defined similarly.
\end{defi} The infinitesimal version of Hamiltonian actions for symplectic groupoids are momentum maps. To be more precise, by a \textbf{momentum map} we mean a Poisson map $J:(S,\omega)\to (M,\pi)$ from a symplectic manifold into a Poisson manifold. That is, for all $f,g\in C^\infty(M)$ it holds that:
\begin{equation*} J^*\{f,g\}_\pi=\{J^*f,J^*g\}_\omega.
\end{equation*}
Every momentum map comes with a symmetry, in the form of a Lie algebroid action. Indeed, a momentum map $J:(S,\omega)\to (M,\pi)$ is acted on by the Lie algebroid $T^*_\pi M$ of the Poisson structure $\pi$. Explicitly, the Lie algebroid action $a_J:\Omega^1(M)\to \X(S)$ along $J$ is determined by the \textbf{momentum map condition}:
\begin{equation}\label{mommapcond} \iota_{a_J(\alpha)}\omega=J^*\alpha, \quad \forall \alpha\in \Omega^1(M).
\end{equation}
Hamiltonian actions integrate such Lie algebroid actions, in the following sense. 
\begin{prop} Let $(\G,\Omega)\rightrightarrows M$ be a symplectic groupoid and let $\pi$ be the induced Poisson structure on $M$ (as in Subsection \ref{backhamactsubsec1}). Suppose that we are given a left Hamiltonian $(\G,\Omega)$-action along $J:(S,\omega)\to M$. Then $J:(S,\omega)\to (M,\pi)$ is a momentum map and the Lie algebroid action: \begin{equation}\label{assliealgact} a:\Omega^1(M)\to \X(S)
\end{equation} associated to the Lie groupoid action (via (\ref{imsymp})) coincides with the canonical $T^*_\pi M$-action along $J$. In other words, (\ref{assliealgact}) satisfies the momentum map condition (\ref{mommapcond}). A similar statement holds for right Hamiltonian actions. 
\end{prop} An appropriate converse to this statement holds as well; see for instance \cite{BuCr}.  
\begin{ex}\label{exhamGsp} Continuing Example \ref{liealgintex}: as observed in \cite{WeMi}, the data of a Hamiltonian $G$-action with equivariant momentum map $J:(S,\omega)\to \g^*$ is the same as that of a Hamiltonian action of the symplectic groupoid $(G\ltimes \g^*,-\d\lambda_{\textrm{can}})$ along $J$. 
\end{ex}
\begin{ex} Any symplectic groupoid has canonical left and right Hamiltonian actions along its target and source map, respectively.
\end{ex}
\subsection{The local invariants}\label{normrephamsec}
\subsubsection{The leaves and normal representations of Lie and symplectic groupoids} To start with, we introduce some more terminology. Let $\G\rightrightarrows M$ be a Lie groupoid and $x\in M$. By the \textbf{leaf of $\G$ through $x$} we mean the set $\L_x$ consisting of points in $M$ that are the target of an arrow starting at $x$. By the \textbf{isotropy group of $\G$} at $x$ we mean the group $\G_x:=s^{-1}(x)\cap t^{-1}(x)$ consisting of arrows that start and end at $x$. In general, $\G_x$ is a submanifold of $\G$ and as such it is a Lie group. The leaf $\L_x$ is an initial submanifold of $M$, with smooth manifold structure determined by the fact that: \begin{equation}\label{sfibprinbun} t:s^{-1}(x)\to \L_x\end{equation} is a (right) principal $\G_x$-bundle. Notice that a leaf of $\G$ may be disconnected. Given a leaf $\L\subset M$ of $\G$, we let $\G_\L:=s^{-1}(\L)$
denote the restriction of $\G$ to $\L$. This is a Lie subgroupoid of $\G$. In all of our main theorems, we assume at least that $\G$ is proper at points in the leaves under consideration, in the sense below.
\begin{defi}[\cite{CrStr}]\label{propatxdefi} A Hausdorff Lie groupoid $\G$ is called \textbf{proper at $x\in M$} if the map \begin{equation*} (t,s):\G\to M\times M
\end{equation*} is proper at $(x,x)$, meaning that any sequence $(g_n)$ in $\G$ such that $(t(g_n),s(g_n))$ converges to $(x,x)$ admits a convergent subsequence.
\end{defi} If $\G$ is proper at some (or equivalently every) point $x\in \L$, then $\L$ and the Lie subgroupoid $\G_\L$ are embedded submanifolds of $M$ and $\G$ respectively, and the isotropy group $\G_x$ is compact. Returning to a general leaf $\L$, the normal bundle $\No_\L$ to the leaf in $M$ is naturally a representation:
\begin{equation*} \No_\L\in \textrm{Rep}(\G_\L)
\end{equation*} of $\G_\L$, with the action defined as:
\begin{equation}\label{normrepleaf} g\cdot[v]=[\d t(\hat{v})]\in \No_{t(g)}, \quad g\in \G_\L,\text{ } [v]\in \No_{s(g)},
\end{equation} where $\hat{v}\in T_g\G$ is any tangent vector satisfying $\d s(\hat{v})=v$. We call this the \textbf{normal representation of $\G$ at $\L$}. It encodes first order data of $\G$ in directions normal to $\L$ (see also \cite{CrStr}). Given $x\in \L$, so that $\L=\L_x$, this restricts to a representation:
\begin{equation}\label{normreppt} \No_x\in \textrm{Rep}(\G_x)
\end{equation}
of the isotropy group $\G_x$ on the fiber $\No_x$ of $\No_\L$ over $x$, which we refer to as the \textbf{normal representation of $\G$ at $x$}. Without loss of information, one can restrict attention to the normal representation at a point, which will often be more convenient for our purposes. This is because the transitive Lie groupoid $\G_\L$ is canonically isomorphic to the gauge-groupoid of the principal bundle (\ref{sfibprinbun}), and the normal bundle $\No_\L$ is canonically isomorphic to the vector bundle associated to the principal bundle (\ref{sfibprinbun}) and the representation (\ref{normreppt}). 
\begin{ex} For the holonomy groupoid of a foliation (assumed to be Hausdorff here), the leaves are those of the foliation and the normal representation at $x$ is the linear holonomy representation (the linearization of the holonomy action on a transversal through $x$).
\end{ex}
\begin{ex} For the action groupoid of a Lie group action, the leaves are the orbits of the action and the normal representation at $x$ is simply induced by the isotropy representation on the tangent space to $x$. 
\end{ex} 
For a symplectic groupoid the basic facts stated below hold, which follow from multiplicativity of the symplectic form on the groupoid (see e.g. \cite {BuCrWeZh} for background on multiplicative $2$-forms).
\begin{prop}\label{normrepsymp} Let $(\G,\Omega)\rightrightarrows M$ be a symplectic groupoid and let $\pi$ be the induced Poisson structure on $M$. Let $x\in M$, let $\L$ be the leaf of $\G$ through $x$ and $\G_\L$ the restriction of $\G$ to $\L$. 
\begin{itemize}\item[a)] There is a unique symplectic form $\omega_\L$ on $\L$ such that:
\begin{equation*} \Omega\vert_{\G_\L}=t^*\omega_\L-s^*\omega_\L\in \Omega^2(\G_\L).
\end{equation*} The connected components of $(\L,\omega_\L)$ are symplectic leaves of the Poisson manifold $(M,\pi)$.
\item[b)] The normal representation (\ref{normreppt}) is isomorphic (via (\ref{imsymp})) to the coadjoint representation:
\begin{equation*} \g_x^*\in \textrm{Rep}(\G_x).
\end{equation*}  
\end{itemize}
\end{prop} 
\subsubsection{The orbits, leaves and normal representations of Hamiltonian actions} Next, we will study the leaves and the normal representations for the action groupoid of a Hamiltonian action. Let $(\G,\Omega)\rightrightarrows M$ be a symplectic groupoid and suppose that we are given a left Hamiltonian $(\G,\Omega)$-action along $J:(S,\omega)\to M$. Let $p\in S$, $x:=J(p)\in M$, let $(\L,\omega_\L)$ be the symplectic leaf of $(\G,\Omega)$ through $x$ (as in Proposition \ref{normrepsymp}) and let $\G_x$ be the isotropy group of $\G$ at $x$. By the \textbf{orbit of the action through $p$} we mean:
\begin{equation*} \O_p:=\{g\cdot p\mid g\in s^{-1}(x)\}\subset S,
\end{equation*} and by the \textbf{isotropy group of the $\G$-action at $p$} we mean the closed subgroup:
\begin{equation*} \G_p:=\{g\in \G_x\mid g\cdot p=p\}\subset \G_x.
\end{equation*} Note that these coincide with the leaf and the isotropy group at $p$ of the action groupoid. We let 
\begin{equation}\label{normrepactgpoid} \No_p\in \textrm{Rep}(\G_p)
\end{equation} denote the normal representation of the action groupoid at $p$. There are various relationships between the orbits, leaves and the normal representations at $p$ and $x$. To state these, consider the symplectic normal space to the orbit $\O$ at $p$: 
\begin{equation}\label{sympnormsp} \S\No_p:=\frac{T_p\O^\omega}{T_p\O\cap T_p\O^\omega},
\end{equation}
where we denote the symplectic orthogonal of the tangent space $T_p\O$ to the orbit through $p$ as: 
\begin{equation}\label{symporthorb} T_p\O^\omega:=\{v\in T_pS\mid\omega(v,w)=0,\text{ }\forall w\in T_p\O\}.\end{equation} 
Further, consider the annihilator of $\g_p$ in $\g_x$:
\begin{equation}\label{annisoliealg} \g_p^0\subset \g_x^*.
\end{equation} 
\begin{prop}\label{normrepham} Let $(\G,\Omega)\rightrightarrows M$ be a symplectic groupoid and suppose that we are given left Hamiltonian $(\G,\Omega)$-action along $J:(S,\omega)\to M$. Let $\O$ be the orbit of the action through $p\in S$. 
\begin{itemize}\item[a)] The map $J$ restricts to surjective submersion $J_\O:\O\to \L$ from the orbit $\O$ onto a leaf $\L$ of $\G$. Moreover, the restriction $\omega_\O\in \Omega^2(\O)$ of $\omega$ coincides with the pull-back of $\omega_\L$:
\begin{equation}\label{pullbackorbitform} \omega_\O=(J_\O)^*\omega_\L.
\end{equation}
\item[b)] The symplectic normal space (\ref{sympnormsp}) to $\O$ at $p$ is a subrepresentation of the normal representation (\ref{normrepactgpoid}) of the action at $p$. In fact, (\ref{sympnormsp}), (\ref{normrepactgpoid}) and (\ref{annisoliealg}) fit into a canonical short exact sequence of $\G_p$-representations: 
\begin{equation}\label{ses1pois} 0\to \S\No_p\to \No_p\to \g_p^0\to 0. 
\end{equation} 
\item[c)] The normal representation (\ref{normreppt}) of $\G$ at $x:=J(p)$ fits into the canonical short exact sequence of $\G_p$-representations:
\begin{equation}\label{ses2pois} 0\to \g_p^0\to \g^*_x\to \g_p^*\to 0.
\end{equation}
\end{itemize}
\end{prop}
\begin{proof} That $J$ maps $\O$ submersively onto a leaf $\L$ follows from the axioms of a Lie groupoid action. The equality (\ref{pullbackorbitform}) is readily derived from (\ref{hammultcond}). Part $c$ is immediate from Proposition \ref{normrepsymp}$b$. To prove part $b$ and provide some further insight into part $c$, observe that $J$ induces a $\G_p$-equivariant map:
\begin{equation*} \underline{\d J}_p:\No_p\to \No_x. 
\end{equation*} Therefore we have two short exact sequences of $\G_p$-representations:
\begin{align} &0\to \ker(\underline{\d J}_p)\to \No_p\to \im(\underline{\d J}_p)\to 0 \label{ses1}\\
&0\to \im(\underline{\d J}_p)\to \No_x\to \coker(\underline{\d J}_p)\to 0 \label{ses2}
\end{align} 
Using the proposition below, the short exact sequence (\ref{ses2}) translates into the short exact sequence (\ref{ses2pois}), whereas (\ref{ses1}) translates into (\ref{ses1pois}). In particular, this proves part $b$. 
\end{proof}

\begin{prop}\label{infmomact} Let $(\G,\Omega)\rightrightarrows M$ be a symplectic groupoid and suppose that we are given a left Hamiltonian $(\G,\Omega)$-action along $J:(S,\omega)\to M$. Further, let $p\in S$. 
\begin{itemize}\item[a)] The symplectic orthogonal (\ref{symporthorb}) of the tangent space $T_p\O$ to the orbit $\O$ through $p$ coincides with $\ker{(\d J_p)}$.
\item[b)] The isotropy Lie algebra $\g_p$, viewed as subset of $T^*_xM$ via (\ref{imsymp}), is the annihilator of $\im(\d J_p)$ in $T_xM$, where $x=J(p)$.
\end{itemize}
\end{prop}
This is readily derived from the momentum map condition $(\ref{mommapcond})$.
\subsubsection{The symplectic normal representation} Notice that the symplectic form $\omega$ on $S$ descends to a linear symplectic form $\omega_p$ on the symplectic normal space (\ref{sympnormsp}). 
\begin{prop} $(\S\No_p,\omega_p)$ is a symplectic $\G_p$-representation. 
\end{prop}
\begin{proof} We ought to show that $\omega_p$ is $\G_p$-invariant. Note that, for any $v\in \ker(\d J_p)$ and $g\in \G_p$: \begin{equation*} g\cdot[v]=[\d m_{(g,p)}(0,v)].
\end{equation*} So, using Proposition \ref{infmomact}$a$ we find that for all $v,w\in T_p\O^\omega$ and $g\in \G_p$:
\begin{equation*} \omega_p(g\cdot[v],g\cdot[w])=(m^*\omega)_{(g,p)}((0,v),(0,w))=\omega_p([v],[w]),
\end{equation*} where in the last step we applied (\ref{hammultcond}).
\end{proof}
\begin{defi} Given a Hamiltonian action as above, we call \begin{equation}\label{sympnormrep} (\S\No_p,\omega_p)\in \textrm{SympRep}(\G_p)
\end{equation} its \textbf{symplectic normal representation at $p$}.
\end{defi}
Given any symplectic representation $(V,\omega_V)$ of a Lie group $H$, the $H$-action is Hamiltonian with quadratic momentum map:
\begin{equation}\label{quadsympmommap} J_V:(V,\omega_V)\to \h^*, \quad \langle J_V(v),\xi \rangle=\frac{1}{2}\omega_V(\xi\cdot v,v).
\end{equation} As we will now show, given a Hamiltonian $(\G,\Omega)$-action along $J:(S,\omega)\to M$, the quadratic momentum map:
\begin{equation}\label{quadsympmommap2} J_{\S\No_p}:(\S\No_p,\omega_p)\to \g_p^*
\end{equation} of the symplectic normal representation at $p$ can be expressed in terms of the quadratic differential of $J$ at $p$. Recall from \cite{AGV1} that the \textbf{quadratic differential} of a map $F:S\to M$ at $p\in S$ is defined to be the quadratic map:
\begin{equation*} \d^2F_p:\ker(\d F_p)\to \coker(\d F_p),\quad \d^2F_p(v)=\left[\frac{1}{2}\left.\frac{\d^2}{\d^2 t}\right|_{t=0}(\psi\circ F\circ \phi^{-1})(tv)\right],
\end{equation*} where $\phi:(U,p)\to (T_pS,0)$ and $\psi:(V,x)\to (T_{x}M,0)$ are any two open embeddings, defined on open neighbourhoods of $p$ and $x:=F(p)$ such that $F(U)\subset V$, with the property that their differentials at $p$ and $x$ are the respective identity maps. Returning to the momentum map $J$, by Proposition \ref{infmomact} its quadratic differential becomes a map:
\begin{equation}\label{quadsympmommap3} \d^2J_p:T_p\O^\omega\to \g_p^*.
\end{equation} 
\begin{prop}\label{quaddifmommap} Let $J:(S,\omega)\to M$ be the momentum map of a Hamiltonian action and $p\in S$. Then the quadratic differential (\ref{quadsympmommap3}) is the composition of the quadratic momentum map (\ref{quadsympmommap2}) with the canonical projection $T_p\O^\omega\to \S\No_p$:
\begin{center} 
\begin{tikzcd} T_p\O^\omega\arrow[d] \arrow[r,"\d^2J_p"] & \g_p^* \\
\S\No_p\arrow[ru,"J_{\S\No_p}"'] & 
\end{tikzcd}
\end{center} 
\end{prop}
For the proof, we use an alternative description of the quadratic differential. Recall that, given a vector bundle $E\to S$ and a germ of sections $e\in \Gamma_p(E)$ vanishing at $p\in S$, the linearization of $e$ at $p$ is the linear map:
\begin{equation*} e^\textrm{lin}_p:T_pS \to E_p,\quad e^\textrm{lin}_p:=\textrm{pr}_{E_p}\circ (\d e)_p,
\end{equation*} where we view the differential $(\d e)_p$ of the map $e$ at $p$ as map into $E_p\oplus T_pS$, via the canonical identification of $(TE)_{(p,0)}$ with $E_p\oplus T_pS$. With this, one can define the \textbf{intrinsic Hessian} of $F$ at $p$ to be the symmetric bilinear map:
\begin{equation*} \textrm{Hess}_p(F):\ker(\d F_p)\times \ker(\d F_p)\to \coker(\d F_p), \quad (X_p,Y_p)\mapsto\left[\frac{1}{2} (\d F(Y))^\textrm{lin}_p(X_p)\right],
\end{equation*} where $Y\in \X_p(S)$ is any germ of vector fields extending $Y_p$ and we see $\d F(Y)$ as a germ of sections of $F^*(TM)$. 
The quadratic differential is now given by the quadratic form:
\begin{equation*} \d^2F_p(v)=\textrm{Hess}_p(F)(v,v), \quad v\in \ker(\d F_p).
\end{equation*} 
We will further use the following immediate, but useful, observation.
\begin{lemma}\label{linseclem1} Let $\Phi:E\to F$ be a map of vector bundles over the same manifold, covering the identity map. If $e\in \Gamma_p(E)$ is a germ of sections that vanishes at $p$, then so does $\Phi(e)\in \Gamma_p(F)$ and we have:
\begin{equation*} \Phi(e)^\textrm{lin}_p=\Phi\circ e^\textrm{lin}_p.
\end{equation*}
\end{lemma} 
\begin{proof}[Proof of Proposition \ref{quaddifmommap}] Let $\alpha_x\in \g_p \subset T_x^*M$ and $X_p\in \ker(\d J_p)=T_p\O^\omega$. We have to prove: \begin{equation*} \langle J_{\S\No_p}([X_p]), \alpha_x\rangle=\langle \alpha_x, \d^2J_p(X_p)\rangle. 
\end{equation*} This will follow by linearizing both sides of equation (\ref{mommapcond}). Let $\alpha\in \Omega^1(M)$ and $X\in \X(S)$ be extensions of $\alpha_x$ and $X_p$, respectively. On one hand, we have:
\begin{align*} \langle (\iota_{a(\alpha)}\omega)^\textrm{lin}_p(X_p),X_p\rangle &=\omega_p(a(\alpha)^\textrm{lin}_p(X_p),X_p)\\
&=2\langle J_{\S\No_p}([X_p]), \alpha_x\rangle.
\end{align*} Here we have first used that, given a $k$-form $\beta$ and a vector field $Y$ that vanishes at $p$, it holds that \begin{equation*} (\iota_Y\beta)^\textrm{lin}_p(X_p)=\iota_{Y^\textrm{lin}_p(X_p)}\beta_p,\end{equation*} as follows from Lemma \ref{linseclem1}. Furthermore, for the second step we have used that the Lie algebra representation $\g_p\to \mathfrak{sp}(\S\No_p,\omega_p)$ induced by the symplectic normal representation is given by:
\begin{equation*} \alpha_x\cdot[X_p]=[a(\alpha)^{\textrm{lin}}_p(X_p)].
\end{equation*}  
On the other hand, linearizing the right-hand side of (\ref{mommapcond}) we find (as desired):
\begin{align*} \langle (J^*\alpha)^\textrm{lin}_p(X_p),X_p\rangle&=(\alpha(\d J(X))^\textrm{lin}_p(X_p)\\
&=2\langle \alpha_x, \d^2J_p(X_p)\rangle.
\end{align*} Here we have first used that, given a vector field $Y$ and a $k$-form $\beta$ that vanishes at $p$, it holds that
\begin{equation*} (\iota_Y\beta)^\textrm{lin}_p(X_p)=\iota_{Y_p}(\beta^\textrm{lin}_p(X_p)),\end{equation*}
as follows from Lemma \ref{linseclem1}. Furthermore, for the second step we have again used Lemma \ref{linseclem1}.
\end{proof} 
\subsubsection{Neighbourhood equivalence and rigidity}\label{normformthmsec}
We now turn to the notion of neighbourhood equivalence, used in the statement of Theorem \ref{normhamthm}. In view of Proposition \ref{normrepsymp}, the restriction of a symplectic groupoid $(\G,\Omega)$ to a leaf $\L$ gives rise to the data of:
\begin{itemize}\item a symplectic manifold $(\L,\omega_\L)$
\item a transitive Lie groupoid $\G_\L\rightrightarrows \L$ equipped with a closed multiplicative $2$-form $\Omega_\L$, 
\end{itemize} subject to the relation:
\begin{equation}\label{data0} \Omega_\L=t_{\G_\L}^*\omega_\L-s_{\G_\L}^*\omega_\L.
\end{equation}
\begin{defi}\label{data0defi} We call a collection of such data a \textbf{zeroth-order symplectic groupoid data}. 
\end{defi}
Further, using Proposition \ref{normrepham}$a$, we observe that the restriction of a Hamiltonian $(\G,\Omega)$-action along $J:(S,\omega)\to M$ to an orbit $\O$ (with corresponding leaf $\L=J(\O)$) encodes the data of:
\begin{itemize}\item a zeroth-order symplectic groupoid data $(\G_\L,\Omega_\L)\rightrightarrows (\L,\omega_\L)$, 
\item a pre-symplectic manifold $(\O,\omega_\O)$,
\item a transitive Lie groupoid action of $\G_\L$ along a map $J_\O:\O\to \L$, 
\end{itemize} subject to the relations:
\begin{equation}\label{data1} (\textrm{pr}_{\G_\L})^*\Omega_\L=(m_\O)^*\omega_\O-(\textrm{pr}_\O)^*\omega_\O \quad \&\quad \omega_\O=(J_\O)^*\omega_\L.
\end{equation} where we denote by: \begin{equation*} m_\O,\textrm{pr}_\O:\G_\L\ltimes \O\to \O,
\quad \textrm{pr}_{\G_\L}:\G_\L\ltimes \O\to \G_\L,
\end{equation*} the map defining the action and the projections from the action groupoid to $\O$ and $\G_\L$.
\begin{defi}\label{data1defi} We call a collection of such data a \textbf{zeroth-order Hamiltonian data}.
\end{defi} 
Next, we define realizations of such zeroth-order data and neighbourhood equivalences thereof. 
\begin{defi}\label{nhoodeq0defi} By a \textbf{realization} of a given \textbf{zeroth-order symplectic groupoid data}:
\begin{center} 
\begin{tikzpicture} \node (G1) at (0,0) {$(\G_\L,\Omega_\L)$};
\node (M1) at (0,-1.3) {$(\L,\omega_\L)$};

\node (G) at (2.7,0) {$(\G,\Omega)$};
\node (M) at (2.7,-1.3) {$(M,\pi)$};
 
\draw[->,transform canvas={xshift=-\shift}](G1) to node[midway,left] {}(M1);
\draw[->,transform canvas={xshift=\shift}](G1) to node[midway,right] {}(M1);

\draw[right hook->] (0.8,-0.65) -- (2,-0.65) node[pos=0.4,above] {$\text{ }\text{ }i$};

\draw[->,transform canvas={xshift=-\shift}](G) to node[midway,left] {}(M);
\draw[->,transform canvas={xshift=\shift}](G) to node[midway,right] {}(M);

\end{tikzpicture}
\end{center} we mean an embedding of Lie groupoids $i:\G_\L\hookrightarrow \G$ with the property that $\Omega$ pulls back to $\Omega_\L$ and that $\G_\L$ embeds as the restriction of $\G$ to a leaf. Of course, $(\L,\omega_\L)$ then automatically embeds as a symplectic leaf of $(\G,\Omega)$. We call two realizations $i_1$ and $i_2$ of the same zeroth-order symplectic groupoid data \textbf{neighbourhood-equivalent} if there are opens $V_1$ and $V_2$ around $\L$ in $M_1$ and $M_2$ respectively, together with an isomorphism of symplectic groupoids: 
\begin{center} 
\begin{tikzpicture} \node (G1) at (0,0) {$(\G_1,\Omega_1)\vert_{V_1}$};
\node (M1) at (0,-1.3) {$(V_1,\pi_1)$};

\node (G) at (2.7,0) {$(\G_2,\Omega_2)\vert_{V_2}$};
\node (M) at (2.7,-1.3) {$(V_2,\pi_2)$};
 
\draw[->,transform canvas={xshift=-\shift}](G1) to node[midway,left] {}(M1);
\draw[->,transform canvas={xshift=\shift}](G1) to node[midway,right] {}(M1);

\draw[transparent] (0.2,-0.65) -- (1.4,-0.65) node[opacity=1] {\resizebox{0.8cm}{0.2cm}{$\cong$}};

\draw[->,transform canvas={xshift=-\shift}](G) to node[midway,left] {}(M);
\draw[->,transform canvas={xshift=\shift}](G) to node[midway,right] {}(M);

\end{tikzpicture}
\end{center}
that intertwines $i_1$ with $i_2$. 
\end{defi}
\begin{defi}\label{nhoodeq1defi} By a \textbf{realization} of a given \textbf{zeroth order Hamiltonian data}: 
\begin{center}
\begin{tikzpicture} \node (G1) at (0,0) {$(\G_\L,\Omega_\L)$};
\node (M1) at (0,-1.3) {$(\L,\omega_\L)$};
\node (O) at (2,0) {$(\O,\omega_\O)$};

\node (G) at (5,0) {$(\G,\Omega)$};
\node (M) at (5,-1.3) {$(M,\pi)$};
\node (S) at (6.8,0) {$(S,\omega)$};
 
\draw[->,transform canvas={xshift=-\shift}](G1) to node[midway,left] {}(M1);
\draw[->,transform canvas={xshift=\shift}](G1) to node[midway,right] {}(M1);
\draw[->](O) to node[pos=0.25, below] {$\text{ }\text{ }J_\O$} (M1);
\draw[->] (1.2,-0.15) arc (315:30:0.25cm);

\draw[right hook->] (2.7,-0.65) -- (4.4,-0.65) node[pos=0.4,above] {$\text{ }\text{ }(i,j)$};

\draw[->,transform canvas={xshift=-\shift}](G) to node[midway,left] {}(M);
\draw[->,transform canvas={xshift=\shift}](G) to node[midway,right] {}(M);
\draw[->](S) to node[pos=0.25, below] {$\text{ }\text{ }J$} (M);
\draw[->] (6.1,-0.15) arc (315:30:0.25cm);

\end{tikzpicture}
\end{center}
we mean a pair $(i,j)$ consisting of:
\begin{itemize}\item a realization $i$ of the zeroth-order symplectic groupoid data $(\G_\L,\Omega_\L)\rightrightarrows (\L,\omega_\L)$,
\item an embedding $j:\O\hookrightarrow S$ that pulls back $\omega$ to $\omega_\O$ and is compatible with $i$, in the sense that $i$ and $j$ together interwine $J_\O$ with $J$, and  the actions along these maps. 
\end{itemize}
We call two realizations $(i_1,j_1)$ and $(i_2,j_2)$ of the same zeroth-order Hamiltonian data \textbf{neighbourhood-equivalent} if there are opens $V_1$ and $V_2$ around $\L$ in $M_1$ and $M_2$ respectively, a $\G_1\vert_{V_1}$-invariant open $U_1$ and a $\G_2\vert_{V_2}$-invariant open $U_2$ around $\O$ in $J_1^{-1}(V_1)$, respectively $J_2^{-1}(V_2)$, together with:
\begin{itemize} 
\item an isomorphism $(\G_1,\Omega_1)\vert_{V_1}\cong (\G_2,\Omega_2)\vert_{V_2}$ that intertwines $i_1$ with $i_2$,
\item a symplectomorphism $(U_1,\omega_1)\cong (U_2,\omega_2)$ that intertwines $j_1$ with $j_2$ and is compatible with the above isomorphism of symplectic groupoids, in the sense that together these intertwine $J_1:U_1\to V_1$ with $J_2:U_2\to V_2$, and the actions along these maps. 
\end{itemize}
In other words, we have an isomorphism of Hamiltonian actions:
\begin{center}
\begin{tikzpicture} \node (G1) at (0,0) {$(\G_1,\Omega_1)\vert_{V_1}$};
\node (M1) at (0,-1.3) {$(V_1,\pi_1)$};
\node (O) at (2,0) {$(U_1,\omega_1)$};

\node[transparent] (A) at (2.2,-0.65) {$A$};
\node[transparent] (B) at (3.4,-0.65) {$B$};

\node (G) at (4,0) {$(\G_2,\Omega_2)\vert_{V_2}$};
\node (M) at (4,-1.3) {$(V_2,\pi_2)$};
\node (S) at (6,0) {$(U_2,\omega_2)$};
 
\draw[->,transform canvas={xshift=-\shift}](G1) to node[midway,left] {}(M1);
\draw[->,transform canvas={xshift=\shift}](G1) to node[midway,right] {}(M1);
\draw[->](O) to node[pos=0.25, below] {$\text{ }\text{ }J_1$} (M1);
\draw[->] (1.3,-0.15) arc (315:30:0.25cm);

\draw[transparent] (A) edge node[opacity=1] {\resizebox{0.8cm}{0.2cm}{$\cong$}} (B);

\draw[->,transform canvas={xshift=-\shift}](G) to node[midway,left] {}(M);
\draw[->,transform canvas={xshift=\shift}](G) to node[midway,right] {}(M);
\draw[->](S) to node[pos=0.25, below] {$\text{ }\text{ }J_2$} (M);
\draw[->] (5.3,-0.15) arc (315:30:0.25cm);

\end{tikzpicture}
\end{center} that intertwines the embeddings of zeroth-order data. Usually the embeddings are clear from the context and we simply call the two Hamiltonian actions neighbourhood-equivalent around $\O$. 
\end{defi} We can now state the rigidity result mentioned in the introduction to this section. 
\begin{thm}\label{righamthm} Suppose that we are given two realizations of the same zeroth-order Hamiltonian data with orbit $\O$ and leaf $\L$. Fix $p\in \O$ and let $x=J_\O(p)\in \L$. If both symplectic groupoids are proper at $x$ (in the sense of Definition \ref{propatxdefi}), then the realizations are neighbourhood-equivalent if and only if their symplectic normal representations at $p$ are isomorphic as symplectic $\G_p$-representations.
\end{thm}

In the coming section, we give an explicit construction to show:
\begin{prop}\label{existlocmodprop} For any zeroth-order Hamiltonian data with orbit $\O$, any choice of $p\in \O$ and any symplectic representation $(V,\omega_V)$ of the isotropy group $\G_p$, there is a realization of the zeroth-order data that has $(V,\omega_V)$ as symplectic normal representation at $p$.
\end{prop} 
Given a Hamiltonian action, we call the realization constructed from the zeroth-order Hamiltonian data obtained by restriction to $\O$ and from the symplectic normal representation at $p$: \textbf{the local model} of the Hamiltonian action around $\O$ (we disregard the choice of $p\in \O$, as different choices result in isomorphic local models). Applying Theorem \ref{righamthm} to the given Hamiltonian action on one hand and, on the other hand, to its local model around $\O$, Theorem \ref{normhamthm} follows. Hence, after the construction of this local model, it remains for us to prove Theorem \ref{righamthm}.

\subsection{The local model}\label{locmodsec}
\subsubsection{Reorganization of the zeroth-order Hamiltonian data}\label{reorghamdat}
Before constructing the local model, we rearrange the zeroth-order data (defined in the previous subsection) into a simpler form. First, due to the relations (\ref{data0}) and (\ref{data1}), the triple of $2$-forms $\Omega_\L$, $\omega_\L$ and $\omega_\O$ can be fully reconstructed from the single $2$-form $\omega_\L$. Therefore, a collection of zeroth-order Hamiltonian data can equivalently be defined as the data of:
\begin{itemize}\item a symplectic manifold $(\L,\omega_\L)$,
\item a transitive Lie groupoid $\G_\L\rightrightarrows \L$, 
\item a transitive Lie groupoid action of $\G_\L$ along a map $J_\O:\O\to \L$. 
\end{itemize} After the choice of a point $p\in \O$, this can be simplified further to a collection consisting of:
\begin{itemize}
\item a symplectic manifold $(\L,\omega_\L)$,
\item a Lie group $G$ (corresponding to $\G_x$),
\item a (right) principal $G$-bundle $P\to \L$ (corresponding to $t:s^{-1}(x)\to \L$),
\item a closed subgroup $H$ of $G$ (corresponding to $\G_p$).
\end{itemize}  
To see this, fix a point $p\in \O$ and let $x=J_\O(p)\in \L$. Since $\G_\L$ is transitive, the choice of $x\in \L$ induces an isomorphism between $\G_\L$ and the gauge-groupoid: 
\begin{equation}\label{gaugegpoidprinbunlocmod} s^{-1}(x)\times_{\G_x}s^{-1}(x) \rightrightarrows \L,
\end{equation} of the principal $\G_x$-bundle $t:s^{-1}(x)\to \L$. In particular, $\G_\L$ is entirely encoded by this principal bundle. Furthermore, due to transitivity the $\G_\L$-action along $J_\O$ is entirely determined by this principal bundle and the subgroup $\G_p$ of $\G_x$. Indeed, the map $J_\O$ can be recovered from this, for we have a commutative square:
\begin{center}
\begin{tikzcd}
 s^{-1}(x)/{\G_p} \arrow[r]\arrow{d}[rotate=90, xshift=-0.8ex, yshift=0.7ex]{\sim} & s^{-1}(x)/{\G_x}\arrow{d}[rotate=90, xshift=-0.8ex, yshift=0.7ex]{\sim} \\
 \O \arrow[r,"J_\O"]  & \L 
 \end{tikzcd} 
 \end{center} where the left vertical map is defined by acting on $p$ and the upper horizontal map is the canonical one. Moreover, the action can be recovered as the action of the groupoid (\ref{gaugegpoidprinbunlocmod}) along the upper horizontal map, given by $[p,q]\cdot [q]=[p]$.  
 \subsubsection{Construction of the local model for the symplectic groupoid}\label{locmodconstsec}
The construction presented here is well-known. For other (more Poisson geometric) constructions of this local model, see \cite{CrMar1,Mar1}. The local model for the symplectic groupoid is built out of the zeroth-order symplectic groupoid data, encoded as above by:
\begin{itemize}
\item a symplectic manifold $(\L,\omega_\L)$,
\item a Lie group $G$,
\item a (right) principal $G$-bundle $P\to \L$.
\end{itemize} 
To construct the local model, we make an auxiliary choice of a connection $1$-form $\theta\in \Omega^1(P;\g)$ and define:\begin{equation}\label{hattheta} \hat{\theta}\in \Omega^1(P\times \g^*), \quad \hat{\theta}_{(q,\alpha)}=\langle \alpha, \theta_q\rangle. \end{equation} Then, we use the symplectic structure $\omega_\L$ on $\L$ to define:
\begin{equation}\label{omegathetalocmod} \omega_\theta=(\textrm{pr}_\L)^*\omega_\L-\d\hat{\theta}\in \Omega^2(P\times \g^*),
\end{equation} where by $\textrm{pr}_\L$ we denote the composition $P\times \g^*\xrightarrow{\textrm{pr}_1} P\to \L$. The $2$-form $\omega_\theta$ is closed, non-degenerate at all points of $P\times \{0\}$ and $(P\times \g^*,\omega_\theta)\to \g^*$ is a (right) pre-symplectic Hamiltonian $G$-space. Therefore, the open subset $\Sigma_\theta\subset P\times \g^*$ on which $\omega_\theta$ is non-degenerate is a $G$-invariant neigbourhood of $P\times \{0\}$. Since the action is free and proper, the symplectic form $\omega_\theta$ descends to a Poisson structure $\pi_\theta$ on the open neighbourhood $M_\theta$ of the zero-section $\L$, defined as: 
\begin{equation*} M_\theta:=\Sigma_\theta/G\subset P\times_G\g^*.
\end{equation*} This is the base of the local model. For the construction of the integrating symplectic groupoid, notice first that the pair groupoid: \begin{equation}\label{symppairlocmod} \left(\Sigma_\theta\times \Sigma_\theta,\omega_\theta\oplus -\omega_\theta\right)
\end{equation} is a symplectic groupoid and, furthermore, it is a (right) free and proper Hamiltonian $G$-space (being a product of two). Therefore, the symplectic form $\omega_\theta\oplus-\omega_\theta$ descends to the symplectic reduced space at $0\in \g^*$: \begin{equation}\label{locmodsymgp} (\G_\theta,\Omega_\theta):=\left( (\Sigma_\theta\times\Sigma_\theta)\sslash G,\Omega_\textrm{red}\right).
\end{equation} The pair groupoid structure on $\Sigma_\theta\times \Sigma_\theta$ descends to a Lie groupoid structure on (\ref{locmodsymgp}), making it a symplectic groupoid integrating $(M_\theta,\pi_\theta)$. This is the symplectic groupoid in the local model. It is canonically a realization of the given zeroth-order symplectic groupoid data: the gauge-groupoid of the principal $G$-bundle $P\to \L$ (corresponding to (\ref{gaugegpoidprinbunlocmod})) embeds into $(\ref{locmodsymgp})$ via the zero-section.

\subsubsection{Construction of the local model for Hamiltonian actions}\label{hamlocmodconsec} The construction below generalizes that in \cite{Ma1,GS4}. The local model is built out of a zeroth-order Hamiltonian data and a symplectic representation of an isotropy group of the action, encoded as in Subsection \ref{reorghamdat} by:
\begin{itemize}
\item a symplectic manifold $(\L,\omega_\L)$,
\item a Lie group $G$,
\item a (right) principal $G$-bundle $P\to \L$,
\item a closed subgroup $H$ of $G$,
\item a symplectic $H$-representation $(V,\omega_V)$.
\end{itemize} Choose an auxiliary connection $1$-form $\theta\in \Omega^1(P;\g)$ and define $\omega_\theta$, $\Sigma_\theta$ and $M_\theta$ as in the construction of the local model for symplectic groupoids. To construct a Hamiltonian action of the symplectic groupoid (\ref{locmodsymgp}), consider the product of the Hamiltonian $H$-spaces: 
\begin{equation*} \textrm{pr}_{\h^*}:(\Sigma_\theta,\omega_\theta)\xrightarrow{\textrm{pr}_{\g^*}}\g^*\to \h^*\quad \& \quad J_V:(V,\omega_V)\to \h^*,
\end{equation*} 
where $J_V$ is as in (\ref{quadsympmommap}). This is another (right) Hamiltonian $H$-space:
\begin{equation*} J_H:(\Sigma_\theta\times V, {\omega}_\theta\oplus\omega_V)\to \h^*,\quad (q,\alpha,v)\mapsto \alpha\vert_\h-J_V(v),
\end{equation*} 
where the action is the diagonal one, which is free and proper. The symplectic manifold in the local model is the reduced space at $0\in \h^*$:
\begin{equation}\label{locmodham1} (S_\theta,\omega_{S_\theta}):=\left((\Sigma_\theta\times V)\sslash H,\omega_{\textrm{red}}\right).
\end{equation}
To equip this with a Hamiltonian action of (\ref{locmodsymgp}), observe that, on the other hand, the symplectic pair groupoid (\ref{symppairlocmod}) acts along: \begin{equation*} \text{pr}_{\Sigma_\theta}:(\Sigma_\theta\times V, {\omega}_\theta\oplus\omega_V)\to \Sigma_\theta\end{equation*} in a Hamiltonian fashion as: $(\sigma,\tau)\cdot(\tau,v)=(\sigma,v)$ for $\sigma,\tau\in \Sigma_\theta$ and $v\in V$. This descends to a Hamiltonian action of (\ref{locmodsymgp}) that fits into a diagram of commuting Hamiltonian actions: 
\begin{center}
\begin{tikzpicture} \node (G1) at (0,0) {$\left(\G_\theta,\Omega_\theta\right)$};
\node (M1) at (0,-1.3) {$M_\theta$};
\node (S) at (3.3,0) {$(\Sigma_\theta\times V, {\omega}_\theta\oplus\omega_V)$};
\node (M2) at (6.6,-1.3) {$\h^*$};
\node (G2) at (6.6,0) {$(T^*H,-\d \lambda_{\textrm{can}})$};
 
\draw[->,transform canvas={xshift=-\shift}](G1) to node[midway,left] {}(M1);
\draw[->,transform canvas={xshift=\shift}](G1) to node[midway,right] {}(M1);
\draw[->,transform canvas={xshift=-\shift}](G2) to node[midway,left] {}(M2);
\draw[->,transform canvas={xshift=\shift}](G2) to node[midway,right] {}(M2);
\draw[->](S) to node[pos=0.25, below] {$\text{ }\text{ }\text{ }\text{ }\textrm{pr}_{M_\theta}$} (M1);
\draw[->] (1.75,-0.15) arc (315:30:0.25cm);
\draw[<-] (4.8,0.15) arc (145:-145:0.25cm);
\draw[->](S) to node[pos=0.25, below] {$J_H$\text{}} (M2);
\end{tikzpicture}
\end{center} with the property that the momentum map of each one is invariant under the action of the other. It therefore follows that the left-hand action descends to a Hamiltonian action along the map:
\begin{equation}\label{locmodham2} J_\theta:\left(S_\theta,\omega_{S_\theta}\right)\to M_\theta, \quad [\sigma,v]\mapsto [\sigma].
\end{equation} This is the Hamiltonian action in the local model. It is canonically a realization of the given zeroth-order Hamiltonian data: as in the previous subsection the gauge-groupoid of the principal $G$-bundle $P\to \L$ embeds into $(\ref{locmodsymgp})$ via the zero-section and similarly $P/H$ embeds into (\ref{locmodham1}). This completes the construction of the local model. Finally, given the starting data in Proposition \ref{existlocmodprop}, one readily verifies that the symplectic normal representation at $p$ of the resulting Hamiltonian action of (\ref{locmodsymgp}) along (\ref{locmodham2}) is isomorphic to $(V,\omega_V)$ as symplectic $\G_p$-representation. So, this also completes the proof of Proposition \ref{existlocmodprop}.
\begin{rem}\label{splitremlocmod} Under the assumption that the short exact sequence:
\begin{equation}\label{ses2poisabs} 0\to \h^0\to \g^*\to \h^*\to 0
\end{equation} splits $H$-equivariantly (which holds if $H$ is compact), the local model can be put in the more familiar form of a vector bundle over $\O$. Indeed,  
let $\p:\h^*\to \g^*$ be such a splitting. Then we have an open embedding:
\begin{equation}\label{locmodvecbun} S_\theta\to P\times_H(\h^0\oplus V), \quad [p,\alpha,v]\mapsto [p,\alpha-\p(J_V(v)),v],
\end{equation} onto an open neighbourhood of the zero-section, which identifies the momentum map (\ref{locmodham2}) with the restriction to this open neighbourhood of the map:
\begin{equation}\label{mommaplocmod2} P\times_H(\h^0\oplus V)\to P\times_G\g^*,\quad [p,\alpha,v]\mapsto [p,\alpha+\p(J_V(v))]. 
\end{equation} 
To identify the action accordingly, observe that, as Lie groupoid, $(\ref{locmodsymgp})$ embeds canonically onto an open subgroupoid of:
\begin{equation}\label{simpmodgp} (P\times P)\times_{G}\g^*\rightrightarrows P\times_G\g^*,
\end{equation} which inherits its Lie groupoid structure from the submersion groupoid of $\textrm{pr}_{\g^*}:P\times \g^*\to \g^*$, being a quotient of it. This identifies the action of (\ref{locmodsymgp}) along (\ref{mommaplocmod2}) with (a restriction of) the action of (\ref{simpmodgp}) along (\ref{mommaplocmod2}), given by:
\begin{equation*} [p_1,p_2,\alpha+\p(J_V(v))]\cdot [p_2,\alpha,v]=[p_1,\alpha,v], \quad p_1,p_2\in P,\quad \alpha\in \h^0,\quad v\in V.
\end{equation*}
\end{rem}

\subsubsection{Relation to the Marle-Guillemin-Sternberg model}\label{MGSrelsec} Let $G$ be a Lie group and consider a Hamiltonian $G$-space $J:(S,\omega)\to \g^*$. As remarked in Example \ref{exhamGsp}, this is the same as a Hamiltonian action of the cotangent groupoid $(G\ltimes \g^*,-\d \lambda_{\textrm{can}})\rightrightarrows \g^*$ along $J$. Let $p\in S$, $\alpha=J(p)$ and suppose that $G\ltimes \g^*$ is proper at $\alpha$ (in the sense of Definition \ref{propatxdefi}). In this case, our local model around the orbit $\O$ through $p$ is equivalent to the local model in the Marle-Guillemin-Sternberg (MGS) normal form theorem for Hamiltonian $G$-spaces (recalled below). To see this, first note that, since the isotropy group $G_\alpha$ is compact, the short exact sequence of $G_\alpha$-representations:
\begin{equation}\label{splitseqcotangentgpoid} 0\to \g_\alpha^0\to \g^*\to \g_\alpha^*\to 0
\end{equation} is split. Let $\sigma:\g_\alpha^*\to \g^*$ be a $G_\alpha$-equivariant splitting of (\ref{splitseqcotangentgpoid}) and consider the connection one-form $\theta\in \Omega^1(G;\g_\alpha)$ on $G$ (viewed as right principal $G_\alpha$-bundle) obtained by composing the left-invariant Maurer-Cartan form on $G$ with $\sigma^*:\g\to \g_\alpha$. The leaf $\L$ through $\alpha$ is a coadjoint orbit and $\omega_{\L}$ is the KKS-symplectic form, which is invariant under the coadjoint action. Therefore, the $2$-form $\omega_\theta\in \Omega^2(G\times \g^*_\alpha)$, defined as in (\ref{omegathetalocmod}), is not only invariant under the right diagonal action of $G_\alpha$, but it also invariant under the left action of $G$ by left translation on the first factor. This implies that the open $\Sigma_\theta$ on which $\omega_\theta$ is non-degenerate is of the form $G\times W$ for a $G_\alpha$-invariant open $W$ around the origin in $\g_\alpha^*$. The local model for the cotangent groupoid around $\L$ becomes: \begin{equation*} 
(G\ltimes (G\times_{G_\alpha} W), \Omega_\theta)\rightrightarrows G\times_{G_\alpha}W,
\end{equation*} the groupoid associated to the action of $G$ by left translation on the first factor. To compare this to the cotangent groupoid itself, consider the $G$-equivariant map:
\begin{equation*}\label{lincoadform}  \phi:G\times_{G_\alpha}W\to \g^*, \quad [g,\beta]\mapsto g\cdot(\alpha+\sigma(\beta)).
\end{equation*} Since $G\ltimes \g^*$ is proper at $\alpha$, we can shrink $W$ so that $\phi$ becomes an embedding onto a $G$-invariant open neighbourhood of $\L$. Then $\phi$ lifts canonically to an isomorphism of symplectic groupoids:
\begin{equation}\label{isolocmodmgsmod} (G\ltimes (G\times_{G_\alpha}W), \Omega_\theta)\xrightarrow{\sim} (G\ltimes \g^*,-\d\lambda_{\textrm{can}})\vert_{\phi(G\times_{G_\alpha}W)},
\end{equation} and this is a neighbourhood equivalence around $G\ltimes \L$ (with respect to the canonical embeddings). Our local model for $(S,\omega)$ around $\O$ is the same as that in the MGS local model, and via (\ref{isolocmodmgsmod}) the Hamiltonian action in our local model is identified with the Hamiltonian $G$-space in the MGS local model. In particular the momentum map (\ref{mommaplocmod2}) is identified with:
\begin{equation*} J_{\textrm{MGS}}:G\times_{G_p}(\g_p^0\oplus \S\No_p)\to \g^*, \quad [g,\beta,v]\mapsto g\cdot\left(\alpha+\sigma\left(\beta+\p\left(J_{\S\No_p}\left(v\right)\right)\right)\right).
\end{equation*} 
\begin{rem}\label{sharphamliegpactrem} As will be clear from the proof of Theorem \ref{righamthm}, the conclusion of Theorem \ref{normhamthm} can be sharpened for Hamiltonian Lie group actions: if we start with a Hamiltonian $G$-space, then under the assumptions of Theorem \ref{normhamthm} we can in fact find a neighbourhood equivalence in which the isomorphism of symplectic groupoids is the explicit isomorphism (\ref{isolocmodmgsmod}). In particular, this neighbourhood equivalence is defined on $G$-invariant neighbourhoods of $\O$ in $S$ and $\L$ in $\g^*$.
\end{rem}

\subsection{The proof}\label{normformpfsubsec}
\subsubsection{Morita equivalence of groupoids}
To prove Theorem \ref{righamthm} (and hence Theorem \ref{normhamthm}), we will reduce to the case where $\O\subset J_X^{-1}(0)$ is an orbit of a Hamiltonian $G$-space $J_X:(X,\omega_X)\to \g^*$ (with $G$ a compact Lie group), to which we can apply the Marle-Guilleming-Sternberg theorem. The idea of such a reduction is by no means new \textemdash in fact, it appears already in the work of Guillemin and Sternberg. To do so, we use the fact that Morita equivalent symplectic groupoids have equivalent categories of modules. In preparation for this, we will now first recall the definition, some useful properties and examples of Morita equivalence. 
\begin{defi} Let $\G_1\rightrightarrows M_1$ and $\G_2\rightrightarrows M_2$ be Lie groupoids. A \textbf{Morita equivalence} from $\G_1$ to $\G_2$ is a principal $(\G_1,\G_2)$-bi-bundle $(P,\alpha_1,\alpha_2)$. This consists of:
\begin{itemize} \item A manifold $P$ with two surjective submersions $\alpha_i:P\to M_i$.
\item A left action of $\G_1$ along $\alpha_1$ that makes $\alpha_2$ into a principal $\G_1$-bundle.
\item A right action of $\G_2$ along $\alpha_2$ that makes $\alpha_1$ into a principal $\G_2$-bundle.
\end{itemize} Furthermore, the two actions are required to commute. We depict this as:
\begin{center}
\begin{tikzpicture} \node (G1) at (0,0) {$\G_1$};
\node (M1) at (0,-1.3) {$M_1$};
\node (S) at (1.4,0) {$P$};
\node (M2) at (2.7,-1.3) {$M_2$};
\node (G2) at (2.7,0) {$\G_2$};
 
\draw[->,transform canvas={xshift=-\shift}](G1) to node[midway,left] {}(M1);
\draw[->,transform canvas={xshift=\shift}](G1) to node[midway,right] {}(M1);
\draw[->,transform canvas={xshift=-\shift}](G2) to node[midway,left] {}(M2);
\draw[->,transform canvas={xshift=\shift}](G2) to node[midway,right] {}(M2);
\draw[->](S) to node[pos=0.25, below] {$\text{ }\text{ }\alpha_1$} (M1);
\draw[->] (0.8,-0.15) arc (315:30:0.25cm);
\draw[<-] (1.9,0.15) arc (145:-145:0.25cm);
\draw[->](S) to node[pos=0.25, below] {$\alpha_2$\text{ }} (M2);
\end{tikzpicture}
\end{center}
For every leaf $\L_1\subset M_1$, there is a unique leaf $\L_2\subset M_2$ such that $\alpha_1^{-1}(\L_1)=\alpha_2^{-1}(\L_2)$; such leaves $\L_1$ and $\L_2$ are called \textbf{$P$-related}. When $(\G_1,\Omega_1)$ and $(\G_2,\Omega_2)$ are symplectic groupoids, then a \textbf{symplectic Morita equivalence} from $(\G_1,\Omega_1)$ to $(\G_2,\Omega_2)$ is a Morita equivalence with the extra requirement that $(P,\omega_P)$ is a symplectic manifold and both actions are Hamiltonian. 
\end{defi}
Morita equivalence is an equivalence relation that, heuristically speaking, captures the geometry transverse to the leaves. The simplest motivation for this principle is the following basic result.
\begin{prop}\label{transgeomgpoid} Let $(P,\alpha_1,\alpha_2)$ be a Morita equivalence from $\G_1\rightrightarrows M_1$ to $\G_2\rightrightarrows M_2$. 
\begin{enumerate} \item[a)] The map \begin{equation}\label{leafsphommoreq} h_P:\underline{M}_1\to \underline{M}_2, \quad \L_1\mapsto \alpha_2(\alpha_1^{-1}(\L_1))\end{equation} that sends a leaf $\L_1$ of $\G_1$ to the unique $P$-related leaf of $\G_2$ is a homeomorphism.
\item[b)] Suppose that $x_1\in M_1$ and $x_2\in M_2$ belong to $P$-related leaves and let $p\in P$ such that $\alpha_1(p)=x_1$ and $\alpha_2(p)=x_2$. Then the map:
\begin{equation}\label{isotgpisomoreq} \Phi_p:(\G_1)_{x_1}\to (\G_2)_{x_2} 
\end{equation} defined by the relation:
\begin{equation*} g\cdot p=p\cdot \Phi_p(g), \quad g\in (\G_1)_{x_1},
\end{equation*} is an isomorphism of Lie groups. Furthermore, the map: \begin{equation}\label{isotrepisomoreq} \phi_p:\No_{x_1}\to \No_{x_2},\quad [v]\mapsto [d\alpha_2(\hat{v})],
\end{equation} where $\hat{v}\in T_pP$ is any tangent vector such that $d\alpha_1(\hat{v})=v$, is a compatible isomorphism between the normal representations at $x_1$ and $x_2$.
\end{enumerate} 
\end{prop}

\begin{ex}\label{idmoreqex} Any Lie groupoid $\G\rightrightarrows M$ is Morita equivalent to itself via the canonical bi-module $(\G,t,s)$. The same goes for symplectic groupoids. Another simple example: any transitive Lie groupoid is Morita equivalent to a Lie group (viewed as groupoid over the one-point space); as a particular case of this, the pair groupoid of a manifold is Morita equivalent to the unit groupoid of the one-point space.
\end{ex}

\begin{ex}\label{crucexmoreq2} Morita equivalences can be restricted to opens. Indeed, let $(P,\alpha_1,\alpha_2)$ be a Morita equivalence between $\G_1\rightrightarrows M_1$ and $\G_2\rightrightarrows M_2$, and let $V_1$ be an open in $M_1$. Then $V_2:=\alpha_2(\alpha_1^{-1}(V_1))$ is an invariant open in $M_2$ and $(\alpha_1^{-1}(V_1),\alpha_1,\alpha_2)$ is a Morita equivalence between $\G_1\vert_{V_1}$ and $\G_2\vert_{V_2}$. In particular, given a $\G\rightrightarrows M$ Lie groupoid and an open $V\subset M$, letting $\widehat{V}:=s(t^{-1}(V))$ denote the saturation of $V$ (the smallest invariant open containing $V$), the first Morita equivalence in Example \ref{idmoreqex} restricts to one between $\G\vert_V$ and $\G\vert_{\widehat{V}}$. The same goes for symplectic Morita equivalences. 
\end{ex}

\begin{ex}\label{crucexmoreq} 
The following example plays a crucial role in our proof of Theorem \ref{righamthm}. Consider the set-up of Subsection \ref{locmodconstsec}. There is a canonical symplectic Morita equivalence:
\begin{center}
\begin{tikzpicture} \node (G1) at (0,0) {$\left(\G_\theta,\Omega_\theta\right)$};
\node (M1) at (0,-1.3) {$M_\theta$};
\node (S) at (3.3,0) {$(\Sigma_\theta,\omega_\theta)$};
\node (M2) at (6.6,-1.3) {$W_\theta$};
\node (G2) at (6.6,0) {$(G\ltimes \g^*,-\d \lambda_{\textrm{can}})\vert_{W_\theta}$};
 
\draw[->,transform canvas={xshift=-\shift}](G1) to node[midway,left] {}(M1);
\draw[->,transform canvas={xshift=\shift}](G1) to node[midway,right] {}(M1);
\draw[->,transform canvas={xshift=-\shift}](G2) to node[midway,left] {}(M2);
\draw[->,transform canvas={xshift=\shift}](G2) to node[midway,right] {}(M2);
\draw[->](S) to node[pos=0.25, below] {$\quad\quad\textrm{pr}_{M_\theta}$} (M1);
\draw[->] (2.25,-0.15) arc (315:30:0.25cm);
\draw[<-] (4.35,0.15) arc (145:-145:0.25cm);
\draw[->](S) to node[pos=0.25, below] {$\textrm{pr}_{\g^*}$\text{ }} (M2);
\end{tikzpicture}
\end{center}
between (\ref{locmodsymgp}) and the restriction of the cotangent groupoid to the $G$-invariant open $W_\theta:=\textrm{pr}_{\g^*}(\Sigma_\theta)$ around the origin in $\g^*$. This relates the central leaf $\L$ in $M_\theta$ to the origin in $\g^*$. 
\end{ex}


\subsubsection{Equivalence between categories of modules} Next, we recall how a Morita equivalence induces an equivalence between the categories of modules. Given a Lie groupoid $\G\rightrightarrows M$, by a \textbf{$\G$-module} we simply mean smooth map $J:S\to M$ equipped with a left action of $\G$. A morphism from a $\G$-module $J_1:S_1\to M$ to $J_2:S_2\to M$ is smooth map $\phi:S_1\to S_2$ that intertwines $J_1$ and $J_2$ and is $\G$-equivariant. This defines a category $\textsf{Mod}(\G)$. 
\begin{ex}\label{restrinvopmodex} Let $\G\rightrightarrows M$ be a Lie groupoid and let $W$ be an invariant open in $M$. Consider the full subcategory $\textsf{Mod}_W(\G)$ of $\textsf{Mod}(\G)$ consisting of those $\G$-modules $J:S\to M$ with the property that $J(S)\subset W$. There is a canonical equivalence of categories between $\textsf{Mod}_W(\G)$ and $\textsf{Mod}(\G\vert_W)$.
\end{ex}
\begin{ex}\label{eqmapmod} Let $G$ be a Lie group and $M$ a left $G$-space. Consider the category $\textsf{Hom}_G(-,M)$ of smooth $G$-equivariant maps from left $G$-spaces into $M$. A morphism between two such maps $J_1:S_1\to M$ and $J_2:S_2\to M$ is a smooth $G$-equivariant map $\phi:S_1\to S_2$ that intertwines $J_1$ and $J_2$. There is a canonical equivalence of categories between $\textsf{Hom}_G(-,M)$ and $\textsf{Mod}(G\ltimes M)$. 
\end{ex}
We now recall:
\begin{thm}\label{moreqasrep} A Morita equivalence $(P,\alpha_1,\alpha_2)$ between two Lie groupoids $\G_1$ and $\G_2$ induces an equivalence of categories between $\textsf{Mod}(\G_1)$ and $\textsf{Mod}(\G_2)$, explicitly given by (\ref{moreqfunct}).
\end{thm}
\begin{proof} To any $\G_1$-module $J:S\to M_1$ we can associate a $\G_2$-module, as follows. The Lie groupoid $\G_1$ acts diagonally on the manifold $P\times_{M_1} S$ along the map $\alpha_1\circ \textrm{pr}_1$, in a free and proper way. Hence, the quotient:
\begin{equation*} P\ast_{\G_1} S:=\frac{(P\times_{M_1} S)}{\G_1}
\end{equation*} is smooth. Moreover, since the actions of $\G_1$ and $\G_2$ commute and $\alpha_1$ is $\G_2$-invariant, we have a left action of $\G_2$ along: 
\begin{equation}\label{asmod} P_*(J):P\ast_{\G_1} S\to M_2, \quad [p_P,p_S]\mapsto \alpha_2(p_P), 
\end{equation} given by:
\begin{equation*} g\cdot [p_P,p_S]=[p_P\cdot g^{-1},p_S].
\end{equation*} We call this the $\G_2$-module associated to the $\G_1$-module $J$. For any morphism of $\G_1$-modules there is a canonical morphism between the associated $\G_2$-modules. So, this defines a functor:
\begin{align}\label{moreqfunct} \textsf{Mod}(\G_1)&\to\textsf{Mod}(\G_2) \\
(J:S\to M_1)&\mapsto (P_*(J):P\ast_{\G_1}S\to M_2). \nonumber
\end{align} An analogous construction from right to left gives an inverse to this functor.
\end{proof}

Next, we recall the analogue for symplectic groupoids. Given a symplectic groupoid $(\G,\Omega)\rightrightarrows M$, by a \textbf{Hamiltonian $(\G,\Omega)$-space} (called \textbf{symplectic left $(\G,\Omega)$-module} in \cite{Xu}) we mean a smooth map $J:(S,\omega)\to M$ equipped with a Hamiltonian $(\G,\Omega)$-action. A morphism $\phi$ from $J_1:(S_1,\omega_1)\to M$ to $J_2:(S_2,\omega_2)\to M$ is a morphism of $\G$-modules satisfying $\phi^*\omega_2=\omega_1$. This defines a category $\textsf{Ham}(\G,\Omega)$. 
\begin{ex}\label{restrinvopsympmodex} Let $(\G,\Omega)\rightrightarrows M$ be a symplectic groupoid and let $W$ be an invariant open in $M$. The equivalence in Example \ref{restrinvopmodex} restricts to an equivalence between the category $\textsf{Ham}_W(\G,\Omega)$, consisting of Hamiltonian $(\G,\Omega)$-spaces with the property that $J(S)\subset W$, and $\textsf{Ham}((\G,\Omega)\vert_W)$.
\end{ex}
\begin{ex}\label{eqcatmodHamGex} Let $G$ be a Lie group and consider the category $\textsf{Ham}(G)$ of left Hamiltonian $G$-spaces. Here, a morphism between Hamiltonian $G$-spaces $J_1:(S_1,\omega_1)\to \g^*$ and $J_2:(S_2,\omega_2)\to \g^*$ is a $G$-equivariant map $\phi:S_1\to S_2$ that intertwines $J_1$ and $J_2$ and satisfies $\phi^*\omega_2=\omega_1$. The equivalence in Example \ref{eqmapmod} restricts to one between $\textsf{Ham}(G)$ and $\textsf{Ham}(G\ltimes \g^*,-\d\lambda_{\textrm{can}})$. This refines the statement in Example \ref{exhamGsp}. 
\end{ex} 
\begin{thm}[\cite{Xu}]\label{sympmoreqasrep} A symplectic Morita equivalence $(P,\omega_P,\alpha_1,\alpha_2)$ between two symplectic groupoids $(\G_1,\Omega_1)$ and $(\G_2,\Omega_2)$ induces an equivalence of categories between $\textsf{Ham}(\G_1,\Omega_1)$ and $\textsf{Ham}(\G_2,\Omega_2)$, explicitly given by (\ref{sympmoreqfunct}). 
\end{thm}
\begin{proof} Let $(P,\omega_P,\alpha_1,\alpha_2)$ be a symplectic Morita equivalence between symplectic groupoids $(\G_1,\Omega_1)$ and $(\G_2,\Omega_2)$ and let $J:(S,\omega_S)\to M_1$ be a Hamiltonian $(\G_1,\Omega_1)$-space. The symplectic form $(-\omega_P)\oplus \omega_S$ descends to a symplectic form $\omega_{PS}$ on $P\ast_{\G_1}S$ and the $(\G_2,\Omega_2)$-action along the associated module $P_*(J)$, as in (\ref{asmod}), becomes Hamiltonian. As before, this extends to a functor:
\begin{align}\label{sympmoreqfunct} \textsf{Ham}(\G_1,\Omega_1)&\to\textsf{Ham}(\G_2,\Omega_2) \\
(J:(S,\omega_S)\to M_1)&\mapsto (P_*(J):(P\ast_{\G_1}S,\omega_{PS})\to M_2) \nonumber
\end{align} and an analogous construction from right to left gives an inverse functor.
\end{proof}
\subsubsection{Proof of rigidity}
The proof of Theorem \ref{righamthm} hinges on the following two known results. The first is a rigidity theorem for symplectic groupoids.
\begin{thm}[\cite{CrFeTo1}]\label{normformsymgp} Suppose that we are given two realizations of the same zeroth-order symplectic groupoid data with leaf $\L$. Fix $x\in \L$. If both symplectic groupoids are proper at $x$ (in the sense of Definition \ref{propatxdefi}), then the realizations are neighbourhood-equivalent.
\end{thm}
\begin{rem} The assumption appearing in \cite[Thm 8.2]{CrFeTo1} is that $\G$ is proper, which is stronger than properness at $x$. However, if $\G$ is proper at $x$, then there is an open  $U$ around the leaf $\L$ through $x$ such that $\G\vert_U$ is proper (see e.g. \cite[Remark 5.1.4]{Hoy}). 
\end{rem}

The second result that we will need is the following rigidity theorem for Hamiltonian $G$-spaces.  
\begin{thm}[\cite{Ma,GS4}]\label{mgsrigzerofib} Let $G$ be a compact Lie group and let $J_1:(S_1,\omega_1)\to \g^*$ and $J_2:(S_2,\omega_2)\to \g^*$ be Hamiltonian $G$-spaces. Suppose that $p_1\in J_1^{-1}(0)$ and $p_2\in J_2^{-1}(0)$ are such that $G_{p_1}=G_{p_2}$. Then there are $G$-invariant neighbourhoods $U_1$ of $p_1$ and $U_2$ of $p_2$, together with an isomorphism of Hamiltonian $G$-spaces that sends $p_1$ to $p_2$: 
\begin{center}
\begin{tikzcd} (U_1,\omega_1,p_1)\arrow[rr,"\sim"]\arrow[dr,"J_1"'] & & (U_2,\omega_2,p_2)\arrow[dl, "J_2"] \\
& \g^* &
\end{tikzcd}
\end{center} if and only if there is an equivariant symplectic linear isomorphism: 
\begin{equation*} (\S\No_{p_1},\omega_{p_1})\cong (\S\No_{p_2},\omega_{p_2}).
\end{equation*} 
\end{thm}
The main step in proving Theorem \ref{righamthm} is to prove the following generalization of Theorem \ref{mgsrigzerofib}. 
\begin{thm}\label{loceqthmhamact} Let $(\G,\Omega)\rightrightarrows M$ be a symplectic groupoid that is proper at $x\in M$. Suppose that we are given two Hamiltonian $(\G,\Omega)$-spaces $J_1:(S_1,\omega_1)\to M$ and $J_2:(S_2,\omega_2)\to M$. Let $p_1\in S_1$ and $p_2\in S_2$ be such that $J_1(p_1)=J_2(p_2)=x$ and $\G_{p_1}=\G_{p_2}$. Then there are $\G$-invariant open neighbourhoods $U_1$ of $p_1$ and $U_2$ of $p_2$, together with an isomorphism of Hamiltonian $(\G,\Omega)$-spaces that sends $p_1$ to $p_2$:
\begin{center}
\begin{tikzcd} (U_1,\omega_1,p_1) \arrow[rr,"\sim"] \arrow[dr, "J_1"'] & & (U_2,\omega_2,p_2) \arrow[dl,"J_2"] \\
 & M & 
\end{tikzcd}
\end{center} if and only if there is an equivariant symplectic linear isomorphism:
\begin{equation*} (\S\No_{p_1},\omega_{p_1})\cong (\S\No_{p_2},\omega_{p_2}).
\end{equation*} 
\end{thm} 
To prove this we further use the lemma below.
\begin{lemma}\label{assrepprop} Let $(P,\omega_P,\alpha_1,\alpha_2)$ be a symplectic Morita equivalence between $(\G_1,\Omega_1)$ and $(\G_2,\Omega_2)$. Further, let $J:(S,\omega_S)\to M$ be a Hamiltonian $(\G_1,\Omega_1)$-space, let $p_S\in S$ and fix a $p_P\in P$ such that $\alpha_1(p_P)=J(p_S)$. Then the isomorphism (\ref{isotgpisomoreq}) restricts to an isomorphism:
\begin{equation*} \Phi_{p_P}:\G_{p_S}\xrightarrow{\sim} \G_{[p_P,p_S]}, 
\end{equation*}
and there is a compatible symplectic linear isomorphism:
\begin{equation*} \left(\S\No_{p_S},(\omega_S)_{p_S}\right)\cong\left(\S\No_{[p_P,p_S]},(\omega_{PS})_{[p_P,p_S]}\right)
\end{equation*} between the symplectic normal representation at $p_S$ of the Hamiltonian $(\G_1,\Omega_1)$-space $J$ and the symplectic normal representation at $[p_P,p_S]$ of the associated Hamiltonian $(\G_2,\Omega_2)$-space $P_*(J)$ of Theorem \ref{sympmoreqasrep}. 
\end{lemma}
Although this lemma can be verified directly, we postpone its proof to Subsection \ref{bashammorinv}, where we give a more conceptual explanation. With this at hand, we can prove the desired theorems.
\begin{proof}[Proof of Theorem \ref{loceqthmhamact}] The forward implication is straightforward. Let us prove the backward implication. Throughout, let $G:=\G_x$ denote the isotropy group of $\G$ at $x$. To begin with observe that, since $\G$ is proper at $x$, there is an invariant open neighbourhood $V$ of the leaf $\L$ through $x$ and a $G$-invariant open neighbourhood $W$ of the origin in $\g^*$, together with a symplectic Morita equivalence:
\begin{center}
\begin{tikzpicture} \node (G1) at (0,0) {$(\G,\Omega)\vert_V$};
\node (M1) at (0,-1.3) {$V$};
\node (S) at (3.3,0) {$(P,\omega_P)$};
\node (M2) at (6.6,-1.3) {$W$};
\node (G2) at (6.6,0) {$(G\ltimes \g^*,-\d \lambda_{\textrm{can}})\vert_W$};
 
\draw[->,transform canvas={xshift=-\shift}](G1) to node[midway,left] {}(M1);
\draw[->,transform canvas={xshift=\shift}](G1) to node[midway,right] {}(M1);
\draw[->,transform canvas={xshift=-\shift}](G2) to node[midway,left] {}(M2);
\draw[->,transform canvas={xshift=\shift}](G2) to node[midway,right] {}(M2);
\draw[->](S) to node[pos=0.25, below] {$\quad\quad\alpha_1$} (M1);
\draw[->] (2.25,-0.15) arc (315:30:0.25cm);
\draw[<-] (4.35,0.15) arc (145:-145:0.25cm);
\draw[->](S) to node[pos=0.25, below] {$\alpha_2$\text{ }} (M2);
\end{tikzpicture}
\end{center}
that relates the leaf $\L$ to the origin in $\g^*$. Indeed, this follows by first applying Theorem \ref{normformsymgp} to:
\begin{itemize}\item the zeroth-order data of $(\G,\Omega)$ at $\L$, 
\item the canonical realization $(\G,\Omega)$,
\item the realization (\ref{locmodsymgp}),
\end{itemize} and then combining the neighbourhood-equivalence of symplectic groupoids obtained thereby with Examples \ref{crucexmoreq2} and \ref{crucexmoreq}. Since $V$ is $\G$-invariant, so are $J_1^{-1}(V)$ and $J_2^{-1}(V)$ and we can consider the Hamiltonian $(\G,\Omega)$-spaces: 
\begin{equation}\label{restmodloceqthmhamact} {J_1^V}:(J_1^{-1}(V),\omega_1)\to M \quad \&\quad {J_2^V}:(J_2^{-1}(V),\omega_2)\to M
\end{equation} obtained by restricting the given Hamiltonian $(\G,\Omega)$-spaces $J_1$ and $J_2$. By Theorem \ref{sympmoreqasrep}, combined with Examples \ref{restrinvopsympmodex} and \ref{eqcatmodHamGex}, the above Morita equivalence induces an equivalence of categories (with explicit inverse) between the category of Hamiltonian $(\G,\Omega)$-spaces $J:(S,\omega)\to M$ with $J(S)\subset V$ and the category of Hamiltonian $G$-spaces $J:(S,\omega)\to \g^*$ with $J(S)\subset W$. Consider the Hamiltonian $G$-spaces associated to $(\ref{restmodloceqthmhamact})$:
\begin{equation*} P_*(J_1^V):(P\ast_{(\G\vert_V)} J_1^{-1}(V),\omega_{PS_1})\to \g^* \quad \& \quad  P_*(J_2^V):(P\ast_{(\G\vert_V)} J_2^{-1}(V),\omega_{PS_2})\to \g^*,
\end{equation*} and fix a $p\in P$ such that $\alpha_1(p)=x$. We will show that these Hamiltonian $G$-spaces satisfy the assumptions of Theorem \ref{mgsrigzerofib} for the points $[p,p_1]$ and $[p,p_2]$. First of all, since the leaf $\L$ is $P$-related to the origin in $\g^*$, it must be that $\alpha_2(p)=0$. Therefore, we find: 
\begin{equation*} P_*(J_1^V)([p,p_1])=\alpha_2(p)=0 \quad \& \quad P_*(J_2^V)([p,p_2])=\alpha_2(p)=0. 
\end{equation*} Second, Lemma \ref{assrepprop} implies that $G_{[p,p_1]}=G_{[p,p_2]}$, as both coincide with the image of $\G_{p_1}=\G_{p_2}$ under $\Phi_p:\G_x\to G$. Third, by the same lemma, there are symplectic linear isomorphisms:
\begin{equation*} \psi_1:\left(\S\No_{p_1},(\omega_1)_{p_1}\right)\xrightarrow{\sim} \left(\S\No_{[p,p_1]},(\omega_{PS_1})_{[p,p_1]}\right)\quad \& \quad \psi_2:\left(\S\No_{p_2},(\omega_2)_{p_2}\right)\xrightarrow{\sim} \left(\S\No_{[p,p_2]},(\omega_{PS_2})_{[p,p_2]}\right),
\end{equation*} that are both compatible with the isomorphism of Lie groups:
\begin{equation*} \G_{p_1}\xrightarrow{\Phi_p} G_{[p,p_1]}=\G_{p_2}\xrightarrow{\Phi_p} G_{[p,p_2]}.
\end{equation*}
By assumption, there is an equivariant symplectic linear isomorphism:
\begin{equation*} \psi:(\S\No_{p_1},\omega_{p_1})\xrightarrow{\sim} (\S\No_{p_2},\omega_{p_2}).
\end{equation*} All together, the composition: 
\begin{equation*} \psi_2\circ \psi\circ \psi_1^{-1}: \left(\S\No_{[p,p_1]},(\omega_{PS_1})_{[p,p_1]}\right)\xrightarrow{\sim} \left(\S\No_{[p,p_2]},(\omega_{PS_2})_{[p,p_2]}\right)
\end{equation*} becomes an equivariant symplectic linear isomorphism. So, the assumptions of Theorem \ref{mgsrigzerofib} hold, which implies that there are $G$-invariant opens $U_{[p,p_1]}$ around $[p,p_1]$ and $U_{[p,p_2]}$ around $[p,p_2]$, together with an isomorphism of Hamiltonian $G$-spaces that sends $[p,p_1]$ to $[p,p_2]$: 
\begin{center}
\begin{tikzcd} \left(U_{[p,p_1]},\omega_{PS_1},[p,p_1]\right)\arrow[rr,"\sim"]\arrow[dr,"P_*(J_1^V)"'] & & \left(U_{[p,p_2]},\omega_{PS_2},[p,p_2]\right)\arrow[dl, "P_*(J_2^V)"] \\
& \g^* &
\end{tikzcd}
\end{center} One readily verifies that, by passing back through the above equivalence of categories via the explicit inverse functor, we obtain $\G$-invariant opens $U_1$ around $p_1$ and $U_2$ around $p_2$, together with an isomorphism of Hamiltonian $(\G,\Omega)$-spaces from $J_1:(U_1,\omega_1)\to M$ to $J_2:(U_2,\omega_2)\to M$ that sends $p_1$ to $p_2$, as desired.  
\end{proof}
\begin{proof}[Proof of Theorem \ref{righamthm}] As in the previous proof, the forward implication is straightforward. For the backward implication, let $(i_1,j_1)$ and $(i_2,j_2)$ be two realizations of the same zeroth-order Hamiltonian data (with notation as in Definition \ref{nhoodeq1defi}). Let $p\in \O$ and $x=J_\O(p)$ and suppose that their symplectic normal representations at $p$ are isomorphic as symplectic $\G_p$-representations. By Theorem \ref{normformsymgp} there are respective opens $V_1$ and $V_2$ around $\L$ in $M_1$ and $M_2$, together with an isomorphism: 
\begin{equation*} \Phi:(\G_1,\Omega_1)\vert_{V_1}\xrightarrow{\sim}(\G_2,\Omega_2)\vert_{V_2}
\end{equation*} 
 that intertwines $i_1$ with $i_2$. Consider, on one hand, the Hamiltonian $(\G_1,\Omega_1)\vert_{V_1}$-space obtained from the given Hamiltonian $(\G_1,\Omega_1)$-space $J_1$ by restriction to $V_1$ and, on the other hand, the Hamiltonian $(\G_1,\Omega_1)\vert_{V_1}$-space $\Phi^*(J_2)$ obtained from the given Hamiltonian $(\G_2,\Omega_2)$-space $J_2$ by restriction to $V_2$ and pullback along $\Phi$. These two Hamiltonian $(\G_1,\Omega_1)\vert_{V_1}$-spaces meet the assumptions of Theorem \ref{loceqthmhamact} at the points $j_1(p)$ and $j_2(p)$. So, there are $(\G_1\vert_{V_1})$-invariant opens $U_1\subset J_1^{-1}(V_1)$ and $U_2\subset J_2^{-1}(V_2)$, together with an isomorphism of Hamiltonian $(\G_1,\Omega_1)\vert_{V_1}$-spaces that sends $j_1(p)$ to $j_2(p)$:
\begin{center}
\begin{tikzcd} (U_1,\omega_1,j_1(p)) \arrow[rr,"\Psi"] \arrow[dr, "J_1"'] & & (U_2,\omega_2,j_2(p)) \arrow[dl,"\Phi^*(J_2)"] \\
 & V_1 & 
\end{tikzcd}
\end{center} As one readily verifies, the pair $(\Phi,\Psi)$ is the desired neighbourhood equivalence.
\end{proof}

\subsection{The transverse part of the local model}\label{translocmodsec}
\subsubsection{Hamiltonian Morita equivalence} In order to define a notion of Morita equivalence between Hamiltonian actions, we first consider a natural equivalence relation between Lie groupoid maps (resp. groupoid maps of Hamilonian type, defined below). In the next subsection we explain how this restricts to an equivalence relation between Lie groupoid actions (resp. Hamiltonian actions). 
\begin{defi}\label{moreqdefLie} Let $\J_1:\H_1\to \G_1$ and $\J_2:\H_2\to \G_2$ be maps of Lie groupoids. By a \textbf{Morita equivalence} from $\J_1$ to $\J_2$ we mean the data consisting of:
\begin{itemize}\item a Morita equivalence $(P,\alpha_1,\alpha_2)$ from $\G_1$ to $\G_2$,
\item a Morita equivalence $(Q,\beta_1,\beta_2)$ from $\H_1$ to $\H_2$,
\item a smooth map $j:Q\to P$ that interwines $J_i\circ\beta_i$ with $\alpha_i$ and that intertwines the $\H_i$-action with the $\G_i$-action via $\J_i$, for both $i=1,2$.
\end{itemize} We depict this as:
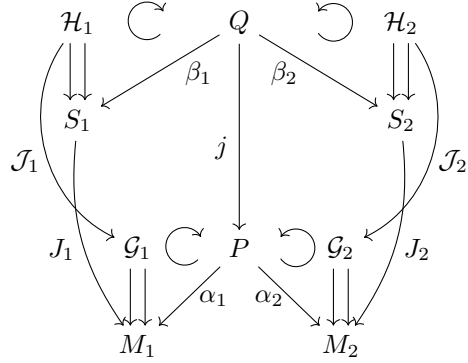
\begin{center}
\begin{tikzpicture} 
\node (H1) at (-0.8,0) {$\H_1$};
\node (S1) at (-0.8,-1.3) {$S_1$};
\node (Q) at (1.35,0) {$Q$};
\node (S2) at (3.5,-1.3) {$S_2$};
\node (H2) at (3.5,0) {$\H_2$};

\node (G1) at (0,-3) {$\G_1$};
\node (M1) at (0,-4.3) {$M_1$};
\node (P) at (1.35,-3) {$P$};
\node (M2) at (2.7,-4.3) {$M_2$};
\node (G2) at (2.7,-3) {$\G_2$};

\draw[->, bend right=50](H1) to node[pos=0.45,below] {$\J_1\text{ }\text{ }\text{ }\text{ }$} (G1);
\draw[->, bend right=20](S1) to node[pos=0.45,below] {$J_1\text{ }\text{ }\text{ }\text{ }$} (M1);
\draw[->, bend left=50](H2) to node[pos=0.45,below] {$\text{ }\text{ }\text{ }\text{ }\J_2$} (G2);
\draw[->, bend left=20](S2) to node[pos=0.45,below] {$\text{ }\text{ }\text{ }\text{ }J_2$} (M2);
\draw[->](Q) to node[pos=0.45,below] {$j\text{ }\text{ }\text{ }\text{ }$} (P);
 
\draw[->,transform canvas={xshift=-\shift}](H1) to node[midway,left] {}(S1);
\draw[->,transform canvas={xshift=\shift}](H1) to node[midway,right] {}(S1);
\draw[->,transform canvas={xshift=-\shift}](H2) to node[midway,left] {}(S2);
\draw[->,transform canvas={xshift=\shift}](H2) to node[midway,right] {}(S2);
\draw[->](Q) to node[pos=0.25, below] {$\text{ }\text{ }\beta_1$} (S1);
\draw[->] (0.3,-0.15) arc (315:30:0.25cm);
\draw[<-] (2.4,0.15) arc (145:-145:0.25cm);
\draw[->](Q) to node[pos=0.25, below] {$\beta_2$\text{ }} (S2); 
 
\draw[->,transform canvas={xshift=-\shift}](G1) to node[midway,left] {}(M1);
\draw[->,transform canvas={xshift=\shift}](G1) to node[midway,right] {}(M1);
\draw[->,transform canvas={xshift=-\shift}](G2) to node[midway,left] {}(M2);
\draw[->,transform canvas={xshift=\shift}](G2) to node[midway,right] {}(M2);
\draw[->](P) to node[pos=0.25, below] {$\text{ }\text{ }\alpha_1$} (M1);
\draw[->] (0.8,-3.15) arc (315:30:0.25cm);
\draw[<-] (1.9,-2.85) arc (145:-145:0.25cm);
\draw[->](P) to node[pos=0.25, below] {$\alpha_2$\text{ }} (M2);
\end{tikzpicture}
\end{center}
\end{defi}
As an analogue of this in the Hamiltonian setting, we propose the following definitions (more motivation for which will be given in the coming subsections).
\begin{defi}\label{gpoidmaphamtyp} Let $(\G,\Omega)\rightrightarrows M$ be a symplectic groupoid and let $\H\rightrightarrows (S,\omega)$ be a Lie groupoid over a pre-symplectic manifold. We call a Lie groupoid map $\J:\H\to \G$ of \textbf{Hamiltonian type} if:
\begin{equation*} \J^*\Omega=(t_\H)^*\omega-(s_\H)^*\omega.
\end{equation*} 
\end{defi}
\begin{defi}\label{moreqdefHam} Let $\J_1:\H_1\to \G_1$ and $\J_2:\H_2\to \G_2$ be of Hamiltonian type. By a \textbf{Hamiltonian Morita equivalence} from $\J_1$ to $\J_2$ we mean: a Morita equivalence (in the sense of Definition \ref{moreqdefLie}) with the extra requirement that $(P,\omega_P,\alpha_1,\alpha_2)$ is a symplectic Morita equivalence and that: 
\begin{equation}\label{hameqmor1} j^*\omega_P=(\beta_1)^*\omega_1-(\beta_2)^*\omega_2.
\end{equation} 
\end{defi}
The same type of arguments as for Morita equivalence of Lie and symplectic groupoids (see \cite{Xu}) show that Hamiltonian Morita equivalence indeed defines an equivalence relation. 
\subsubsection{Morita equivalence between groupoid maps of action type} To see that the equivalence relation(s) in the previous subsection induce an equivalence relation between Lie groupoid actions (resp. Hamiltonian actions), the key remark is that a left action of a Lie groupoid $\G$ along a map $J:S\to M$ gives rise to a map of Lie groupoids covering $J$:
\begin{equation}\label{asslgpoidmapact} \textrm{pr}_\G:\G\ltimes S\to \G.
\end{equation} 
Further, notice that the groupoid map (\ref{asslgpoidmapact}) is of Hamiltonian type precisely when the action is Hamiltonian (that is, when (\ref{hammultcond}) holds). 
\begin{defi}\label{moreqdefHamacttyp} By a \textbf{Morita equivalence between (left) Lie groupoid actions} we mean a Morita equivalence between their associated Lie groupoid maps (\ref{asslgpoidmapact}). Similarly, by a \textbf{Morita equivalence between (left) Hamiltonian actions} we mean a Hamiltonian Morita equivalence between their associated groupoid maps (\ref{asslgpoidmapact}). 
\end{defi}
In the remainder of this subsection, we further unravel what it means for to Hamiltonian actions to be Morita equivalent. The starting point for this is the following example, which concerns the modules appearing in Theorems \ref{moreqasrep} and \ref{sympmoreqasrep}.
\begin{ex}\label{asrepmoreq} Let $\G_1\rightrightarrows M_1$ be a Lie groupoid acting along $J:S\to M_1$ and suppose that we are given a Morita equivalence $(P,\alpha_1,\alpha_2)$ from $\G_1$ to another Lie groupoid $\G_2\rightrightarrows M_2$. Consider the associated $\G_2$-action along $P_*(J):P\ast_{\G_1}S\to M_2$. The Morita equivalence from $\G_1$ to $\G_2$ extends to a canonical Morita equivalence between these two actions:
\begin{center}
\begin{tikzpicture} 
\node (H1) at (-1,0) {$\G_1\ltimes S$};
\node (S1) at (-1,-1.3) {$S$};
\node (Q) at (1.35,0) {$P\times_{M_1} S$};
\node (S2) at (3.9,-1.3) {$P\ast_{\G_1}S$};
\node (H2) at (3.9,0) {$\G_2\ltimes \left(P\ast_{\G_1}S\right)$};

\node (G1) at (0,-3) {$\G_1$};
\node (M1) at (0,-4.3) {$M_1$};
\node (P) at (1.35,-3) {$P$};
\node (M2) at (2.7,-4.3) {$M_2$};
\node (G2) at (2.7,-3) {$\G_2$};

\draw[->, bend right=50](H1) to node[pos=0.45,below] {$\textrm{pr}_{\G_1}\quad\quad$} (G1);
\draw[->, bend right=20](S1) to node[pos=0.45,below] {$J\text{ }\text{ }\text{ }\text{ }$} (M1);
\draw[->, bend left=50](H2) to node[pos=0.45,below] {$\quad\quad\quad\textrm{pr}_{\G_2}$} (G2);
\draw[->, bend left=20](S2) to node[pos=0.45,below] {$\quad\quad\text{ }\text{ }P_*(J)$} (M2);
\draw[->](Q) to node[pos=0.45,below] {$\textrm{pr}_P\quad\quad$} (P);
 
\draw[->,transform canvas={xshift=-\shift}](H1) to node[midway,left] {}(S1);
\draw[->,transform canvas={xshift=\shift}](H1) to node[midway,right] {}(S1);
\draw[->,transform canvas={xshift=-\shift}](H2) to node[midway,left] {}(S2);
\draw[->,transform canvas={xshift=\shift}](H2) to node[midway,right] {}(S2);
\draw[->](Q) to node[pos=0.25, below] {$\quad\text{ }\textrm{pr}_{S}$} (S1);
\draw[->] (0.5,-0.15) arc (315:30:0.25cm);
\draw[<-] (2.2,0.15) arc (145:-145:0.25cm);
\draw[->](Q) to node[pos=0.25, below] {$\textrm{pr}_{PS}$\text{ }} (S2); 
 
\draw[->,transform canvas={xshift=-\shift}](G1) to node[midway,left] {}(M1);
\draw[->,transform canvas={xshift=\shift}](G1) to node[midway,right] {}(M1);
\draw[->,transform canvas={xshift=-\shift}](G2) to node[midway,left] {}(M2);
\draw[->,transform canvas={xshift=\shift}](G2) to node[midway,right] {}(M2);
\draw[->](P) to node[pos=0.25, below] {$\text{ }\text{ }\alpha_1$} (M1);
\draw[->] (0.8,-3.15) arc (315:30:0.25cm);
\draw[<-] (1.9,-2.85) arc (145:-145:0.25cm);
\draw[->](P) to node[pos=0.25, below] {$\alpha_2$\text{ }} (M2);
\end{tikzpicture}
\end{center}
Here the upper left action is induced by the diagonal $\G_1$-action, whereas the upper right action is induced by the $\G_2$-action on the first factor. When $(\G_1,\Omega_1)$ and $(\G_2,\Omega_2)$ are symplectic groupoids, the action along $J_1:(S_1,\omega_1)\to M_1$ is Hamiltonian, and the Morita equivalence $(P,\omega_P,\alpha_1,\alpha_2)$ is symplectic, then the associated $(\G_2,\Omega_2)$-action along $P_*(J):(P\ast_{\G_1}S,\omega_{PS})\to M_2$ is Hamiltonian. In this case, the above Morita equivalence is Hamiltonian. 
\end{ex}
In fact, we will show that more is true:
\begin{prop}\label{moreqactgpoids} Every Morita equivalence between two Lie groupoid maps that are both of action type is of the form of Example \ref{asrepmoreq}. The same holds for Hamiltonian Morita equivalence. 
\end{prop}
Here, for convenience, we used the following terminology. 
\begin{defi} Let $\J:\H\to \G$ be map of Lie groupoids covering $J:S\to M$. We say that $\J$ is of \textbf{action type} if there is a smooth left action of $\G$ along $J$ and an isomorphism of Lie groupoids from $\G\ltimes S$ to $\H$ that covers the identity on $S$ and makes the diagram: 
\begin{center}
\begin{tikzcd} 
\G\ltimes S\arrow[rr,"\sim"]\arrow[dr,"\textrm{pr}_\G"'] & & \H \arrow[ld,"\J"]  \\
& \G &
\end{tikzcd}
\end{center}
commute.
\end{defi} 
This has the following more insightful characterization. 
\begin{prop}\label{acttypcharprop} A Lie groupoid map $\mathcal{J}:\H\to\G$ is of action type if and only if for every $p\in S$ the map $\mathcal{J}$ restricts to a diffeomorphism from the source-fiber of $\H$ over $p$ onto that of $\G$ over $J(p)$. \end{prop}
This is readily verified. To prove Proposition \ref{moreqactgpoids} we use the closely related lemma below, the proof of which is also left to the reader. 
\begin{lemma}\label{acttypeprop} Let $\J_1:\H_1\to \G_1$ and $\J_2:\H_2\to \G_2$ be maps of Lie groupoids and let a Morita equivalence between them (denoted as Definition \ref{moreqdefLie}) be given. Let $q\in Q$, and denote $p=j(q)$, $p_i=\beta_i(q)$ and $x_i=J_i(p_i)$ for $i=1,2$. Then we have a commutative square:
\begin{center} 
\begin{tikzcd} \beta_2^{-1}(p_2)\arrow[r,"j"] & \alpha_2^{-1}(x_2)\\
s_{\H_1}^{-1}(p_1)\arrow[r,"\J_1"] \arrow[u,"m_q"] & s^{-1}_{\G_1}(x_1) \arrow[u,"m_p"]
\end{tikzcd}
\end{center}
in which all vertical arrows are diffeomorphisms. In particular, $\J_1$ is of action type if and only if $j$ restricts to a diffeomorphism between the $\beta_2$- and $\alpha_2$-fibers. Analogous statements hold for $\J_2$, replacing $\alpha_2$ and $\beta_2$ by $\alpha_1$ and $\beta_1$. 
\end{lemma} 

\begin{proof}[Proof of Proposition \ref{moreqactgpoids}] Suppose that we are given Lie groupoids $\G_1\rightrightarrows M_1$ and $\G_2\rightrightarrows M_2$, together with a $\G_1$-module $J_1:S_1\to M_1$ and a $\G_2$-module $J_2:S_2\to M_2$ and a Morita equivalence between the associated Lie groupoid maps (\ref{asslgpoidmapact}), denoted as in Definition \ref{moreqdefLie}. It follows from Lemma \ref{acttypeprop} that the map:
\begin{equation}\label{isomoreqactgpoidspf} (j,\beta_1):Q\to P\times_{M_1}S_1
\end{equation} is a diffeomorphism. The diagonal action of $\G_1$ along $\alpha_1\circ \textrm{pr}_P:P\times_{M_1}S_1\to M_1$ induces an action of $\G_1\ltimes S_1$ along $\textrm{pr}_{S_1}:P\times_{M_1}S_1\to S_1$, which is the upper left action in Example \ref{asrepmoreq}. The diffeomorphism $(\ref{isomoreqactgpoidspf})$ intertwines $\beta_1$ with $\text{pr}_{S_1}$ and is equivariant with respect this action. In particular, by principality of the $\G_1\ltimes S_1$-action, there is an induced diffeomorphism:
\begin{equation}\label{isomoreqactgpoidspf2} S_2\xrightarrow{\sim} P\ast_{\G_1}S_1.
\end{equation} One readily verifies that, when identifying $Q$ with $P\times_{M_1}S_1$ via (\ref{isomoreqactgpoidspf}) and $S_2$ with $P\ast_{\G_1}S_1$ via (\ref{isomoreqactgpoidspf2}), the given Morita equivalence is identified with that in Example \ref{asrepmoreq}. Furthermore, when $(\G_1,\Omega_1)$ and $(\G_2,\Omega_2)$ are symplectic groupoids, $(S_1,\omega_1)$ and $(S_2,\omega_2)$ are symplectic manifolds, the given actions along $J_1$ and $J_2$ are Hamiltonian and the Morita equivalence between $(\G_1,\Omega_1)$ and $(\G_2,\Omega_2)$ is symplectic, then one readily verifies that (\ref{isomoreqactgpoidspf2}) is a symplectomorphism from $(S_2,\omega_2)$ to $(P\ast_{\G_1}S_1,\omega_{PS_1})$ if and only if the relation (\ref{hameqmor1}) is satisfied. This proves the proposition. 
\end{proof}
\subsubsection{The transverse local model}
In this paper we will mainly be interested in Hamiltonian Morita equivalences between Hamiltonian actions, rather than between the more general groupoid maps of Hamiltonian type (as in Definition \ref{gpoidmaphamtyp}). There is, however, one important exception to this:
\begin{ex}\label{locmodmoreq} This example gives a Hamiltonian Morita equivalence between the local model for Hamiltonian actions and a groupoid map $\J_\p$ that is built out of less data and is often easier to work with. The use of this Morita equivalence makes many of the proofs in Section \ref{canhamstratsec} both simpler and more conceptual. Let $(\G_\theta,\Omega_\theta)$ be the symplectic groupoid (\ref{locmodsymgp}) and let $J_\theta:(S_\theta,\omega_{S_\theta})\to M_\theta$ be the Hamiltonian $(\G_\theta,\Omega_\theta)$-space (\ref{locmodham2}). The Morita equivalence of Example \ref{crucexmoreq} extends to a Hamiltonian Morita equivalence between the action along $J_\theta$ and a groupoid map of Hamiltonian type from $H\ltimes (\h^0\oplus V)$ to $G\ltimes \g^*$ (restricted to appropriate opens). To see this, let $\p:\h^*\to \g^*$ be an $H$-equivariant splitting of (\ref{ses2poisabs}). Consider the $H$-equivariant map: \begin{equation}\label{slicemommap} J_\p:\h^0\oplus V\to \g^*, \quad (\alpha,v)\mapsto \alpha+\p(J_V(v)),
\end{equation} where $J_V:V\to \h^*$ is the quadratic momentum map (\ref{quadsympmommap}). By $H$-equivariance, this lifts to a groupoid map: 
\begin{equation}\label{bigJ_p} \J_\p:H\ltimes (\h^0\oplus V)\to G\ltimes \g^*,\quad (h,\alpha,v)\mapsto (h,J_\p(\alpha,v)).
\end{equation} This groupoid map is not of action type, but it is of Hamiltonian type with respect to the pre-symplectic form $0\oplus \omega_V$ on $\h^0\oplus V$ and there is a canonical Hamiltonian Morita equivalence:
\begin{center}
\begin{tikzpicture} 
\node (H1) at (-0.85,0) {$\G_\theta\ltimes S_\theta$};
\node (S1) at (-0.85,-1.3) {$(S_\theta,\omega_{S_\theta})$};
\node (Q) at (3.3,0) {$\Sigma_\theta \tensor[_{\textrm{pr}_{\h^*}}]{\times}{_{J_V}} V$};
\node (S2) at (7.45,-1.3) {$(U_\theta,0\oplus \omega_V)$};
\node (H2) at (7.45,0) {$H\ltimes (\h^0\oplus V)\vert_{U_\theta}$};

 \node (G1) at (0,-3) {$(\G_\theta,\Omega_\theta)$};
\node (M1) at (0,-4.3) {$M_\theta$};
\node (P) at (3.3,-3) {$(\Sigma_\theta,\omega_\theta)$};
\node (M2) at (6.6,-4.3) {$W_\theta$};
\node (G2) at (6.6,-3) {$(G\ltimes \g^*,-\d \lambda_{\textrm{can}})\vert_{W_\theta}$};

\draw[->, bend right=80](H1) to node[pos=0.55,below] {$\textrm{pr}_{\G_\theta}\quad\quad$} (G1);
\draw[->, bend right=20](S1) to node[pos=0.15,below] {$J_\theta\text{ }\text{ }\text{ }\text{ }$} (M1);
\draw[->, bend left=80](H2) to node[pos=0.55,below] {$\quad\quad\J_\p$} (G2);
\draw[->, bend left=20](S2) to node[pos=0.15,below] {$\text{ }\text{ }\text{ }\text{ }J_\p$} (M2);
\draw[->](Q) to node[pos=0.45,below] {$\textrm{pr}_{\Sigma_\theta}\quad\quad\quad$} (P);
 
\draw[->,transform canvas={xshift=-\shift}](H1) to node[midway,left] {}(S1);
\draw[->,transform canvas={xshift=\shift}](H1) to node[midway,right] {}(S1);
\draw[->,transform canvas={xshift=-\shift}](H2) to node[midway,left] {}(S2);
\draw[->,transform canvas={xshift=\shift}](H2) to node[midway,right] {}(S2);
\draw[->](Q) to node[pos=0.25, below] {$\text{ }\text{ }\textrm{pr}_{S_\theta}$} (S1);
\draw[->] (1.25,-0.15) arc (315:30:0.25cm);
\draw[<-] (5.35,0.15) arc (145:-145:0.25cm);
\draw[->](Q) to node[pos=0.25, below] {$\beta_\p$\text{ }} (S2); 
 
\draw[->,transform canvas={xshift=-\shift}](G1) to node[midway,left] {}(M1);
\draw[->,transform canvas={xshift=\shift}](G1) to node[midway,right] {}(M1);
\draw[->,transform canvas={xshift=-\shift}](G2) to node[midway,left] {}(M2);
\draw[->,transform canvas={xshift=\shift}](G2) to node[midway,right] {}(M2);
\draw[->](P) to node[pos=0.35, below] {$\quad\quad\textrm{pr}_{M_\theta}$} (M1);
\draw[->] (2.25,-3.15) arc (315:30:0.25cm);
\draw[<-] (4.35,-2.85) arc (145:-145:0.25cm);
\draw[->](P) to node[pos=0.25, below] {$\textrm{pr}_{\g^*}$\quad\quad} (M2);
\end{tikzpicture}
\end{center}
that relates the central orbit in $S_\theta$ to the origin in $\h^0\oplus V$. Here $W_\theta:=\textrm{pr}_\g^*(\Sigma_\theta)$ and $U_\theta:=J_\p^{-1}(W_\theta)$ are invariant open neighbourhoods of the respective origins in $\g^*$ and $\h^0\oplus V$. Furthermore, the map $\beta_\p$ is defined as:
\begin{equation*}\beta_\p: \Sigma_\theta \tensor[_{\textrm{pr}_{\h^*}}]{\times}{_{J_V}} V\to U, \quad (p,\alpha,v)\mapsto (\alpha-\p(J_V(v)),v).
\end{equation*} With this in mind, we think of the groupoid map $\J_\p$ as a local model for the ``transverse part" of a Hamiltonian action near a given orbit.  
\end{ex}

\subsubsection{Elementary Morita invariants}\label{bashammorinv}
As will be apparent in the rest of this paper, many invariants for Morita equivalence between Lie groupoids have analogues for Morita equivalence between Hamiltonian actions \textemdash in fact, the canonical Hamiltonian stratification can be thought of as an analogue of the canonical stratification on the leaf space of a proper Lie groupoid. In this subsection we give analogues of Proposition \ref{transgeomgpoid}. We start with a version for Lie groupoid maps.
\begin{prop}\label{transgeommap} Let $\J_1:\H_1\to \G_1$ and $\J_2:\H_2\to \G_2$ be maps of Lie groupoids and let a Morita equivalence between them (denoted as Definition \ref{moreqdefLie}) be given. 
\begin{itemize}
\item[a)] The induced homeomorphisms between the orbit and leaf spaces (\ref{leafsphommoreq}) intertwine the maps induced by $J_1$ and $J_2$. That is, we have a commutative square:
\begin{center}
\begin{tikzcd} 
\underline{S}_1\arrow[d,"\underline{J}_1"] \arrow[r,"h_Q"]  & \underline{S}_2 \arrow[d,"\underline{J}_2"]\\
\underline{M}_1 \arrow[r, "h_P"] & \underline{M}_2
\end{tikzcd}
\end{center}
\end{itemize}
Further, suppose that $p_1\in S_1$ and $p_2\in S_2$ belong to $Q$-related orbits and let $q\in Q$ such that $\beta_1(q)=p_1$ and $\beta_2(q)=p_2$. Let $p=j(q)$, $x_1=J_1(p_1)$ and $x_2=J_2(p_2)$.  
\begin{itemize}
\item[b)] The induced isomorphisms of isotropy groups (\ref{isotgpisomoreq}) interwine the maps induced by $\J_1$ and $\J_2$. That is, we have a commutative square:
\begin{center}
\begin{tikzcd} 
(\H_1)_{p_1} \arrow[d,"\J_1"] \arrow[r, "\Phi_q"] & (\H_2)_{p_2} \arrow[d,"\J_2"]\\
(\G_1)_{x_1} \arrow[r,"\Phi_p"]  & (\G_2)_{x_2} 
\end{tikzcd}
\end{center}
\item[c)] The induced isomorphisms of normal representations (\ref{isotrepisomoreq}) intertwine the maps induced by $J_1$ and $J_2$. That is, we have a commutative square:
\begin{center}
\begin{tikzcd} 
\No_{p_1}\arrow[d,"\underline{\d J}_1"]\arrow[r,"\phi_q"] & \No_{p_2}\arrow[d,"\underline{\d J}_2"]\\
\No_{x_1} \arrow[r,"\phi_p"] & \No_{x_2}
\end{tikzcd}
\end{center}
\end{itemize}
\end{prop}
The proof is straightforward. 
\begin{ex} The Morita equivalence in Example \ref{locmodmoreq} induces an identification (of maps of topological spaces) between the transverse momentum map $\underline{J_\theta}$ and (a restriction of) the map:
\begin{equation*} \underline{J_\p}:(\h^0\oplus V)/H\to \g^*/G.
\end{equation*}
\end{ex} 
We now turn to Morita equivalences between Hamiltonian actions. 
\begin{prop}\label{transgeomham} Let a Hamiltonian $(\G_1,\Omega_1)$-action along $J_1:(S_1,\omega_1)\to M_1$, a Hamiltonian $(\G_2,\Omega_2)$-action along $J_2:(S_2,\omega_2)\to M_2$ and a Hamiltonian Morita equivalence between them (denoted as in Definitions \ref{moreqdefLie} and \ref{moreqdefHam}) be given. Suppose that $p_1\in S_1$ and $p_2\in S_2$ belong to $Q$-related orbits and let $q\in Q$ such that $\beta_1(q)=p_1$ and $\beta_2(q)=p_2$. Let $p=j(q)$, $x_1=J(p_1)$ and $x_2=J(p_2)$. 
\begin{itemize}\item[a)]The isomorphism $\Phi_p:(\G_1)_{x_1}\xrightarrow{\sim}(\G_2)_{x_2}$ restricts to an isomorphism: 
\begin{equation*} (\G_1)_{p_1}\xrightarrow{\sim}(\G_2)_{p_2}.
\end{equation*}
\item[b)] There is a compatible symplectic linear isomorphism: 
\begin{equation*} (\S\No_{p_1},(\omega_1)_{p_1})\cong(\S\No_{p_2},(\omega_2)_{p_2})
\end{equation*} between the symplectic normal representations at $p_1$ and $p_2$. 
\end{itemize}
\end{prop}
\begin{proof} Part $a$ is immediate from Proposition \ref{transgeommap}$b$. For the proof of part $b$ observe that, by Proposition \ref{transgeommap}$c$, the isomorphism $\phi_q$ restricts to one between $\ker(\underline{\d J}_1)_{p_1}$ and $\ker(\underline{\d J}_2)_{p_2}$, so that we obtain an isomorphism of representations, compatible with part $a$, and given by:
\begin{equation}\label{indisosympnormrephammoreq} \S\No_{p_2}\to \S\No_{p_1},\quad [v]\mapsto [\d \beta_1(\hat{v})],
\end{equation} where $\hat{v}\in T_qQ$ is any vector such that $\d \beta_2(\hat{v})=v$ and $\d j(\hat{v})=0$. Note here (to see that such $\hat{v}$ exists) that, given $v\in \ker(\d J_2)_{p_2}$ and $\hat{w}\in T_qQ$ such that $\d \beta_2(\hat{w})=v$, we have $\d j(\hat{w})\in \ker(\d\alpha_2)$, hence by Lemma \ref{acttypeprop} there is a $\hat{u}\in \ker(\d\beta_2)_q$ such that $\d j(\hat{u})=\d j(\hat{w})$, so that $\hat{v}:=\hat{w}-\hat{u}$ has the desired properties. With this description of (\ref{indisosympnormrephammoreq}) it is immediate from (\ref{hameqmor1}) that $(\ref{indisosympnormrephammoreq})$ pulls $(\omega_1)_{p_1}$ back to $(\omega_2)_{p_2}$, which concludes the proof.
\end{proof}
We can now give a more conceptual proof of Lemma \ref{assrepprop}. 
\begin{proof}[Proof of Lemma \ref{assrepprop}] Apply Proposition \ref{transgeomham} to Example \ref{asrepmoreq}.
\end{proof}
For Hamiltonian Morita equivalences as in Example \ref{locmodmoreq} (where one of the two groupoid maps is not of action type) it is not clear to us whether there is a satisfactory generalization of Proposition \ref{transgeomham}. The arguments in the proof of that proposition do show the following, which will be enough for our purposes.  
\begin{prop}\label{finremhammoreq} Let a Hamiltonian Morita equivalence (denoted as in Definitions \ref{moreqdefLie} and \ref{moreqdefHam}) between groupoids maps $\J_1:\H_1\to (\G_1,\Omega_1)$ and $\J_2:\H_2\to (\G_2,\Omega_2)$ of Hamiltonian type be given. Suppose that $p_1\in S_1$ and $p_2\in S_2$ belong to $Q$-related orbits and let $q\in Q$ such that $\beta_1(q)=p_1$ and $\beta_2(q)=p_2$. Further, assume that $\J_1$ is of action type and the canonical injection:
\begin{equation*} \S\No_{p_2}:=\frac{\ker(\d J_2)_{p_2}}{\ker(\d J_2)_{p_2}\cap T_{p_2}\O}\hookrightarrow \ker(\underline{\d J}_2)_{p_2}
\end{equation*} is an isomorphism. The form $\omega_2$ on the base $S_2$ of $\H_2$ may be degenerate. Then:
\begin{itemize}\item[a)] the form $\omega_2$ descends to a linear symplectic form $(\omega_2)_{p_2}$ on $\S\No_{p_2}$, which is invariant under the $(\H_2)_{p_2}$-action defined by declaring the isomorphism with $\ker(\underline{\d J})_{p_2}$ to be equivariant,
\item[b)] there is a symplectic linear isomorphism $(\S\No_{p_1},\omega_{p_1})\cong (\S\No_{p_2},\omega_{p_2})$ that is compatible with the isomorphism of Lie groups $\Phi_q:(\H_1)_{p_1}\xrightarrow{\sim}(\H_2)_{p_2}$.
\end{itemize}
\end{prop}

\section{The canonical Hamiltonian stratification}
In this part we apply our normal form results to study stratifications on orbit spaces of Hamiltonian actions. To elaborate: in Section \ref{stratsec} we give background on Whitney stratifications of reduced differentiable spaces and we discuss the canonical Whitney stratification of the leaf space of a proper Lie groupoid. A novelty in our discussion is that we point out a criterion (Lemma \ref{whitreghomprop}) for a partition into submanifolds of a reduced differentiable space to be Whitney regular, which may be of independent interest. Furthermore, we give a similar criterion (Corollary \ref{fibwhitstratcor}) for the fibers of a map between reduced differentiable spaces to inherit a natural Whitney stratification from a constant rank partition of the map. In Section \ref{canhamstratsec} we introduce the canonical Hamiltonian stratification and prove Theorem \ref{canhamstratthm} and Theorem \ref{redspstratthm}, by verifying that the canonical Hamiltonian stratification of the orbit space and the Lerman-Sjamaar stratification of the symplectic reduced spaces meet the aforementioned criteria, using basic features of Hamiltonian Morita equivalence and the normal form theorem. In Section \ref{regpartsec} we study the regular (or principal) parts of these stratifications. There we will also consider the infinitesimal analogue of the canonical Hamiltonian stratification on $S$, because its regular part turns out to be better behaved. Section \ref{poisstratthmsec} concerns the Poisson structure on the orbit space. The main theorem of this section shows that the canonical Hamiltonian stratification is a constant rank Poisson stratification of the orbit space, and describes the symplectic leaves in terms of the fibers of the transverse momentum map. Finally, in Section \ref{sympintstratsec} we construct explicit proper integrations of the Poisson strata of the canonical Hamiltonian stratification. Section \ref{regpartsec} can be read independently of Section \ref{poisstratthmsec} and Section \ref{sympintstratsec}.

\subsection{Background on Whitney stratifications of reduced differentiable spaces}\label{stratsec} 

\subsubsection{Stratifications of topological spaces}\label{stratdefsec}
In this paper, by a stratification we mean the following.
\begin{defi}\label{topstratdefi} Let $X$ be a Hausdorff, second-countable and paracompact topological space. A \textbf{stratification} of $X$ is a locally finite partition $\S$ of $X$ into smooth manifolds (called \textbf{strata}), that is required to satisfy:
\begin{itemize} \item[i)] Each stratum $\Sigma\in \S$ is a connected and locally closed topological subspace of $X$.
\item[ii)] For each $\Sigma\in \S$, the closure $\overline{\Sigma}$ in $X$ is a union of $\Sigma$ and strata of strictly smaller dimension. 
\end{itemize} The second of these is called the \textbf{frontier condition}. A pair $(X,\S)$ is called a \textbf{stratified space}. By a map of stratified spaces $\phi:(X,\S_X)\to (Y,\S_Y)$ we mean a continuous map $\phi:X\to Y$ with the property that for each $\Sigma_X\in \S_X$:
\begin{itemize}\item[i)] There is a stratum $\Sigma_Y\in \S_Y$ such that $\phi(\Sigma_X)\subset \Sigma_Y$.
\item[ii)] The restriction $\phi:\Sigma_X\to \Sigma_Y$ is smooth.
\end{itemize}
\end{defi}
Due to the connectedness assumption on the strata, the frontier condition (a priori of a global nature) can be verified locally with the lemma below.
\begin{lemma}\label{frontcondloc} Let $X$ be a topological space and $\S$ a partition of $X$ into connected manifolds (equipped with the subspace topology). Then $\S$ satisfies the frontier condition if and only if for every $x\in X$ and every $\Sigma\in \S$ such that $x\in \overline{\Sigma}$ and $x\notin\Sigma$ the following hold:
\begin{itemize}
\item[i)] there is an open neighbourhood $U$ of $x$ such that $U\cap \Sigma_x\subset \overline{\Sigma}$,
\item[ii)] $\dim(\Sigma_x)<\dim(\Sigma)$.
\end{itemize}
\end{lemma}
\begin{rem}\label{compremstratdefi} Throughout, we will make reference to various texts that use slightly different definitions of stratifications. After restricting attention to Whitney stratifications (Definition \ref{whitstratdef}), the differences between these definitions become significantly smaller (also see Remark \ref{passtoconncomprem}). A comparison of Definition \ref{topstratdefi} with the notion of stratification in \cite{Mat,Pfl} can be found in \cite{CrMe}. 
\end{rem}
The constructions of the stratifications in this paper follow a general pattern: one first defines a partition $\P$ of $X$ into manifolds (possibly disconnected, with connected components of varying, but bounded, dimension) which in a local model for $X$ have a particularly simple description. This partition $\P$ is often natural to the given geometric situation from which $X$ arises. Then, one passes to the partition $\S:=\P^\textrm{c}$ consisting of the connected components of the members of $\P$, and verifies that $\S$ is a stratification of $X$. 
\begin{rem} When speaking of a manifold, we always mean that its connected components are of one and the same dimension, unless explicitly stated otherwise (such as above).
\end{rem}
\begin{ex}\label{exmortyp} The leaf space of a proper Lie groupoid admits a canonical stratification. To elaborate, let $\G\rightrightarrows M$ be a proper Lie groupoid, meaning that $\G$ is Hausdorff and the map $(t,s):\G\to M\times M$ is proper. This is equivalent to requiring that $\G$ is proper at every $x\in M$ (as in Definition \ref{propatxdefi}) and that its leaf space $\underline{M}$ is Hausdorff \cite[Proposition 5.1.3]{Hoy}. In fact, $\underline{M}$ is locally compact, second countable and Hausdorff (so, in particular it is paracompact). To define the stratifications of $M$ and $\underline{M}$, first consider the partition $\P_\mathcal{M}(M)$ of $M$ by \textbf{Morita types}. This is given by the equivalence relation: $x_1\sim_\mathcal{M} x_2$ if and only if there are invariant opens $V_1$ and $V_2$ around $\L_{x_1}$ and $\L_{x_2}$, respectively, together with a Morita equivalence: 
\begin{equation*} \G\vert_{V_1}\simeq \G\vert_{V_2},
\end{equation*} that relates $\L_{x_1}$ to $\L_{x_2}$. Its members are invariant and therefore descend to a partition $\P_\mathcal{M}(\underline{M})$ of the leaf space $\underline{M}$. The partitions $\S_\textrm{Gp}(M)$ and $\S_\textrm{Gp}(\underline{M})$ obtained from $\P_\mathcal{M}(M)$ and $\P_\mathcal{M}(\underline{M})$ after passing to connected components form the so-called \textbf{canonical stratifications} of the base $M$ and the leaf space $\underline{M}$ of the Lie groupoid $\G$. These indeed form stratifications. This is proved in \cite{PfPoTa} and \cite{CrMe}, using the local description given by the linearization theorem for proper Lie groupoids (see \cite{We4,Zu,CrStr,FeHo}). There, the partition by Morita types is defined by declaring that $x,y\in M$ belong to the same Morita type if and only if there is an isomorphism of Lie groups:
\begin{equation*} \G_x\cong \G_y
\end{equation*}
together with a compatible linear isomorphism:
\begin{equation*} \No_x\cong \No_y
\end{equation*} between the normal representations of $\G$ at $x$ and $y$, as in (\ref{normreppt}). This is equivalent to the description given before, as a consequence of Proposition \ref{transgeomgpoid}$b$ and the linearization theorem. 
\end{ex} 

Often there are various different partitions that, after passing to connected components, induce the same stratification. This too can be checked locally, using the following lemma.
\begin{lemma}\label{passconncomp} Let $\P_1$ and $\P_2$ be partitions of a topological space $X$ into manifolds (equipped with the subspace topology) with connected components of possibly varying dimension. Then the partitions $\P_1^c$ and $\P_2^c$, obtained after passing to connected components, coincide if and only if every $x\in X$ admits an open neighbourhood $U$ in $X$ such that 
\begin{equation*} P_1\cap U=P_2 \cap U,
\end{equation*} where $P_1$ and $P_2$ are the members of $\P_1$ and $\P_2$ through $x$. 
\end{lemma}
\begin{ex}\label{exisotyp} Given a proper Lie groupoid $\G\rightrightarrows M$, there is a coarser partition of $M$ (resp. $\underline{M}$) that yields the canonical stratification on $M$ (resp. $\underline{M}$) after passing to connected components: the partition by \textbf{isomorphism types}. On $M$, this partition is given by the equivalence relation: $x\cong y$ if and only if the isotropy groups $\G_x$ and $\G_y$ are isomorphic (as Lie groups). We denote this partition as $\P_{\cong}(M)$. Its members are invariant and therefore descend to a partition of $\P_{\cong}(\underline{M})$ of the leaf space $\underline{M}$. The fact that these indeed induce the canonical stratifications $\S_\textrm{Gp}(M)$ and $\S_\textrm{Gp}(\underline{M})$ follows from Lemma \ref{passconncomp} and the linearization theorem for proper Lie groupoids.
\end{ex}
\begin{ex}\label{exorbtyp} The canonical stratification on the orbit space of a proper Lie group action is usually defined using the partition by \textbf{orbit types}. To elaborate, let $M$ be a manifold, acted upon by a Lie group $G$ in a proper fashion. The partition $\P_\sim(M)$ by orbit types is defined by the equivalence relation: $x\sim y$ if and only if the isotropy groups $G_x$ and $G_y$ are conjugate subgroups of $G$. Its members are $G$-invariant, and hence this induces a partition $\P_\sim(\underline{M})$ of the orbit space $\underline{M}:=M/G$ as well. The partitions obtained from $\P_\sim(M)$ and $\P_\sim(\underline{M})$ after passing to connected components coincide with the canonical stratifications $\S_\textrm{Gp}(M)$ and $\S_\textrm{Gp}(\underline{M})$ of the action groupoid $G\ltimes M$ (as in Example \ref{exmortyp}). Another interesting partition that induces the canonical stratifications in this way is the partition by \textbf{local types}, defined by the equivalence relation: $x\cong y$ if and only if there is a $g\in G$ such that $G_x=gG_yg^{-1}$, together with a compatible linear isomorphism $\No_x\cong \No_y$ between the normal representations at $x$ and $y$. That these partitions induce the canonical stratifications follows from Lemma \ref{passconncomp} and the tube theorem for proper Lie group actions (see e.g. \cite{DuKo}). 
\end{ex}

\begin{rem} The discussion above is largely a recollection of parts of \cite{CrMe}. There the reader can find most details and proofs of the claims made in this subsection. A further discussion can be found in \cite{CrFeTo2}, where the canonical stratifications are studied in the context Poisson manifolds of compact types.   
\end{rem}
\subsubsection{Reduced differentiable spaces}\label{reddiffspsec} Further interesting properties of a stratified space can be defined when the space $X$ comes equipped with the structure of reduced differentiable space (a notion of smooth structure on $X$) and the stratification is compatible with this structure. We now recall what this means. Throughout, a sheaf will always mean a sheaf of $\R$-algebras. 
\begin{defi} A \textbf{reduced ringed space} is a pair $(X,\O_X)$ consisting of a topological space $X$ and a subsheaf $\O_X$ of the sheaf of continuous functions $\mathcal{C}_X$ on $X$ that contains all constant functions. We refer to $\O_X$ as the \textbf{structure sheaf}. A \textbf{morphism of reduced ringed spaces}: 
\begin{equation}\label{morphringsp} \phi:(X,\O_X)\to (Y,\O_Y)
\end{equation} is a continuous map $\phi:X\to Y$ with the property that for every open $U$ in $Y$ and every function $f\in \O_Y(U)$, it holds that $f\circ \phi\in \O_X(\phi^{-1}(U))$. Given such a morphism, we let
\begin{equation}\label{pullbackringmor1} \phi^*:\O_Y\to \phi_*\O_X
\end{equation} denote the induced map of sheaves over $Y$ and we use the same notation for the corresponding map of sheaves over $X$: \begin{equation}\label{pullbackringmor2} 
\phi^*:\phi^*\O_Y\to \O_X. 
\end{equation} 
\end{defi}
\begin{ex} Let $M$ be a smooth manifold and $\mathcal{C}^\infty_M$ its sheaf of smooth functions. Then $(M,\mathcal{C}^\infty_M)$ is a reduced ringed space. A map $M\to N$ between smooth manifolds is smooth precisely when it is a morphism of reduced ringed spaces $(M,\mathcal{C}^\infty_M)\to (N,\mathcal{C}^\infty_N)$. 
\end{ex}
\begin{ex}\label{smoothfunsubsp} Let $Y$ be a subspace of $\R^n$. We call a function defined on (an open in) $Y$ smooth if it extends to a smooth function on an open in $\R^n$. This gives rise to the sheaf of smooth functions $\mathcal{C}^\infty_Y$ on $Y$. 
\end{ex}
\begin{ex}\label{smoothfunleafsp} The leaf space $\underline{M}$ of a Lie groupoid $\G\rightrightarrows M$ is naturally a reduced ringed space, with structure sheaf $\mathcal{C}^\infty_{\underline{M}}$ given by:
\begin{equation*} \mathcal{C}^\infty_{\underline{M}}(\underline{U})=\{f\in \mathcal{C}_{\underline{M}}(\underline{U})\mid f\circ q\in \mathcal{C}^\infty_M(q^{-1}(\underline{U}))\},
\end{equation*} where $q:M\to \underline{M}$ denotes the projection onto the leaf space. We simply refer to this as the \textbf{sheaf of smooth functions on the leaf space}. Often we implicitly identify $\mathcal{C}^\infty_{\underline{M}}$ with the (push-forward of) the sheaf of $\G$-invariant smooth functions on $M$, via $q^*:\mathcal{C}_{\underline{M}}\to q_*\mathcal{C}_M$.
\end{ex}

\begin{defi}[\cite{GoSa}]\label{reddifsp} A \textbf{reduced differentiable space} is a reduced ringed space $(X,\O_X)$ with the property that for every $x\in X$ there is an open neighbourhood $U$, a locally closed subspace $Y$ of $\R^n$ (where $n$ may depend on $x$) and a homeomorphism $\chi:U\to Y$ that induces an isomorphism of reduced ringed spaces:
\begin{equation*} (U,\O_X\vert_U)\cong (Y,\mathcal{C}^\infty_Y).
\end{equation*} We call such a homeomorphism $\chi$ a \textbf{chart} of the reduced differentiable space. A \textbf{morphism of reduced differentiable spaces} is simply a morphism of the underlying reduced ringed spaces. 
\end{defi}
\begin{ex} A reduced differentiable space $(X,\O_X)$ is a $n$-dimensional smooth manifold if and only if around every $x\in X$ there is a chart for $(X,\O_X)$ that maps onto an open in $\R^n$.
\end{ex}
\begin{ex}\label{leafspreddiffspex} The leaf space $\underline{M}$ of a proper Lie groupoid $\G\rightrightarrows M$, equipped with the structure sheaf of Example \ref{smoothfunleafsp}, is a reduced differentiable space.  The proof of this will be recalled at the end of this subsection.
\end{ex}
\begin{rem}\label{partofunityreddiffbsp} A reduced differentiable space $(X,\O_X)$ is locally compact. So, if it is Hausdorff and second countable, then it is also paracompact. Moreover, it then admits $\O_X$-partitions of unity subordinate to any open cover (this can be proved as for manifolds, see e.g. \cite{GoSa}). 
\end{rem} 
To say what it means for a stratification to be compatible with the structure of reduced differentiable space, we will need an appropriate notion of submanifold. 

\begin{defi}\label{embredringdef} Let $(Y,\O_Y)$ and $(X,\O_X)$ be reduced ringed spaces and $i:Y\hookrightarrow X$ a topological embedding. We call $i$ an \textbf{embedding of reduced ringed spaces} if it is a morphism of reduced ringed spaces and $i^*:\O_X\vert_Y\to \O_Y$ is a surjective map of sheaves. In other words, $\O_Y$ coincides with the image sheaf of the map $i^*:\O_X\vert_Y\to \mathcal{C}_Y$, meaning that for every open $U$ in $Y$:
\begin{equation*} \O_Y(U)=\{f\in \mathcal{C}_Y(U)\mid \forall y\in U,\text{ }\exists (\widehat{f})_{i(y)}\in (\O_X)_{i(y)} : (f)_y=(\widehat{f}\vert_Y)_y\}.
\end{equation*} 
\end{defi}
\begin{rem}\label{indfuncstr} Let us stress that for any subspace $Y$ of a reduced ringed space $(X,\O_X)$ there is a unique subsheaf $\O_Y\subset \mathcal{C}_Y$ making $i:(Y,\O_Y)\hookrightarrow (X,\O_X)$ into an embedding of reduced ringed spaces. We will call this the \textbf{induced structure sheaf} on $Y$. Note that, if $(X,\O_X)$ is a reduced differentiable space and $Y$ is locally closed in $X$, then $Y$, equipped with its induced structure sheaf, is a reduced differentiable space as well, because charts for $X$ restrict to charts for $Y$.
\end{rem}
\begin{ex}\label{embredringspex} Here are some examples of embeddings:
\begin{itemize} \item[i)] For maps between smooth manifolds, the above notion of embedding is the usual one.
\item[ii)] In Example \ref{smoothfunsubsp}, the inclusion $i:(Y,\mathcal{C}^\infty_Y)\hookrightarrow (\R^n,\mathcal{C}^\infty_{\R^n})$ is an embedding. 
\item[iii)] Let $(X,\O_X)$ be a reduced ringed space and $U\subset X$ open. A homeomorphism $\chi:U\to Y$ onto a locally closed subspace $Y$ of $\R^n$ is a chart if and only if $\chi:(U,\O_X\vert_U)\to (\R^n,\mathcal{C}^\infty_{\R^n})$ is an embedding. 
\end{itemize}
\end{ex}
\begin{rem}\label{morphsmstrspsimp} Let $\phi:(X_1,\O_{X_1})\to (X_2,\O_{X_2})$ be a morphism of reduced ringed spaces and let $Y_1\subset X_1$ and $Y_2\subset X_2$ be subspaces such that $\phi(Y_1)\subset Y_2$. Then $\phi$ restricts to a morphism of reduced ringed spaces $(Y_1,\O_{Y_1})\to (Y_2,\O_{Y_2})$ with respect to the induced structure sheaves. 
 \end{rem}
\begin{defi} Let $(X,\O_X)$ be a reduced differentiable space and $Y$ a locally closed subspace of $X$. We call $Y$ a \textbf{submanifold} of $(X,\O_X)$, when endowed with its induced structure sheaf it is a smooth manifold. 
\end{defi} 
\begin{rem} Let $(X,\O_X)$ be a reduced differentiable space. Let $Y$ be a subspace of $X$. Then $Y$ is a $d$-dimensional submanifold of $(X,\O_X)$ if and only if for every chart $(U,\chi)$ of $(X,\O_X)$ the image $\chi(U\cap Y)$ is a $d$-dimensional submanifold of $\R^n$. 
\end{rem}
\begin{ex}\label{canstratsubmanex} Let $\G\rightrightarrows M$ be a proper Lie groupoid. Each Morita type in $\underline{M}$ is a submanifold of the leaf space $(\underline{M},\mathcal{C}^\infty_{\underline{M}})$. The same holds for each stratum of the canonical stratification. 
\end{ex}
We end this subsection by recalling proofs of the claims in Example \ref{leafspreddiffspex} and \ref{canstratsubmanex}. The following observation will be useful for this and for later reference.
\begin{prop}\label{globchar} Let $(Y,\O_Y)$ be a reduced ringed space, $(X,\O_X)$ a Hausdorff and second countable reduced differentiable space. Suppose that $i:Y\hookrightarrow X$ is both a topological embedding and a morphism of reduced ringed spaces. Then $i$ is an embedding of reduced ringed spaces if and only if every global function $f\in \O_Y(Y)$ extends to a function $g\in \O_X(U)$ defined on some open neighbourhood $U$ of $i(Y)$ in $X$. Moreover, if $i(Y)$ is closed in $X$, then $U$ can be chosen to be $X$.
\end{prop}
\begin{proof} For the forward implication, let $f\in \O_Y(Y)$. Since $i$ is an embedding of reduced ringed spaces, for every $y\in Y$ there is a local extension of $f$, defined on an open around $i(y)$ in $X$. By Remark \ref{partofunityreddiffbsp}, any open in $X$ admits $\O_X$-partitions of unity subordinate to any open cover. So, using the standard partition of unity argument we can construct, out of the local extensions, an extension $g\in \O_X(U)$ of $f$ defined on an open neighbourhood $U$ of $i(Y)$ in $X$, which can be taken to be all of $X$ if $i(Y)$ is closed in $X$. For the backward implication, it suffices to show that every germ in $\O_Y$ can be represented by a globally defined function in $\O_Y(Y)$. For this, it is enough to show that for every $y\in Y$ and every open neighbourhood $U$ of $y$ in $Y$, there is a function $\rho\in \O_Y(Y)$, supported in $U$, such that $\rho=1$ on an open neighbourhood of $y$ in $U$. To verify the latter, let $y$ and $U$ be as above. Let $V$ be an open in $X$ around $i(y)$ such that $V\cap i(Y)=i(U)$. Using a chart for $(X,\O_X)$ around $i(y)$, we can find a function $\rho_X \in \O_X(X)$, supported in $V$, such that $\rho_X=1$ on an open neighbourhood of $i(y)$ in $V$. Now, $\rho:=i^*(\rho_X)\in \O_Y(Y)$ is supported in $U$ and equal to $1$ on an open neighbourhood of $y$ in $U$. This proves the proposition.
\end{proof}
Returning to Example \ref{leafspreddiffspex}: first consider the case of a compact Lie group $G$ acting linearly on a real finite-dimensional vector space $V$ (that is, $V$ is a representation of $G$). The algebra $P(V)^G$ of $G$-invariant polynomials on $V$ is finitely generated. Given a finite set of generators $\{\rho_1,...,\rho_n\}$ of $P(V)^G$, one can consider the polynomial map:
\begin{equation}\label{orbmaprep} \rho=(\rho_1,...,\rho_n):V\to \R^n.
\end{equation} We call this a \textbf{Hilbert map} for the representation $V$. Any such map factors through an embedding of topological spaces $\underline{\rho}:V/G\to \R^n$ onto a closed subset of $\R^n$. Furthermore:
\begin{thm}[\cite{Schw1}] Let $G$ be a compact Lie group, $V$ a real finite-dimensional representation of $G$ and $\rho:V\to \R^n$ a Hilbert map. Then the associated map (\ref{orbmaprep}) satisfies:
\begin{equation*} \rho^*(C^\infty(\R^n))=C^\infty(V)^G.
\end{equation*}
\end{thm}
So, in view of Proposition \ref{globchar}, the morphism of reduced ringed spaces:
\begin{equation}\label{schwarzchart} \underline{\rho}:(V/G,\mathcal{C}^\infty_{V/G})\to (\R^n,\mathcal{C}^\infty_{\R^n}).
\end{equation}
is in fact an embedding of reduced ringed spaces (Definition \ref{embredringdef}), and hence a globally defined chart for the orbit space $V/G$ (by Example \ref{embredringspex}). Next, we show how this leads to charts for the leaf space of a proper Lie groupoid. Recall:
\begin{prop}\label{moreqisoredring} The homeomorphism of leaf spaces (\ref{leafsphommoreq}) induced by a Morita equivalence of Lie groupoids is an isomorphism of reduced ringed spaces:\begin{equation*}h_P:(\underline{M}_1,\mathcal{C}^\infty_{\underline{M}_1})\xrightarrow{\sim} (\underline{M}_2,\mathcal{C}^\infty_{\underline{M}_2}).
\end{equation*}
\end{prop}
\begin{proof} Suppose we are given a Morita equivalence between Lie groupoids:
\begin{center}
\begin{tikzpicture} \node (G1) at (0,0) {$\G_1$};
\node (M1) at (0,-1.3) {$M_1$};
\node (S) at (1.4,0) {$P$};
\node (M2) at (2.7,-1.3) {$M_2$};
\node (G2) at (2.7,0) {$\G_2$};
 
\draw[->,transform canvas={xshift=-\shift}](G1) to node[midway,left] {}(M1);
\draw[->,transform canvas={xshift=\shift}](G1) to node[midway,right] {}(M1);
\draw[->,transform canvas={xshift=-\shift}](G2) to node[midway,left] {}(M2);
\draw[->,transform canvas={xshift=\shift}](G2) to node[midway,right] {}(M2);
\draw[->](S) to node[pos=0.25, below] {$\text{ }\text{ }\alpha_1$} (M1);
\draw[->] (0.8,-0.15) arc (315:30:0.25cm);
\draw[<-] (1.9,0.15) arc (145:-145:0.25cm);
\draw[->](S) to node[pos=0.25, below] {$\alpha_2$\text{ }} (M2);
\end{tikzpicture}
\end{center}
Then, given two $P$-related invariant opens $U_1\subset M_1$ and $U_2\subset M_2$, we have algebra isomorphisms:
\begin{center}
\begin{tikzpicture} \node (S_1) at (0,0) {$\mathcal{C}_{M_1}^\infty(U_1)^{\G_1}$};
\node (S_2) at (8,0) {$\mathcal{C}_{M_2}^\infty(U_2)^{\G_2}$};
\node (Q) at (4,1) {$\mathcal{C}_P^\infty(\alpha_1^{-1}(U_1))^{\G_1}\cap \mathcal{C}_P^\infty(\alpha_2^{-1}(U_2))^{\G_2}$};
\node (X_1) at (0,-2) {$\mathcal{C}_{\underline{M}_1}^\infty(\underline{U}_1)$};
\node (X_2) at (8,-2) {$\mathcal{C}_{\underline{M}_2}^\infty(\underline{U}_2)$};

\draw[->](S_1) to node[pos=0.45, below] {$\text{ }\text{ }\alpha_1^*$} (Q);
\draw[->](S_2) to node[pos=0.45, below] {$\text{ }\text{ }\alpha_2^*$} (Q);
\draw[->](X_1) to node[pos=0.55, right] {$q_1^*\text{ }\text{ }$} (S_1);
\draw[->](X_2) to node[pos=0.55, left] {$\text{ }\text{ }q_2^*$} (S_2);
\draw[<-, dashed](X_1) to node[pos=0.45, below] {$\text{ }\text{ }h_P^*$} (X_2);

\end{tikzpicture}
\end{center} that complete to a commutative diagram via $h_P^*:\mathcal{C}_{\underline{M}_2}\to (h_P)_*\mathcal{C}_{\underline{M}_1}$. 
\end{proof}

Now, the linearization theorem for proper Lie groupoids implies that, given a proper Lie groupoid $\G\rightrightarrows M$ and an $x\in M$, there is an invariant open neighbourhood $U$ of $x$ in $M$ and a Morita equivalence between $\G\vert_U$ and the action groupoid $\G_x\ltimes \No_x$ of the normal representation at $x$, as in (\ref{normreppt}), that relates $\L_x$ to the origin in $\No_x$.
So, applying Proposition \ref{moreqisoredring} we find an isomorphism: 
\begin{equation}\label{locmodliegpoidchart} (\underline{U},\mathcal{C}^\infty_{\underline{M}}\vert_{\underline{U}})\cong (\No_x/\G_x,\mathcal{C}^\infty_{\No_x/\G_x}),
\end{equation} which composes with the embedding (\ref{schwarzchart}) to a chart for $(\underline{M},\mathcal{C}^\infty_{\underline{M}})$, as desired. We conclude that $(X,\mathcal{C}^\infty_X)$ is a reduced differentiable space, as claimed in Example \ref{leafspreddiffspex}. To see why the claims in Example \ref{canstratsubmanex} hold true, let $\underline{\Sigma}\in \P_\mathcal{M}(\underline{M})$ be a Morita type. Suppose that $\L_x\in \underline{\Sigma}$. The isomorphism (\ref{locmodliegpoidchart}) identifies $\underline{U}\cap \underline{\Sigma}$ with the Morita type of $\G_x\ltimes \No_x$ through the origin, which is the fixed point set $\No_x^{\G_x}$ \textemdash a submanifold of $\No_x/\G_x$. Therefore $\underline{\Sigma}$ is a submanifold of $\underline{M}$ near $\L_x$. This being true for all points in $\underline{\Sigma}$, it follows that $\underline{\Sigma}$ is a submanifold with connected components of possibly varying dimension. The dimension of the connected component through $\L_x$ is $\dim(\No_x^{\G_x})$, hence it follows from Proposition \ref{transgeomgpoid}$b$ that all connected components of $\underline{\Sigma}$ in fact have the same dimension. So, the Morita types are indeed submanifolds of the leaf space, and so are their connected components.
 \subsubsection{Whitney stratifications of reduced differentiable spaces}\label{whitneysec} 
\begin{defi}\label{smoothstratspdefi} Let $(X,\O_X)$ be a Hausdorff and second countable reduced differentiable space. A \textbf{stratification} $\S$ of $(X,\O_X)$ is a stratification of $X$ by submanifolds of $(X,\O_X)$. That is, $\S$ is a stratification of $X$ with the property that the given smooth structure on each stratum coincides with its induced structure sheaf. We call the triple $(X,\O_X,\S)$ a \textbf{smooth stratified space}. A \textbf{morphism of smooth stratified spaces} is a morphism of the underlying stratified spaces that is simultaneously a morphism of the underlying reduced ringed spaces. 
\end{defi}
\begin{rem} As noted in \cite{PfPoTa}, the notion of smooth stratified space is equivalent (up to the slight difference pointed out in Remark \ref{compremstratdefi}) to the notion of stratified space with smooth structure in \cite{Pfl}, which is defined starting from an atlas of compatible singular charts, rather than a structure sheaf. 
\end{rem}
On stratifications of reduced differentiable spaces, we can impose an important extra regularity condition: Whitney's condition (b). We now recall this, starting with: 
\begin{defi} Let $R$ and $S$ be disjoint submanifolds of $\R^n$, and let $y\in S$. Then $R$ is called \textbf{Whitney regular} over $S$ at $y$ if the following is satisfied. For any two sequences $(x_n)$ in $R$ and $(y_n)$ in $S$ that both converge to $y$ and satisfy:
\begin{itemize}\item[i)] $T_{x_n}R$ converges to some $\tau$ in the Grassmannian of $\dim(R)$-dimensional subspaces of $\R^n$,
\item[ii)] the sequence of lines $[x_n-y_n]$ in $\R P^{n-1}$ converges to some line $\ell$,
\end{itemize} it must hold that $\ell \subset \tau$. 
\end{defi} 
Using charts, this generalizes to reduced differentiable spaces, as follows. 
\begin{defi} Let $(X,\O_X)$ be a reduced differentiable space and let $R$ and $S$ be disjoint submanifolds. Then $R$ is called \textbf{Whitney regular} over $S$ at $y\in S$ if for every chart $(U,\chi)$ around $y$, the submanifold $\chi(R\cap U)$ of $\R^n$ is Whitney regular over $\chi(S\cap U)$ at $\chi(y)$. We call $R$ Whitney regular over $S$ if it is so at every $y\in S$. Moreover, we call a partition $\P$ of $(X,\O_X)$ into submanifolds Whitney regular if every member of $\P$ is Whitney regular over each other member. 
\end{defi} 
\begin{defi}\label{whitstratdef} A smooth stratified space $(X,\O_X,\S)$ is called a \textbf{Whitney stratified space} when the partition $\S$ of $(X,\O_X)$ is Whitney regular. 
\end{defi}
To verify Whitney regularity of $R$ over $S$ at $y$, it is enough to do so in a single chart around $y$. To see this, the key remark is the proposition below, combined with the fact that  Whitney regularity is invariant under smooth local coordinate changes of the ambient space $\R^n$.
\begin{prop} Let $(X,\O_X)$ be a reduced differentiable space. Any two charts $(U_1,\chi_1)$ and $(U_2,\chi_2)$ onto locally closed subsets of $\R^n$ are smoothly compatible, in the sense that: for any $y\in U_1\cap U_2$, there is a diffeomorphism $H:O_1\to O_2$ from an open neighbourhood $O_1$ of $\chi_1(y)$ in $\R^n$ onto an open neighbourhood $O_2$ of $\chi_2(y)$ in $\R^n$ such that:
\begin{equation*} H\vert_{O_1\cap\chi_1(U_1\cap U_2)}=\chi_2\circ (\chi_1^{-1})\vert_{O_1\cap\chi_1(U_1\cap U_2)}.
\end{equation*}
\end{prop}
\begin{proof} Although this is surely known, we could not find a proof in the literature. The argument here is closely inspired by that of \cite[Proposition 1.3.10]{Pfl}. Turning to the proof: it is enough to show that, given two subspaces $Y_1,Y_2\subset \R^n$ and an isomorphism of reduced ringed spaces:
\begin{equation*} \phi:(Y_1,\mathcal{C}^\infty_{Y_1})\xrightarrow{\sim}(Y_2,\mathcal{C}^\infty_{Y_2}),
\end{equation*} there are, for every $y\in Y_1$, an open $U_1$ in $\R^n$ around $y$ and a smooth open embedding $\widehat{\phi}:U_1\to \R^n$ such that $\widehat{\phi}\vert_{U_1\cap Y_1}=\phi\vert_{U_1\cap Y_1}$. To this end, let us first make a general remark. Given $Y\subset \R^n$ and $y\in Y$, let $\mathfrak{m}^Y_y$ and $\mathfrak{m}^{\R^n}_y$ denote the respective maximal ideals in the stalks $(\mathcal{C}^\infty_Y)_y$ and $(\mathcal{C}^\infty_{\R^n})_y$, consisting of germs of those functions that vanish at $y$. Further, let $(\mathcal{I}_Y)_y$ denote the ideal in $(\mathcal{C}^\infty_{\R^n})_y$ consisting of germs of those functions that vanish on $Y$. Notice that we have a canonical short exact sequence:
\begin{equation*} 0\to \left((\mathcal{I}_Y)_y+(\mathfrak{m}_y^{\R^n})^2\right)/(\mathfrak{m}_y^{\R^n})^2 \to \mathfrak{m}_y^{\R^n}/(\mathfrak{m}_y^{\R^n})^2 \xrightarrow{(i_Y)_y^*} \mathfrak{m}_y^{Y}/(\mathfrak{m}_y^{Y})^2\to 0.
\end{equation*} Furthermore, recall that there is a canonical isomorphism of vector spaces: 
\begin{equation*} \mathfrak{m}_y^{\R^n}/(\mathfrak{m}_y^{\R^n})^2\xrightarrow{\sim} T^*_y\R^n,\quad (f)_y\mod (\mathfrak{m}_y^{\R^n})^2\mapsto \d f_y.
\end{equation*} It follows that, for any $(h_1)_y,...,(h_k)_y\in \mathfrak{m}^{\R^n}_y$ that project to a basis of $\mathfrak{m}_y^{Y}/(\mathfrak{m}_y^{Y})^2$, we can find $(h_{k+1})_y,...,(h_n)_y\in (\mathcal{I}_Y)_y$ such that $\d (h_1)_y,...,\d (h_n)_y\in T^*_y\R^n$ form a basis, or in other words, such that $(h_1,...,h_n)_y$ is the germ of a diffeomorphism from an open neighbourhood of $y$ in $\R^n$ onto an open neighbourhood of the origin in $\R^n$. Now, we return to the isomorphism $\phi$. Let $k$ be the dimension of $\mathfrak{m}_{\phi(y)}^{Y_2}/(\mathfrak{m}_{\phi(y)}^{Y_2})^2$. Using the above remark we can, first of all, find a diffeomorphism: 
\begin{equation*} f=(f_1,...,f_n):U_2\xrightarrow{\sim} V_2
\end{equation*} from an open $U_2$ in $\R^n$ around $\phi(y)$ onto an open $V_2$ in $\R^n$ around the origin, such that: 
\begin{equation*} (f_1)_{\phi(y)},...,(f_k)_{\phi(y)}\in \m_{\phi(y)}^{\R^n}
\end{equation*} project to a basis of $\mathfrak{m}_{\phi(y)}^{Y_2}/(\mathfrak{m}_{\phi(y)}^{Y_2})^2$ and such that $f_{k+1},...,f_n$ vanish on $U_2\cap Y_2$. Since $\phi$ is an isomorphism of reduced ringed spaces, it induces an isomorphism:
\begin{equation*} (\phi^*)_y: \mathfrak{m}_{\phi(y)}^{Y_2}/(\mathfrak{m}_{\phi(y)}^{Y_2})^2 \xrightarrow{\sim} \mathfrak{m}_{y}^{Y_1}/(\mathfrak{m}_{y}^{Y_1})^2,
\end{equation*} which maps the above basis to a basis of $\mathfrak{m}_{y}^{Y_1}/(\mathfrak{m}_{y}^{Y_1})^2$. Using this and the remark above once more, we can find a diffeomorphism:
\begin{equation*} g=(g_1,...,g_n):U_1\xrightarrow{\sim} V_1,
\end{equation*} from an open $U_1$ in $\R^n$ around $y$ such that $\phi(U_1\cap Y_1)\subset U_2$, onto an open $V_1\subset V_2$ around the origin in $\R^n$, with the property that:
\begin{equation*} g_j\vert_{U_1\cap Y_1}=f_j\circ (\phi\vert_{U_1\cap Y_1}), \quad \forall j=1,...,k,
\end{equation*} and that $g_{k+1},...,g_n$ vanish on $U_1\cap Y_1$. Then, in fact $g\vert_{U_1\cap Y_1}=f\circ (\phi\vert_{U_1\cap Y_1})$, so that the smooth open embedding:
\begin{equation*} \widehat{\phi}:=f^{-1}\circ g:U_1\to \R^n,
\end{equation*} restricts to $\phi$ on $U_1\cap Y_1$, as desired. 
\end{proof}

\begin{rem}\label{passtoconncomprem} Contuining Remark \ref{compremstratdefi}: 
\begin{itemize} \item[i)] Let $(X,\O_X)$ be a Hausdorff and second countable reduced differentiable space and let $\P$ be a locally finite partition of $(X,\O_X)$ into submanifolds. In the terminology of \cite{GWPL}, such a partition $\P$ would be called a stratification. If $\P$ is Whitney regular, then the partition $\P^c$ (obtained after passing to connected components) is locally finite and satisfies the frontier condition. Hence, $\P^c$ is then a Whitney stratification of $(X,\O_X)$. In the case that $(X,\O_X)$ is a locally closed subspace of $\R^n$ equipped with its induced structure sheaf, this statement is proved in \cite{GWPL} using the techniques developed in \cite{Th,Mat,Mat1}. The general statement follows from this case by using charts and Lemma \ref{frontcondloc}. 
\item[ii)] Combined with the discussion in \cite[Section 4.1]{CrMe} and \cite[Proposition 1.2.7]{Pfl}, the previous remark shows that the notion of Whitney stratified space used here is actually equivalent to that in \cite{Pfl}. 
\end{itemize}
\end{rem}

\subsubsection{Semi-algebraic sets and homogeneity} 
For proofs of the facts on semi-algebraic sets that we use throughout, we refer to \cite{BoCoRo}; further see \cite{GWPL} for a concise introduction. By a semi-algebraic subset of $\R^n$, we mean a finite union of subsets defined by real polynomial equalities and inequalities. Semi-algebraic sets are rather rigid geometric objects. For instance, any semi-algebraic set $A\subset \R^n$ has a finite number of connected components and admits a canonical Whitney stratification with finitely many strata (in contrast: any closed subset of $\R^n$ is the zero-set of some smooth function). As remarked in \cite{GWPL}, there is a useful criterion for stratifications in $\R^n$ to be Whitney regular, when the strata are semi-algebraic. This criterion can be extended to smooth stratified spaces, as follows. 
\begin{defi}\label{locsemalgdefi} We call a partition $\P$ of a reduced differentiable space $(X,\O_X)$ \textbf{locally semi-algebraic} at $x\in X$ if there is a chart $(U,\chi)$ around $x$ that maps every member of $\P\vert_U$ onto a semi-algebraic subset of $\R^n$. We call the partition locally semi-algebraic if it is so at every $x\in X$. 
\end{defi}
\begin{defi}\label{homdefi} We call a partition $\P$ of a topological space $X$ \textbf{homogeneous} if for any two $x_1,x_2\in X$ that belong to the same member of $\P$, there is a homeomorphism: 
\begin{equation*} h:U_1\xrightarrow{\sim} U_2
\end{equation*} from an open $U_1$ around $x_1$ onto an open $U_2$ around $x_2$ in $X$, with the property that $h(x_1)=x_2$ and for every $\Sigma\in \P$: 
\begin{equation*} h(U_1\cap \Sigma)=U_2\cap \Sigma.
\end{equation*} If $(X,\O_X)$ is a reduced differentiable space and the members of $\P$ are submanifolds, then we call $\P$ \textbf{smoothly homogeneous} if the homeomorphisms $h$ can in fact be chosen to be isomorphisms of reduced differentiable spaces:
\begin{equation*} h:(U_1,\O_X\vert_{U_1})\xrightarrow{\sim} (U_2,\O_X\vert_{U_2}).
\end{equation*}
\end{defi}

\begin{rem} Notice that: \begin{itemize}\item[i)] Homogeneity of a partition $\P$ of $X$ implies that $\P$ satisfies the topological part of the frontier condition: the closure of any member $\Sigma\in\P$ is a union of $\Sigma$ with other members.
\item[ii)] If $\P$ is smoothly homogeneous, then a map $h$ as above restricts to diffeomorphisms between the members of $\P\vert_{U_1}$ and $\P\vert_{U_2}$ (by Remark \ref{morphsmstrspsimp}). 
\end{itemize}
\end{rem}
Together the above conditions give a criterion for Whitney regularity.
\begin{lemma}\label{whitreghomprop} Let $(X,\O_X)$ be a reduced differentiable space and let $\P$ be a partition of $(X,\O_X)$ into submanifolds. If $\P$ is smoothly homogeneous and locally semi-algebraic, then it is Whitney regular.  
\end{lemma}
\begin{proof} Let $R,S\in \P$ be two distinct members. Since $\P$ is smoothly homogeneous, either $R$ is Whitney regular over $S$ at all points in $S$, or at no points at all. Indeed, this follows from the simple fact that Whitney regularity is invariant under isomorphisms of reduced differentiable spaces. As $\P$ is locally semi-algebraic, the latter option cannot happen, and hence the partition must be Whitney regular. In order to explain this, suppose first that $R,S\subset \R^n$ are semi-algebraic and submanifolds of $\R^n$ (also called Nash submanifolds of $\R^n$). Consider the set of bad points: \begin{equation*} \mathcal{B}(R,S),
\end{equation*} which consists of those $y\in S$ at which $R$ is not Whitney regular over $S$. The key fact is now that, because $R$ and $S$ are semi-algebraic, the subset $\mathcal{B}(R,S)$ has empty interior in $S$ (see \cite{Wall} for a concise proof), hence it cannot be all of $S$. In general, we can pass to a chart around any $y\in S$ in which the strata $R$ and $S$ are semi-algebraic and the same argument applies, because Whitney regularity can be verified in a single chart. \end{proof}
To exemplify the use of Lemma \ref{whitreghomprop}, let us point out how it leads to a concise proof of:
\begin{thm}[\cite{PfPoTa}]\label{stratleafspthm} The canonical stratification of the leaf space of a proper Lie groupoid is a Whitney stratification.
\end{thm}
To verify the criteria of Lemma \ref{whitreghomprop}, we use:
\begin{prop}[\cite{Bi}]\label{semialgstratrep} Let $G$ be a compact Lie group and let $V$ be a real finite-dimensional representation of $G$. Then any Hilbert map $\rho:V\to \R^n$ (see Subsection \ref{reddiffspsec}) identifies the strata of the canonical stratification $\S_\textrm{Gp}(V/G)$ with semi-algebraic subsets of $\R^n$. 
\end{prop}
See also \cite[Theorem 1.5.2]{Schw2} for a more elementary proof. 
\begin{proof}[Proof of Theorem \ref{stratleafspthm}] Let $\G\rightrightarrows M$ be a proper Lie groupoid. We return to the discussion at the end of Subsection \ref{reddiffspsec}. As recalled there, for any $x\in M$ there is an open $\underline{U}$ around the leaf $\L_x\in \underline{M}$ and an isomorphism (\ref{locmodliegpoidchart}) that identifies $\underline{U}$, as a reduced differentiable space, with $\No_x/\G_x$. Furthermore, (\ref{locmodliegpoidchart}) identifies the partition $\P_\mathcal{M}(\underline{M})\vert_{\underline{U}}$ by Morita types of $\G\vert_U$ with the partition of $\No_x/\G_x$ by Morita types of $\G_x\ltimes \No_x$. Recall that the canonical stratification on the orbit space of a real, finite-dimensional representation of a compact Lie group has finitely many strata (see e.g. \cite[Proposition 2.7.1]{DuKo}). In combination with Proposition \ref{semialgstratrep}, this implies that a Hilbert map $\rho:\No_x\to \R^n$ for the normal representation $\No_x$ maps the Morita types in $\No_x/\G_x$ onto semi-algebraic subsets of $\R^n$. This shows that $\P_\mathcal{M}(\underline{M})$ is locally semi-algebraic. Secondly, Proposition \ref{moreqisoredring} implies that the partition by Morita types is homogeneous. To see this, note that by the very definition of the partition by Morita types on $\underline{M}$, for any two leaves $\L_1$ and $\L_2$ in the same Morita type, there are invariant opens $V_1$ around $\L_1$, $V_2$ around $\L_2$ in $M$ and a Morita equivalence $\G\vert_{V_1}\simeq \G\vert_{V_2}$ relating $\L_1$ to $\L_2$. The homeomorphism of leaf spaces induced by this Morita equivalence is an isomorphism of reduced differentiable spaces:
\begin{equation*} (\underline{V}_1,\mathcal{C}^\infty_{\underline{M}} \vert_{\underline{V}_1})\cong (\underline{V}_2,\mathcal{C}^\infty_{\underline{M}} \vert_{\underline{V}_2})
\end{equation*} that identifies $\L_1$ with $\L_2$ and $\underline{V}_1\cap \underline{\Sigma}$ with $\underline{V}_2\cap \underline{\Sigma}$ for every Morita type $\underline{\Sigma}$. So, the partition by Morita types is indeed smoothly homogeneous. In light of Lemma \ref{whitreghomprop}, it follows that the partition by Morita types is Whitney regular. Hence, passing to connected components, we find that $\S_\textrm{Gp}(\underline{M})$ is a Whitney stratification of the leaf space $(\underline{M},\mathcal{C}^\infty_{\underline{M}})$ (as in Remark \ref{passtoconncomprem}). 
\end{proof}

\subsubsection{Constant rank stratifications of maps}\label{constrkstratsec} Finally, we turn to constant rank stratifications of maps between reduced differentiable spaces. In this subsection, let $(X,\O_X)$ and $(Y,\O_Y)$ be Hausdorff, second countable reduced differentiable spaces. 
\begin{defi}\label{constrkstratdefi} By a \textbf{partition of a morphism $f:(X,\O_X)\to (Y,\O_Y)$ into submanifolds} we mean a pair $(\P_X,\P_Y)$ consisting of a partition $\P_X$ of $(X,\O_X)$ and a partition $\P_Y$ of $(Y,\O_Y)$ into submanifolds, such that $f$ maps every member of $\P_X$ into a member of $\P_Y$. We call this a \textbf{constant rank partition} of $f$ if in addition, for every $\Sigma_X\in \P_X$ and $\Sigma_Y\in \P_Y$ such that $f(\Sigma_X)\subset \Sigma_Y$, the smooth map $f:\Sigma_X\to \Sigma_Y$ has constant rank. Furthermore, by a \textbf{constant rank stratification} of $f$ we mean a constant rank partition for which both partitions are stratifications. 
\end{defi}
In the remainder of this subsection we focus on the partition induced on the fibers of a morphism $f:(X,\O_X)\to (Y,\O_Y)$ by a constant rank partition. The fibers of such a morphism are the reduced differentiable spaces $(f^{-1}(y),\O_{f^{-1}(y)})$, equipped with the induced structure sheaf as in Remark \ref{indfuncstr}. Given a constant rank partition $(\P_X,\P_Y)$ of $f$, its fibers have an induced partition: 
\begin{equation}\label{indpartfibconstrk} \P_X\vert_{f^{-1}(y)}=\{ \Sigma_X\cap f^{-1}(y) \mid \Sigma_X\in \P_X\},
\end{equation} the members of which are submanifolds, being the fibers of the constant rank maps obtained by restricting $f$ to the members of $(\P_X,\P_Y)$. The example below shows that the connected components of the members of (\ref{indpartfibconstrk}) need not form a stratification, even if $\P_X^c$ and $\P_Y^c$ are Whitney stratifications.
\begin{ex} Consider the polynomial map: 
\begin{equation*} f:\R^3\to \R,\quad f(x,y,z)=x^2-zy^2.
\end{equation*} The fiber of $f$ over the origin in $\R$ is the Whitney umbrella. Consider the stratification of $\R^3$ by the five strata $\{y<0\}$, $\{y>0\}$, $\{y=0,x<0\}$, $\{y=0,x>0\}$ and the $z$-axis $\{x=y=0\}$. Together with the stratification of $\R$ consisting of a single stratum, this forms a constant rank stratification of $f$. The induced partition (\ref{indpartfibconstrk}) of the fiber of $f$ over the origin consists of two connected surfaces and the $z$-axis. This does not satisfy the frontier condition, because the negative part of the $z$-axis is not contained in the closure of these surfaces. 
\end{ex}
We will now give a criterion that does ensure that the induced partitions (\ref{indpartfibconstrk}) of the fibers form stratifications. Recall that a map between semi-algebraic sets is called semi-algebraic when its graph is a semi-algebraic set. Below, let $f:(X,\O_X)\to (Y,\O_Y)$ be a morphism and $(\P_X,\P_Y)$ a partition $f$ into submanifolds. 
\begin{defi}\label{locsemalgdefimap} We call $(f,\P_X,\P_Y)$ \textbf{locally semi-algebraic} at $x\in X$ if there are a chart $(U,\chi)$ around $x$ and a chart $(V,\phi)$ around $f(x)$ with $f(U)\subset V$, that map the respective members of $\P_X\vert_U$ and $\P_Y\vert_V$ onto semi-algebraic sets, and have the property that the coordinate representation $\phi\circ f\circ \chi^{-1}$ restricts to semi-algebraic maps between the members of $\chi(\P_X\vert_U)$ and $\phi(\P_Y\vert_V)$. We call $(f,\P_X,\P_Y)$ locally semi-algebraic if it is so at every $x\in X$. 
\end{defi} 
\begin{defi}\label{homdefimap} We call $(f,\P_X,\P_Y)$ \textbf{smoothly homogeneous} if for any two $x_1,x_2\in X$ that belong to the same member of $\P_X$, there are isomorphisms of reduced differentiable spaces: 
\begin{equation*} h_X:(U_1,\O_X\vert_{U_1})\xrightarrow{\sim} (U_2,\O_X\vert_{U_2}) \quad \& \quad h_Y:(V_1,\O_Y\vert_{V_1})\xrightarrow{\sim} (V_2,\O_Y\vert_{V_2})
\end{equation*}
from an open $U_1$ around $x_1$ onto an open $U_2$ around $x_2$ in $X$, and from an open $V_1$ around $f(x_1)$ onto an open $V_2$ around $f(x_2)$ in $Y$, that fit in a commutative diagram:
\begin{center}
\begin{tikzcd} 
(U_1,\O_X\vert_{U_1},x_1)\arrow[d,"f"]\arrow[r,"h_X"] & (U_2,\O_X\vert_{U_2},x_2) \arrow[d,"f"] \\
 (V_1,\O_Y\vert_{V_1},f(x_1)) \arrow[r,"h_Y"] & (V_2,\O_Y\vert_{V_2},f(x_2))
\end{tikzcd}
\end{center}
and have the property that for all $\Sigma_X\in \P_X$, $\Sigma_Y\in \P_Y$: 
\begin{equation*} h_X(U_1\cap \Sigma_X)=U_2\cap \Sigma_X \quad \& \quad h_Y(V_1\cap \Sigma_Y)=V_2\cap \Sigma_Y.
\end{equation*} 
\end{defi} 
\begin{rem}\label{smoothhomimpliesconstrk} Notice that if $(f,\P_X,\P_Y)$ is smoothly homogeneous, then $(\P_X,\P_Y)$ is necessarily a constant rank partition of $f$. 
\end{rem}
The following shows that, if both of these criteria are met, then the fibers of $f$ meet the criteria of Lemma \ref{whitreghomprop}.
\begin{prop}\label{whitreghompropmap} Let $f:(X,\O_X)\to (Y,\O_Y)$ be a morphism and $(\P_X,\P_Y)$ a constant rank partition of $f$. If $(f,\P_X,\P_Y)$ is smoothly homogeneous and locally semi-algebraic, then so are the induced partitions (\ref{indpartfibconstrk}) of the fibers of $f$.
\end{prop}
The proof of this is straightforward. Appealing to Lemma \ref{whitreghomprop} and Remark \ref{passtoconncomprem} we obtain:
\begin{cor}\label{fibwhitstratcor} Let $f:(X,\O_X)\to (Y,\O_Y)$ be a morphism and $(\P_X,\P_Y)$ a constant rank partition of $f$. Suppose that $\P_X$ is locally finite and that $(f,\P_X,\P_Y)$ is smoothly homogeneous and locally semi-algebraic. Then the partitions of the fibers of $f$ obtained from (\ref{indpartfibconstrk}) after passing to connected components are Whitney stratifications of the fibers. 
\end{cor}
\subsection{The stratifications associated to Hamiltonian actions}\label{canhamstratsec}
\subsubsection{The canonical Hamiltonian stratification and Hamiltonian Morita types}
Throughout, let $(\G,\Omega)$ be a proper symplectic groupoid and suppose that we are given a Hamiltonian $(\G,\Omega)$-action along $J:(S,\omega)\to M$. Let $\underline{S}:=S/\G$ denote the orbit space of the action and $\underline{M}:=M/\G$ the leaf space of the groupoid. The construction of the canonical Hamiltonian stratifications on $S$ and $\underline{S}$ is of the sort outlined in Subsection \ref{stratdefsec}. To begin with, we give a natural partition that, after passing to connected components, will induce the desired stratification.

\begin{defi} The partition $\P_{\textrm{Ham}}(S)$ of $S$ by \textbf{Hamiltonian Morita types} is defined by the equivalence relation: $p_1\sim p_2$ if and only if there are invariant opens $V_i$ around $\L_{J(p_i)}$ in $M$, invariant opens $U_i$ around $\O_{p_i}$ in $J^{-1}(V_i)$, together with a Hamiltonian Morita equivalence (as in Definition \ref{moreqdefHamacttyp}) that relates $\O_{p_1}$ to $\O_{p_2}$: 
\begin{center}
\begin{tikzpicture} \node (G1) at (0,0) {$(\G,\Omega)\vert_{V_1}$};
\node (M1) at (0,-1.3) {$V_1$};
\node (O) at (2,0) {$(U_1,\omega)$};

\node[transparent] (A) at (2.2,-0.65) {$A$};
\node[transparent] (B) at (3.4,-0.65) {$B$};

\node (G) at (4,0) {$(\G,\Omega)\vert_{V_2}$};
\node (M) at (4,-1.3) {$V_2$};
\node (S) at (6,0) {$(U_2,\omega)$};
 
\draw[->,transform canvas={xshift=-\shift}](G1) to node[midway,left] {}(M1);
\draw[->,transform canvas={xshift=\shift}](G1) to node[midway,right] {}(M1);
\draw[->](O) to node[pos=0.25, below] {$\text{ }\text{ }J$} (M1);
\draw[->] (1.3,-0.15) arc (315:30:0.25cm);

\draw[transparent] (A) edge node[opacity=1] {\resizebox{0.8cm}{0.2cm}{$\simeq$}} (B);

\draw[->,transform canvas={xshift=-\shift}](G) to node[midway,left] {}(M);
\draw[->,transform canvas={xshift=\shift}](G) to node[midway,right] {}(M);
\draw[->](S) to node[pos=0.25, below] {$\text{ }\text{ }J$} (M);
\draw[->] (5.3,-0.15) arc (315:30:0.25cm);

\end{tikzpicture} 
\end{center}
The members of $\P_\textrm{Ham}(S)$ are invariant with respect to the $\G$-action, so that $\P_{\textrm{Ham}}(S)$ descends to a partition $\P_{\textrm{Ham}}(\underline{S})$ of $\underline{S}$. 
\end{defi} 
\begin{rem}\label{improphammortyp} Let us point out some immediate properties of these partitions.
\begin{itemize} \item[i)] They are invariant under Hamiltonian Morita equivalence, meaning that the homeomorphism of orbit spaces induced by a Hamiltonian Morita equivalence (Proposition \ref{transgeommap}$a$) identifies the partitions by Hamiltonian Morita types. 
\item[ii)] The transverse momentum map sends each member of $\P_\textrm{Ham}(\underline{S})$ into a member of $\P_\mathcal{M}(\underline{M})$ (the partition of $\underline{M}$ by Morita types of $\G$; see Example \ref{exmortyp}). 
\end{itemize}
\end{rem} 
In analogy with Example \ref{exmortyp}, the partition by Hamiltonian Morita types has the following alternative characterization. 
\begin{prop}\label{eqcharhammortyp} Two points $p,q\in S$ belong to the same Hamiltonian Morita type if and only if there is an isomorphism of pairs of Lie groups:
\begin{equation*} (\G_{J(p)},\G_p)\cong(\G_{J(q)},\G_q)
\end{equation*} together with a compatible symplectic linear isomorphism:
\begin{equation*} (\S\No_p,\omega_p)\cong (\S\No_q,\omega_q).
\end{equation*}
\end{prop}
\begin{proof} The forward implication is immediate from Proposition \ref{transgeomham}. For the converse, notice the following. Let $p\in S$, write $G=\G_{J(p)}$, $H=\G_p$ and $(V,\omega_V)=(\S\No_p,\omega_p)$, and let $\p:\h^*\to \g^*$ be any choice of $H$-equivariant splitting of (\ref{ses2poisabs}). Then from the normal form theorem, Example \ref{locmodmoreq} and Example \ref{crucexmoreq2}, it follows that there are invariant opens $W$ around $\L_x$ in $M$ and $U$ around $\O_p$ in $J^{-1}(W)$, together with a Hamiltonian Morita equivalence between the Hamiltonian $(\G,\Omega)\vert_W$-action along $J:(U,\omega)\to W$ and (a restriction of) the groupoid map of Hamiltonian type (\ref{bigJ_p}) (to invariant opens around the respective origins in $\h^0\oplus V$ and $\g^*$), that relates $\O_p$ to the origin in $\h^0\oplus V$. With this at hand the backward implication is clear, for (\ref{bigJ_p}) is built naturally out of the pair $(G,H)$, the symplectic representation $(V,\omega_V)$ and the splitting $\p$. 
\end{proof}
We now turn to the stratifications induced by the Hamiltonian Morita types.
\begin{defi} Let $\S_\textrm{Ham}(S)$ and $\S_\textrm{Ham}(\underline{S})$ denote the partitions obtained from the Hamiltonian Morita types on $S$ and $\underline{S}$, respectively, by passing to connected components. We call $\S_\textrm{Ham}(S)$ and $\S_\textrm{Ham}(\underline{S})$ the \textbf{canonical Hamiltonian stratifications}.  
\end{defi} 
The main aim of this section will be to prove:
\begin{thm}\label{canhamstratthm} Let $(\G,\Omega)\rightrightarrows M$ be a proper symplectic groupoid and suppose we are given a Hamiltonian $(\G,\Omega)$-action along $J:(S,\omega)\to M$. 
\begin{itemize}
\item[a)] The partition $\S_\textrm{Ham}(\underline{S})$ is a Whitney stratification of the orbit space $(\underline{S},\mathcal{C}^\infty_{\underline{S}})$.
 \item[b)] The pair consisting of the canonical Hamiltonian stratification of the orbit space $\underline{S}$ and the canonical stratification of the leaf space $\underline{M}$ of $\G$:
 \begin{equation*} \left(\S_\textrm{Ham}(\underline{S}),\S_\textrm{Gp}(\underline{M})\right)
 \end{equation*} is a constant rank stratification (as in Definition \ref{constrkstratdefi}) of the transverse momentum map:
\begin{equation}\label{transmommap} \underline{J}:(\underline{S},\mathcal{C}^\infty_{\underline{S}})\to (\underline{M},\mathcal{C}^\infty_{\underline{M}}).
\end{equation}
\end{itemize}
\end{thm}
The fiber of (\ref{transmommap}) over a leaf $\L$ of $(\G,\Omega)$, is (as topological space) the quotient $J^{-1}(\L)/\G$. This is the reduced space at $\L$ appearing in the procedure of symplectic reduction. Throughout, we will denote this as:
\begin{equation*} \underline{S}_\L:=\underline{J}^{-1}(\L),
\end{equation*} and we will simply denote the induced structure sheaf on the fiber space as $\mathcal{C}^\infty_{\underline{S}_\L}$. As we will show, $(\P_\textrm{Ham}(\underline{S}),\P_\mathcal{M}(\underline{M}))$ is a constant rank partition of the transverse momentum map (\ref{transmommap}), so that (as discussed in Subsection \ref{constrkstratsec}) the fiber $(\underline{S}_\L,\mathcal{C}^\infty_{\underline{S}_\L})$ has a natural partition into submanifolds:
\begin{equation}\label{indpartfibtrmommap} \P_\textrm{Ham}(\underline{S}_\L):=\{P\cap \underline{S}_\L\mid P\in \mathcal{P}_\textrm{Ham}(\underline{S})\}.
\end{equation} 
Besides Theorem \ref{canhamstratthm}, in this section we will prove:
\begin{thm}\label{redspstratthm} The fibers $(\underline{S}_{\L},\mathcal{C}^\infty_{\underline{S}_{\L}})$ of the transverse momentum map, endowed with the partition $\S_\textrm{Ham}(\underline{S}_{\L})$ obtained from (\ref{indpartfibtrmommap}) after passing to connected components, are Whitney stratified spaces. 
\end{thm}
In the case of a Hamiltonian action of a compact Lie group, the stratification $\S_\textrm{Ham}(\underline{S}_{\L})$ coincides with that in \cite{LeSj} (see also Remark \ref{hamliegpactdiffpartex}). \\

The partition $\S_\textrm{Ham}(S)$ of the smooth manifold $S$ turns out to be a Whitney stratification as well. Furthermore, in contrast to the stratification $\S_\textrm{Gp}(S)$ associated to the action groupoid, it is a constant rank stratification of the momentum map $J:S\to M$. This can be proved using the normal form theorem. Here we will not go into details on this, but rather focus on the proof of the theorems concerning the transverse momentum map. We can already give an outline of this. 
\begin{proof}[Outline of the proof of Theorem \ref{canhamstratthm} and \ref{redspstratthm}] In the coming subsection we will show that the Hamiltonian Morita types are submanifolds of the orbit space. By part ii) of Remark \ref{improphammortyp}, it then follows that the pair $(\P_\textrm{Ham}(\underline{S}),\P_\mathcal{M}(\underline{M}))$ is a partition of the transverse momentum map (\ref{transmommap}) into submanifolds (as in Definition \ref{constrkstratdefi}). In complete analogy with our proof of Theorem \ref{stratleafspthm},  Proposition \ref{transgeommap}$a$ and \ref{moreqisoredring} imply that the triple $(\underline{J},\P_\textrm{Ham}(\underline{S}),\P_\mathcal{M}(\underline{M}))$ is smoothly homogeneous (as in Definition \ref{homdefimap}). In particular, $(\P_\textrm{Ham}(\underline{S}),\P_\mathcal{M}(\underline{M}))$ is a constant rank partition of (\ref{transmommap}) (see Remark \ref{smoothhomimpliesconstrk}). In Subsection \ref{locpropcanhamstratpropsec} we further prove that $\P_\textrm{Ham}(\underline{S})$ is locally finite and that $(\underline{J},\P_\textrm{Ham}(\underline{S}),\P_\mathcal{M}(\underline{M}))$ is locally semi-algebraic (as in Definition \ref{locsemalgdefimap}). Combining Lemma \ref{whitreghomprop} with part i) of Remark \ref{passtoconncomprem}, it then follows that $\S_\textrm{Ham}(\underline{S})$ is indeed a Whitney stratification of the orbit space, completing the proof of Theorem \ref{canhamstratthm}. Furthermore, Theorem \ref{redspstratthm} is then a consequence of Corollary \ref{fibwhitstratcor}. 
\end{proof}

In the coming subsections we will address the remaining parts of the proof.

\subsubsection{Different partitions inducing the canonical stratifications} In this and the next subsection we study various local properties of the partition by Hamiltonian Morita types. To this end, it will be useful to consider the coarser partitions:
\begin{equation*} \P_{\cong_J}(S):=\P_{\cong}(S)\cap J^{-1}(\P_{\cong}(M))\quad \&\quad \P_{\cong_J}(\underline{S}):=\P_{\cong}(\underline{S})\cap \underline{J}^{-1}(\P_{\cong}(\underline{M})),
\end{equation*} where we take memberwise pre-images and intersections. Explicitly: $p,q\in S$ belong to the same member of $\P_{\cong_J}(S)$ if and only if $\G_{p}\cong \G_q$ and $\G_{J(p)}\cong \G_{J(q)}$. 
\begin{defi} We call $\P_{\cong_J}(S)$ and $\P_{\cong_J}(\underline{S})$ the partitions by \textbf{$J$-isomorphism types}. 
\end{defi}
In the remainder of this subsection, we will prove:
\begin{prop}\label{partprop} Both on $S$ and $\underline{S}$, the following hold. 
\begin{itemize}\item[a)] Each $J$-isomorphism type is a submanifold with connected components of possibly varying dimension.
\item[b)] The $J$-isomorphism types and the Hamiltonian Morita types yield the same partition after passing to connected components.
\item[c)] Each Hamiltonian Morita type is (in fact) a submanifold with connected components of a single dimension.
\end{itemize} Moreover, the orbit projection $S\to \underline{S}$ restricts to a submersion between the Hamiltonian Morita types (respectively the $J$-isomorphism types).
\end{prop}

To prove this proposition we will compute the Hamiltonian Morita types and the $J$-isomorphism types in the local model for Hamiltonian actions. There are two important remarks here that simplify this computation: first of all, the partitions by $J$-isomorphism types introduced above make sense for any groupoid map and, secondly they are invariant under Morita equivalence of Lie groupoid maps. Therefore, the computation of these reduces to that for the groupoid map $\J_\p$ of Example \ref{locmodmoreq}, which is the content of the lemma below.

\begin{lemma}\label{techlemisotype} Let $G$ be a compact Lie group, $H\subset G$ a closed subgroup and $(V,\omega_V)$ a symplectic $H$-representation. Fix an $H$-equivariant splitting $\p:\h^*\to \g^*$ of (\ref{ses2poisabs}). Consider the groupoid map $\J_\p$ defined in (\ref{bigJ_p}). 
\begin{itemize}
\item[a)] The $J_\p$-isomorphism type through the origin in $\h^0\oplus V$ is equal to the linear subspace: \begin{equation*} (\h^0)^G\oplus V^H
\end{equation*} where $(\h^0)^G$ and $V^H$ are the sets of points in $\h^0$ and $V$ that are fixed by $G$ and $H$.
\item[b)] The $G$-isomorphism type through the origin in $\g^*$ is equal to $(\g^*)^G$.
\item[c)] The restriction of $J_\p$ to these isomorphism types is given by:
\begin{equation}\label{mommaplocmodcentisotyp} (\h^0)^G\oplus V^H\to (\g^*)^G, \quad (\alpha,v)\mapsto \alpha.
\end{equation} 
\item[d)] Considered as subspace of the reduced differentiable space $(\h^0\oplus V)/H$ (resp. $\g^*/G$), the $J_\p$-isomorphism type $(\h^0)^G\oplus V^H$ (resp. $G$-isomorphism type $(\g^*)^G$) is a closed submanifold. 
\end{itemize}
\end{lemma}
\begin{proof} We use a standard fact: given a compact Lie group $H$ and a closed subgroup $K$, if $K$ is diffeomorphic to $H$, then $K=H$. Since the origin is fixed by $H$ it follows from this fact that for $(\alpha,v)\in \h^0\oplus V$ we have: 
\begin{align*} (\alpha,v)\cong (0,0) &\iff H_{(\alpha,v)}\cong H\\
&\iff H_{(\alpha,v)}=H\\
&\iff \alpha\in (\h^0)^H \quad \&\quad v\in V^H.
\end{align*} 
Similarly, for $\alpha\in \g^*$, it follows that: 
\begin{equation*} \alpha\cong 0\iff \alpha\in (\g^*)^G.
\end{equation*} Moreover, (\ref{quadsympmommap}) implies that $J_V$ vanishes on $V^H$ and hence $J_\p(\alpha,v)=\alpha$ for $v\in V^H$. Therefore: \begin{align*} 
(\alpha,v)\cong_J(0,0)&\iff (\alpha,v)\cong (0,0)\quad \&\quad \alpha\cong 0, \\
&\iff \alpha\in (\h^0)^G\quad \&\quad v\in V^H,
\end{align*} and we conclude that both $a$ and $b$ hold. Since $J_V$ vanishes on $V^H$, part $c$ follows as well. As for part $d$, it is clear that the canonical inclusion $(\h^0)^G\oplus V^H\hookrightarrow (\h^0\oplus V)/H$ is a closed topological embedding and a morphism of reduced differentiable spaces with respect to the standard manifold structure on the vector space $(\h^0)^G\oplus V^H$. Furthermore, choosing an $H$-invariant linear complement to $(\h^0)^G\oplus V^H$, we can extend any smooth function defined on an open in the vector space $(\h^0)^G\oplus V^H$ (by zero) to an $H$-invariant smooth function defined on an open in $\h^0\oplus V$. So, $(\h^0)^G\oplus V^H$ is indeed a closed submanifold of $(\h^0\oplus V)/H$. The argument for $(\g^*)^G$ in $\g^*/G$ is the same. 
\end{proof}
\begin{proof}[Proof of Proposition \ref{partprop}] Near a given orbit in $S$, we can identify the member of $\P_{\cong_J}(S)$ (resp. $\P_\textrm{Ham}(S)$) through this orbit (via the normal form theorem) with the corresponding member through the orbit $\O:=P/H$ in the local model for the Hamiltonian action (in the notation of Subsection \ref{hamlocmodconsec}). Using the Morita equivalence of Example \ref{locmodmoreq}, combined with Lemma \ref{techlemisotype} and the Morita invariance of the partitions by isomorphism types, we find that the $J_\theta$-isomorphism type through the orbit $\O$ in $S_\theta$ is a submanifold, being an open around $\O$ in: 
\begin{equation}\label{centhamtyp} \O\times \left((\h^0)^G\oplus V^H\right).
\end{equation} Therefore, the $J$-isomorphism types are submanifolds of $S$ with connected components of possibly varying dimension. Passing to the orbit space of the local model, we can again use the Morita equivalence of Example \ref{locmodmoreq} to identify the orbit space of the local model with an open neighbourhood of the origin in $(\h^0\oplus V)/H$, as reduced differentiable spaces (see Corollary \ref{moreqisoredring}). By Lemma \ref{techlemisotype} and Morita invariance of the partitions by isomorphism types, the $J_\theta$-isomorphism type through $\O$ is identified with an open neighbourhood of the origin in the submanifold $(\h^0)^G\oplus V^H$ of $(\h^0\oplus V)/H$. Therefore, the $J$-isomorphism types are submanifolds of $(\underline{S},\mathcal{C}^\infty_{\underline{S}})$ with connected components of possibly varying dimension. This proves part $a$. For part $b$, it suffices to prove that the Hamiltonian Morita type through the orbit $\O$ in the local model coincides with the $J_\theta$-isomorphism type computed above (by Lemma \ref{passconncomp} and the normal form theorem). That is, we have to verify that all $[p,\alpha,v]\in S_\theta$ such that $(\alpha,v)\in (\h^0)^G\oplus V^H$ belong to the same Hamiltonian Morita type. To this end, we again use the Hamiltonian Morita equivalence of Example \ref{locmodmoreq}. Let $[p,\alpha,v]$ be as above. Then the Morita equivalence relates $[p,\alpha,v]$ to $(\alpha,v)$. Since $v\in V^H$, it holds for all $w\in V$ that:
\begin{equation}\label{quadmomfixset} 
J_V(w+v)=J_V(w),
\end{equation} as follows from (\ref{quadsympmommap}). This implies that $\S\No_{(\alpha,v)}=V$ and therefore the conditions in Proposition \ref{finremhammoreq} are satisfied for the aforementioned Morita equivalence, at the points $[p,\alpha,v]$ and $(\alpha,v)$. Moreover, we have $H_{(\alpha,v)}=H$, $G_{J_\p(\alpha,v)}=G$ and, by linearity of the $H$-action, $\S\No_{(\alpha,v)}$ and $V$ in fact coincide as $H$-representations. So, applying the proposition, we obtain an isomorphism $\G_{J_\theta([p,\alpha,v])}\cong G$ that restricts to an isomorphism $\G_{[p,\alpha,v]}\cong H$, and we obtain a compatible isomorphism of symplectic representations: \begin{equation*} (\S\No_{[p,\alpha,v]},\omega_{[p,\alpha,v]})\cong (V,\omega_V).\end{equation*} So, all such $[p,\alpha,v]$ indeed belong to the same Hamiltonian Morita type. For part $c$ it remains to show for each Hamiltonian Morita type in $S$ or $\underline{S}$, the connected components have the same dimension. This follows from Proposition \ref{transgeomham} and a dimension count. Finally, in the above description of the Hamiltonian Morita types and $J$-isomorphism types in $S$ and $\underline{S}$ through $\O$, the orbit projection is identified (near $\O$) with the projection $\O\times (\h^0)^G\oplus V^H\to (\h^0)^G\oplus V^H$. This shows that it restricts to a submersion between the members in $S$ and $\underline{S}$. 
\end{proof}

\begin{rem}\label{hamliegpactdiffpartex} Let $G$ be a compact Lie group and $J:(S,\omega)\to \g^*$ a Hamiltonian $G$-space. The partition in Example \ref{hamGspex}, which is an analogue of the partition by orbit types for proper Lie group actions (cf. Example \ref{exorbtyp}), induces the canonical Hamiltonian stratification as well after passing to connected components. Another interesting partition that induces the canonical Hamiltonian stratification in this way can be defined by the equivalence relation: $p\sim q$ if and only if there is a $g\in G$ such that $G_p=gG_qg^{-1}$ and $G_{J(p)}=gG_{J(q)}g^{-1}$, together with a compatible symplectic linear isomorphism $(\S\No_p,\omega_p)\cong (\S\No_q,\omega_q)$. This is an analogue of the partition by local types for proper Lie group actions. The fact that these indeed induce the canonical Hamiltonian stratification follows from the same arguments as in the proof above, using the normal form theorem with the explicit isomorphism of symplectic groupoids (\ref{isolocmodmgsmod}) (see Remark \ref{sharphamliegpactrem}). Similarly, the partition (\ref{indpartfibtrmommap}) of $\underline{S}_\L$ and the partition used in \cite{LeSj} (given by: $\O_p\sim \O_q$ if and only if there is a $g\in G$ such that $G_p=gG_qg^{-1}$) yield the same partition after passing to connected components. 
\end{rem}

\begin{ex} Let $G$ be a compact Lie group and $J:(S,\omega)\to \g^*$ a Hamiltonian $G$-space. The fixed point set $M^G$ is a member of the partition in Example \ref{hamGspex} (provided it is non-empty). From the above remark we recover the well-known fact that for any two points $p,q\in S$ belonging to the same connected component of $M^G$, the symplectic $G$-representations $(T_pM,\omega_p)$ and $(T_qM,\omega_q)$ are isomorphic.
\end{ex}

\begin{ex} Let $T$ be a torus and $J:(S,\omega)\to \t^*$ a Hamiltonian $T$-space. In this case, the partition in Example \ref{hamGspex} coincides with the partition by orbit types of the $T$-action. Furthermore, the above remark implies that for any two points $p,q\in S$ belonging to the same connected component of an orbit type with isotropy group $H$, the symplectic normal representations at $p$ and $q$ are isomorphic as symplectic $H$-representations.
\end{ex}

\subsubsection{End of the proof}\label{locpropcanhamstratpropsec}
To complete the proof of Theorem \ref{canhamstratthm} and \ref{redspstratthm}, it remains to show:
\begin{prop}\label{weakwhitstratthm} The partition by Hamiltonian Morita types $\P_\textrm{Ham}(\underline{S})$ is locally finite and the triple $(\underline{J},\P_\textrm{Ham}(\underline{S}),\P_{\mathcal{M}}(\underline{M}))$ is locally semi-algebraic (as in Definition \ref{locsemalgdefimap}). 
\end{prop}
\begin{proof}[Proof of Proposition \ref{weakwhitstratthm}] Let $p\in S$, let $\O_p$ be the orbit through $p$ and let $\L_x=J(\O_p)$ be the corresponding leaf through $x=J(p)$. Further, let $G=\G_x$ denote the isotropy group of $\G$ at $x$, $H=\G_p$ the isotropy group of the action at $p$, and let $(V,\omega_V)=(\S\No_p,\omega_p)$ denote the symplectic normal representation at $p$. As in the proof of Proposition \ref{eqcharhammortyp}, there are invariant opens $W$ around $\L_x$ in $M$ and $U$ around $\O_p$ in $J^{-1}(W)$, together with a Hamiltonian Morita equivalence between the action of $(\G,\Omega)\vert_W$ along $J:U\to W$ and a restriction of the groupoid map (\ref{bigJ_p}), that relates $\O_p$ to the origin in $\h^0\oplus V$. Here, we can arrange the opens in $\h^0\oplus V$ and $\g^*$ to which $(\ref{bigJ_p})$ is restricted to be invariant open balls $B_{\g^*}\subset \g^*$ and $B_{\h^0\oplus V}\subset J_\p^{-1}(B_{\g^*})$ (with respect to a choice of invariant inner products) centered around the respective origins. Let $\rho:\h^0\oplus V\to \R^n$ and $\sigma:\g^*\to \R^m$ be Hilbert maps (see Subsection \ref{reddiffspsec}). By the same reasoning as in \cite[Example 6.5]{LeSj}, since $J_\p:\h^0\oplus V\to \g^*$ is an $H$-equivariant and polynomial map, there is a polynomial map $P:\R^n\to \R^m$ that fits into a commutative square:
\begin{center}
\begin{tikzcd} 
 (\h^0\oplus V)/H \arrow[d,"\underline{J_\p}"] \arrow[r,"\underline{\rho}"] & \R^n\arrow[d,"P"] \\
 \g^*/G \arrow[r,"\underline{\sigma}"] & \R^m
\end{tikzcd}
\end{center}
In view of Proposition \ref{transgeommap}$a$, Corollary \ref{moreqisoredring} and the discussion at the end of Subsection \ref{reddiffspsec}, we obtain a diagram of reduced differentiable spaces:
\begin{center}
\begin{tikzcd} 
 \left(\underline{U},\mathcal{C}^\infty_{\underline{S}} \vert_{\underline{U}}\right) \arrow[r,"\sim"] \arrow[d,"\underline{J}"] & \left(\underline{B}_{\h^0\oplus V},\mathcal{C}^\infty_{(\h^0\oplus V)/H}\vert_{\underline{B}_{\h^0\oplus V}}\right) \arrow[d,"\underline{J_\p}"] \arrow[r,"\sim"] & \left(\rho(B_{\h^0\oplus V}), \mathcal{C}^\infty_{\rho(B_{\h^0\oplus V})}\right)\arrow[d,"P"]  \\
 \left(\underline{W},\mathcal{C}^\infty_{\underline{M}} \vert_{\underline{W}}\right) \arrow[r,"\sim"] & \left(\underline{B}_{\g^*},\mathcal{C}^\infty_{\g^*/G}\vert_{\underline{B}_{\g^*}}\right) \arrow[r,"\sim"] & \left(\sigma(B_{\g^*}),\mathcal{C}^\infty_{\sigma(B_{\g^*})}\right)
\end{tikzcd}
\end{center}
in which all horizontal arrows are isomorphisms. Due to Morita invariance of the partitions by isomorphism types, the partition of $\underline{U}$ by $J$-isomorphism types is identified with the partition of $\rho(B_{\h^0\oplus V})$ consisting of the subsets of the form:
\begin{equation*} \rho(B_{\h^0\oplus V})\cap\rho(\Sigma_{\h^0\oplus V})\cap P^{-1}(\sigma(\Sigma_{\g^*})) , \quad\quad \Sigma_{\h^0\oplus V}\in \P_{\cong}(\h^0\oplus V), \quad \Sigma_{\g^*}\in \P_{\cong}(\g^*).
\end{equation*} 
Recall from the proof of Theorem \ref{stratleafspthm} that the canonical stratification of the orbit space of a real, finite-dimensional representation of a compact Lie group has finitely many strata, each of which is mapped onto a semi-algebraic set by any Hilbert map. The same must then hold for the partition by isomorphism types of such a representation. The above partition of $\rho(B_{\h^0\oplus V})$ therefore also has finitely many members, each of which is semi-algebraic, for $P$ is polynomial and $\rho(B_{\h^0\oplus V})$ is semi-algebraic (being the image of a semi-algebraic set under a semi-algebraic map). The same then holds for the partition obtained after passing to connected components, because any semi-algebraic set has finitely many connected components, each of which is again semi-algebraic. By Proposition \ref{partprop}$b$, the members of $\P_\textrm{Ham}(\underline{S})\vert_{\underline{U}}$ are unions of the connected components of the $J$-isomorphism types in $\underline{U}$. So, $\P_\textrm{Ham}(\underline{S})\vert_{\underline{U}}$ has finitely many members, each of which is mapped onto a semi-algebraic set by the above chart for $(\underline{S},\mathcal{C}^\infty_{\underline{S}})$, the image being a finite union of semi-algebraic sets. By similar reasoning, the above chart for $(\underline{M},\mathcal{C}^\infty_{\underline{M}})$ maps the members of $\P_\mathcal{M}(\underline{M})\vert_{\underline{W}}$ onto semi-algebraic sets. Since $P$ is polynomial, it restricts to semi-algebraic maps between the images of these members under the above charts. So, this proves the proposition.   
\end{proof}

We end this section with a concrete example, similar to that in \cite{ArCuGot}.  
\begin{ex}\label{concreteex} Let $G=\textrm{SU}(2)\times\textrm{SU}(2)$ and consider the circle in $G$ given by the closed subgroup:
\begin{equation*} H=\left\{ \left(\begin{pmatrix} e^{i\theta} & 0 \\ 0 & e^{-i\theta} \end{pmatrix}, \begin{pmatrix} e^{i\theta} & 0 \\ 0 & e^{-i\theta}  \end{pmatrix}\right)\in \textrm{SU}(2)\times \textrm{SU}(2) \textrm{ }\middle\vert\textrm{ }\theta\in \R\right\}.
\end{equation*} The cotangent bundle $T^*(G/H)$ of the homogeneous space $G/H$ is naturally a Hamiltonian $G$-space, and the canonical Hamiltonian strata can be realized as concrete semi-algebraic submanifolds of $\R^5$, as follows. The orbit space of the $G$-action on $T^*(G/H)$ can be canonically identified with the orbit space of the linear $H$-action on $\h^0$ induced by the coadjoint action of $G$ on $\g^*$, and the transverse momentum map becomes the map $\underline{J}:\h^0/H\to \g^*/G$ induced by the inclusion $\h^0\hookrightarrow \g^*$. To find Hilbert maps for $\g^*$ and $\h^0$, consider the $\text{SU}(2)$-invariant inner product on $\mathfrak{su}(2)$ given by:
\begin{equation}\label{invinprodsu(2)} \langle A, B \rangle_{\mathfrak{su}(2)}=-\textrm{Trace}(AB)\in \R,
\end{equation} and notice that under the identification of $\mathfrak{su}(2)$ with $\R\times \C$ obtained by writing:
\begin{equation*} \mathfrak{su}(2)=\left\{\begin{pmatrix} i\theta & -\bar{z} \\ z & -i\theta \end{pmatrix}\in \mathfrak{gl}(2,\C) \text{ }\middle\vert\text{ } \theta \in \R,\text{ } z\in \C\right\},
\end{equation*} (\ref{invinprodsu(2)}) corresponds (up to a factor) to the standard Euclidean inner product. Using the induced $G$-invariant inner product on $\g=\mathfrak{su}(2)\times \mathfrak{su}(2)$, we identify $\g^*$ with $\g$. The orbits of the adjoint $\textrm{SU}(2)$-action on $\mathfrak{su}(2)$ are the origin and the concentric spheres centered at the origin. Using this, one readily sees that the algebra of $\textrm{SU}(2)$-invariant polynomials on $\mathfrak{su}(2)$ is generated by the single polynomial given by the square of the norm induced by (\ref{invinprodsu(2)}). So, the algebra of $G$-invariant polynomials on $\g^*$ is generated by:
\begin{equation*} \sigma_1(\theta_1,z_1,\theta_2,z_2)=\theta_1^2+|z_1|^2, \quad \sigma_2(\theta_1,z_1,\theta_2,z_2)=\theta_2^2+|z_2|^2, \quad \theta_1,\theta_2\in \R, \quad z_1,z_2\in \C.
\end{equation*} 
On the other hand, $\h^0$ is identified with the orthogonal complement:
\begin{equation*} \h^{\perp}=\left\{ \left(\begin{pmatrix} i\theta & -\bar{z}_1 \\ z_1 & -i\theta \end{pmatrix}, \begin{pmatrix} -i\theta & -\bar{z}_2 \\ z_2 & i\theta  \end{pmatrix}\right)\in \mathfrak{su}(2)\times \mathfrak{su}(2) \textrm{ }\middle\vert\textrm{ }\theta\in \R,\text{ }z_1,z_2\in \C\right\}.
\end{equation*} Identifying $\h^\perp$ with $\R\times \C^2$ accordingly, the $H$-orbits are identified with those of the $\mathbb{S}^1$-action: 
\begin{equation*} \lambda\cdot (\theta,z_1,z_2)=(\theta,\lambda z_1,\lambda z_2), \quad \lambda\in \mathbb{S}^1, \quad (\theta,z_1,z_2)\in \R\times \C^2.
\end{equation*} In light of this, the algebra of $H$-invariant polynomials on $\h^0$ is generated by:
\begin{alignat*}{2} \rho_1(\theta,&z_1,z_2)=\theta,\quad\quad\quad \rho_2(\theta,z_1,z_2)=&&|z_1|^2,\quad\quad\quad \rho_3(\theta,z_1,z_2)=|z_2|^2,\\
&\rho_4(\theta,z_1,z_2)=\textrm{Re}(z_1\bar{z}_2),\quad\quad\quad &&\rho_5(\theta,z_1,z_2)=\textrm{Im}(z_1\bar{z}_2).
\end{alignat*} 
Now, consider the polynomial map: 
\begin{equation*} P:\R^5\to \R^2, \quad P(x_1,...,x_5)=\left(x_1^2+x_2,x_1^2+x_3\right).
\end{equation*} Then we have a commutative square:
\begin{center}
\begin{tikzcd} \h^0/H\arrow[r,"\underline{\rho}"]\arrow[d,"\underline{J}"] & \R^5 \arrow[d,"P"]\\
\g^*/G \arrow[r,"\underline{\sigma}"] & \R^2
\end{tikzcd} 
\end{center} The image of $\h^0/H$ under $\underline{\rho}$ is the semi-algebraic subset of $\R^5$ given by: 
\begin{equation*}\{x_2\geq 0,\textrm{ }x_3\geq 0, \textrm{ }x_4^2+x_5^2=x_2x_3\}, 
\end{equation*} whereas the image of $\g^*/G$ under $\underline{\sigma}$ is the semi-algebraic subset of $\R^2$ given by: 
\begin{equation*} \{y_1\geq 0,\text{ }y_2\geq 0\}.
\end{equation*} The canonical stratification of the orbit space of the $G$-action on $T^*(G/H)$ has two strata, corresponding to the semi-algebraic submanifolds of $\R^5$ given by: 
\begin{align}\label{orbtypstratex1} &\{x_2=x_3=x_4=x_5=0\}, \\
\label{orbtypstratex2} &\{x_4^2+x_5^2=x_2x_3\}\cap (\{x_2> 0\}\cup\{x_3>0\}).
\end{align} On the other hand, the canonical stratification of $\g^*/G$ has four strata, corresponding to the semi-algebraic submanifolds of $\R^2$ given by: \begin{equation*} \{y_1=y_2=0\},\quad \{y_1>0,\text{ }y_2=0\},\quad \{y_1=0,\text{ }y_2>0\}, \quad\{y_1>0,\text{ }y_2>0\}.\end{equation*} From this we see that the canonical Hamiltonian stratification of the orbit space of the Hamiltonian $G$-space $T^*(G/H)$ has six strata, three of which correspond to the semi-algebraic submanifolds of (\ref{orbtypstratex1}) given by the respective intersections of (\ref{orbtypstratex1}) with $\{x_1<0\}$, $\{x_1=0\}$ and $\{x_1>0\}$, and the other three of which correspond to the semi-algebraic submanifolds of (\ref{orbtypstratex2}) given by:
\begin{align*} \{&x_1=0,\text{ }x_2>0,\text{ }x_3=x_4=x_5=0\},\quad \{x_1=x_2=0,\text{ }x_3>0,\text{ }x_4=x_5=0\},\\
&\{x_1^2+x_2>0,\text{ }x_1^2+x_3>0,\text{ }x_4^2+x_5^2=x_2x_3\}\cap (\{x_2> 0\}\cup\{x_3>0\}).
\end{align*} The restriction of $P$ to any of the first five strata is injective, hence its fibers are points. The restriction of $P$ to the last stratum has $2$-dimensional fibers. In fact, given $y_1,y_2>0$ the fiber of this restricted map over $(y_1,y_2)\in \R^2$ is projected diffeomorphically onto the semi-algebraic submanifold of $\R^3$ given by:
\begin{equation*} \{(x_1,x_4,x_5)\in \R^3\mid x_4^2+x_5^2=(y_1-x_1^2)(y_2-x_1^2),\text{ }x_1^2<\max(y_1,y_2)\},
\end{equation*} which is semi-algebraically diffeomorphic to a $2$-sphere if $y_1\neq y_2$, whereas it is semi-algebraically diffeomorphic to a $2$-sphere with two points removed if $y_1=y_2$. 
\end{ex}
\subsection{The regular parts of the stratifications}\label{regpartsec} 
\subsubsection{The regular part of a stratification}
To start with, we give a reminder on the regular part of a stratification, mostly following the exposition in \cite{CrMe}. A stratification $\S$ of a space $X$ comes with a natural partial order given by:
\begin{equation}\label{partordstrat} \Sigma\leq \Sigma' \quad \iff \quad \Sigma \subset \overline{\Sigma'}.
\end{equation} We say that a stratum $\Sigma\in \S$ is \textbf{maximal} if it is maximal with respect to this partial order. Maximal strata can be characterized as follows.
\begin{prop}\label{regpartopen} Let $(X,\S)$ be a stratified space. Then $\Sigma \in\S$ is maximal if and only if it is open in $X$. Moreover, the union of all maximal strata is open and dense in $X$. 
\end{prop}
\begin{defi} The union of all maximal strata of a stratified space $(X,\S)$ is called the \textbf{regular part} of the stratified space. 
\end{defi}
Given a stratification $\S$, an interesting question is whether it admits a greatest element with respect to the partial order (\ref{partordstrat}). This is equivalent to asking whether the regular part of $\S$ is connected. 
\begin{ex}\label{prinorbthm} Let $G$ be a Lie group acting properly on a manifold $M$. The partition by orbit types $\P_\sim(\underline{M})$ (see Example \ref{exorbtyp}) comes with a partial order of its own. Namely, if $\underline{M}_x$ and $\underline{M}_y$ denote the orbit types containing the respective orbits $\O_x$ and $\O_y$, then by definition:
\begin{equation*} \underline{M}_x\leq \underline{M}_y\quad \iff G_y \text{ is conjugate in $G$ to a subgroup of }G_x.
\end{equation*} The principal orbit type theorem states that, if $\underline{M}$ is connected, then there is a greatest element with respect to this partial order, called the principal orbit type, which is connected, open and dense in $\underline{M}$. In this case, the regular part of $\S_\textrm{Gp}(\underline{M})$ coincides with the principal orbit type; in particular, it is connected. On the other hand, the regular part of $\S_\textrm{Gp}(M)$ need not be connected, even if $M$ is connected.
\end{ex}
\begin{ex}\label{princtypliegpoidex} Let $\G\rightrightarrows M$ be a proper Lie groupoid. We denote the respective regular parts of $\S_\textrm{Gp}(M)$ and $\S_\textrm{Gp}(\underline{M})$ as $M^\textrm{princ}$ and $\underline{M}^\textrm{princ}$. From the linearization theorem it follows that a point $x$ in $M$ belongs to $M^\textrm{princ}$ if and only if the action of $\G_x$ on $\No_x$ is trivial. From this it is clear that $M^\textrm{princ}$ and $\underline{M}^\textrm{princ}$ are unions of Morita types. The analogue of the principal orbit type theorem for Lie groupoids \cite[Theorem 15]{CrMe} states that, if $\underline{M}$ is connected, then $\underline{M}^\textrm{princ}$ is connected.
\end{ex}
The lemma below gives a useful criterion for the regular part to be connected. 
\begin{lemma}\label{cod1stratlem} Let $M$ be a connected manifold and $\S$ a stratification of $M$ by submanifolds. If $\S$ has no codimension one strata, then the regular part of $\S$ is connected.  
\end{lemma}
\begin{proof} As in the proof of \cite[Proposition 2.8.5]{DuKo}, by a transversality principle \cite[pg. 73]{GuPo} any smooth path $\gamma$ that starts and ends in the regular part is homotopic in $M$ to a path $\widetilde{\gamma}$ that intersects only strata of codimension at most $1$ and starts and ends at the same points as $\gamma$. 
\end{proof}
\begin{ex}\label{infstratex} Although $\S_\textrm{Gp}(M)$ may have codimension one strata, the base $M$ of a proper Lie groupoid $\G$ admits a second interesting Whitney stratification that does not have codimension one strata: the \textbf{infinitesimal stratification} $\S_\textrm{Gp}^\textrm{inf}(M)$. As for the canonical stratification, the infinitesimal stratification is induced by various different partitions of $M$. Indeed, each of the partitions mentioned in Subsection \ref{stratdefsec} has an infinitesimal analogue, obtained by replacing the Lie groups in their defining equivalence relations by the corresponding Lie algebras. Yet another partition that induces the infinitesimal stratification on $M$ is the partition $\P_\textrm{dim}(M)$ of $M$ by \textbf{dimension types}, defined by the equivalence relation: $x\sim y$ if and only if $\dim(\L_x)=\dim(\L_y)$, or equivalently, $\dim(\g_x)=\dim(\g_y)$. The members of each of these partitions are invariant. Therefore, each of these descends to a partition of $\underline{M}$. However, the members of $\S_\textrm{Gp}^\textrm{inf}(\underline{M})$ may fail to be submanifolds of the leaf space. For this reason we only consider the stratification on $M$.  
We let $M^\textrm{reg}$ denote the regular part of the infinitesimal stratification $\S_\textrm{Gp}^\textrm{inf}(M)$. As for the canonical stratification, this has a Lie theoretic description: a point $x$ in $M$ belongs to $M^\textrm{reg}$ if and only if the action of $\g_x$ on $\No_x$ is trivial. Since the infinitesimal stratification has no codimension one strata, Lemma \ref{cod1stratlem} applies. Therefore, $M^\textrm{reg}$ is connected if $M$ is connected. 
\end{ex}
\subsubsection{The infinitesimal Hamiltonian stratification}
In the remainder of this section we will study the regular part of both the canonical Hamiltonian stratification and of a second stratification associated to a Hamiltonian action of a proper symplectic groupoid, that we call the \textbf{infinitesimal Hamiltonian stratification}. We include the latter in this section, because a particularly interesting property of this stratification is that its regular part is better behaved than that of the canonical Hamiltonian stratification. To introduce the infinitesimal Hamiltonian stratification, let $(\G,\Omega)\rightrightarrows M$ be a proper symplectic groupoid and suppose that we are given a Hamiltonian $(\G,\Omega)$-action along $J:(S,\omega)\to M$. Each of the partitions of $S$ defined in Section \ref{canhamstratsec} has an infinitesimal counterpart, obtained by replacing the role of the isotropy Lie groups by the corresponding isotropy Lie algebras. For example, by definition two points $p,q\in S$ belong to the same \textbf{infinitesimal Hamiltonian Morita type} if there is an isomorphism of pairs of Lie algebras:
\begin{equation*} (\g_{J(p)},\g_p)\cong (\g_{J(q)},\g_q)
\end{equation*} together with a compatible symplectic linear isomorphism:
\begin{equation*} (\S\No_p,\omega_p)\cong (\S\No_q,\omega_q),
\end{equation*} where compatibility is now meant with respect to the Lie algebra actions. These partitions induce, after passing to connected components, one and the same Whitney stratification $\S_\textrm{Ham}^\text{inf}(S)$ of $S$: the infinitesimal Hamiltonian stratification. There is in fact an even simpler partition that induces this stratification, obtained from the partitions by dimensions of the orbits on $S$ and the leaves of $M$ (see Example \ref{infstratex}):
\begin{equation}\label{J-dimtypdef} \P_{\textrm{dim}_J}(S):=\P_\textrm{dim}(S)\cap J^{-1}(\P_\textrm{dim}(M)),
\end{equation} where we take memberwise intersections. Explicitly, two points $p,q\in S$ belong to the same member of (\ref{J-dimtypdef}) if and only if $\dim(\O_p)=\dim(\O_q)$ and $\dim(\L_{J(p)})=\dim(\L_{J(q)})$. That the members of the above partitions are submanifolds of $S$ (with connected components of possibly varying dimension) and that all of these partitions indeed yield one and the same partition $\S_\textrm{Ham}^\textrm{inf}(S)$ after passing to connected components follows from the same type of arguments as in the proof of Proposition \ref{partprop}. From the normal form theorem it further follows that $\S_\textrm{Ham}^\textrm{inf}(S)$ is a constant rank stratification of the momentum map. 
  
\subsubsection{Lie theoretic description of the regular parts}\label{liedescrpregpartsec}
Given a proper symplectic groupoid $(\G,\Omega)$ and a Hamiltonian $(\G,\Omega)$-action along $J:(S,\omega)\to M$, we will use the following notation for the regular parts of the various stratifications that we consider. 
\begin{itemize} \item For the canonical Hamiltonian stratifications $\S_\textrm{Ham}(S)$ and $\S_\textrm{Ham}(\underline{S})$, and the infinitesimal Hamiltonian stratification $\S_\textrm{Ham}^\text{inf}(S)$ of the Hamiltonian $(\G,\Omega)$-action:
\begin{equation*} S^\textrm{princ}_\textrm{Ham}, \quad \quad \underline{S}^\textrm{princ}_\textrm{Ham},\quad\quad S^\textrm{reg}_\textrm{Ham}.
\end{equation*} 
\item For the canonical stratifications $\S_\textrm{Gp}(S)$ and $\S_\textrm{Gp}(\underline{S})$ and the infinitesimal stratification $\S^\textrm{inf}_\textrm{Gp}(S)$ of the $\G$-action: 
\begin{equation*} S^\textrm{princ}, \quad \quad \underline{S}^\textrm{princ},\quad\quad S^\textrm{reg}.
\end{equation*}
\item For the stratification $\S_\textrm{Ham}(\underline{S}_\L)$ on the reduced space over a leaf $\L$:
\begin{equation*}
\underline{S}_\L^\textrm{princ}.
\end{equation*}
\end{itemize}

\begin{rem} Proposition \ref{regpartopen}, together with the fact that the orbit projection $q$ is open, implies: 
\begin{equation*} S^\textrm{princ}=q^{-1}(\underline{S}^\textrm{princ})\quad\& \quad S^\textrm{princ}_\textrm{Ham}=q^{-1}(\underline{S}^\textrm{princ}_\textrm{Ham}).
\end{equation*} Furthermore, there are obvious inclusions:
\begin{center}
\begin{tikzcd} &  S^\textrm{princ} \arrow[hookrightarrow]{rd} & \\
S^\textrm{princ}_\textrm{Ham} \arrow[hookrightarrow]{ru} \arrow[hookrightarrow]{rd} & & S^\textrm{reg} \\
& S^\textrm{reg}_\textrm{Ham} \arrow[hookrightarrow]{ru} &
\end{tikzcd}
\end{center}
\end{rem} We have the following Lie theoretic description of the regular parts. 
\begin{prop}\label{regpartalg} Let $p\in S$ and denote $x=J(p)\in M$. Then the following hold.
\begin{itemize}\item[a)] $p\in S^\textrm{princ}$ if and only if the actions of $\G_p$ on both $\g_p^0$ and on $\S\No_p$ are trivial.
\item[b)] $p\in S^\textrm{reg}$ if and only if the actions of $\g_p$ on both $\g_p^0$ and on $\S\No_p$ are trivial.
\item[c)] $p\in S^\textrm{princ}_\textrm{Ham}$ if and only if $p\in S^\textrm{princ}$ and $\G_x$ fixes $\g_p^0$.
\item[d)] $p\in S^\textrm{reg}_\textrm{Ham}$ if and only if $p\in S^\textrm{reg}$ and $\g_x$ fixes $\g_p^0$.
\item[e)] $\O_p\in \underline{S}^\textrm{princ}_\L$ if and only if the action of $\G_p$ on $(J_{\S\No_p})^{-1}(0)$ is trivial. 
\end{itemize}
\end{prop}
\begin{proof} We will only prove statement $c$, as the other statements follow by entirely similar reasoning. In view of the above remark, we may as well work on the level of $\underline{S}$. Let $G=\G_x$, $H=\G_p$ and $V=\S\No_p$. As in the proof of Proposition \ref{partprop}, near $\O_p$ we can identify the orbit space $\underline{S}$ with an open neighbourhood of the origin in $(\h^0\oplus V)/H$, in such a way that $\O_p$ is identified with the origin and the stratum $\Sigma\in \S_\textrm{Ham}(\underline{S})$ through $\O_p$ is identified (near $\O_p$) with an open in $(\h^0)^G\oplus V^H$. By invariance under scaling, the origin lies in the interior of $(\h^0)^G\oplus V^H$ in $(\h^0\oplus V)/H$ if and only if $(\h^0)^G=\h^0$ and $V=V^H$. So, statement $c$ follows. 
\end{proof}
Proposition \ref{regpartalg} has the following direct consequence.

\begin{cor}\label{prinJstrat} The canonical Hamiltonian stratification $\S_\textrm{Ham}(S^\textrm{princ})$ of the restriction of the Hamiltonian $(\G,\Omega)$-action on $S$ to $S^\textrm{princ}$ consists of strata of $\S_\textrm{Ham}(S)$. In particular, the regular part of $\S_\textrm{Ham}(S^\textrm{princ})$ coincides with $S^\textrm{princ}_\textrm{Ham}$. The same goes for the stratifications on $\underline{S}$ and the infinitesimal counterparts on $S$. 
\end{cor}

\subsubsection{Principal type theorems}
Next, for each of the stratifications listed before, we address the question of whether the regular part is connected. As in \cite[Section 2.8]{DuKo}, our strategy to answer this will be to study the occurence of codimension one strata. First of all, we have:
\begin{thm}\label{printypthminf} The infinitesimal Hamiltonian stratification $\S_\textrm{Ham}^\text{inf}(S)$ has no codimension one strata. In particular, if $S$ is connected, then $S^\textrm{reg}_\textrm{Ham}$ is connected as well. 
\end{thm}
The following will be useful to prove this.
\begin{lemma}\label{onedimreplem} Let $H$ be a compact Lie group and $W$ a real one-dimensional representation of $H$. Then $H$ acts by reflection in the origin. In particular, if $H$ is connected, then $H$ acts trivially. 
\end{lemma}
\begin{proof} By compactness of $H$, there is an $H$-invariant inner product $g$ on $W$. Therefore the representation $H\to \textrm{GL}(W)$ takes image in the orthogonal group $\textrm{O}(W,g)=\{\pm 1\}$. 
\end{proof}
\begin{proof}[Proof of Theorem \ref{printypthminf}] We will argue by contradiction. Suppose that $p\in S$ belongs to a codimension one stratum. Let $H$ and $G$ denote the respective identity components of $\G_p$ and $\G_{J(p)}$, and let $V=\S\No_p$. The normal form theorem and a computation analogous to the one for Lemma \ref{techlemisotype} show that $(\h^0)^{G}\oplus V^{H}$ must have codimension one in $\h^0\oplus V$. Since $H$ is compact, $V^{H}\subset V$ is a symplectic linear subspace, and so it has even codimension. Therefore, it must be so that $V^H=V$ and $(\h^0)^G$ has codimension one in $\h^0$. Appealing to Lemma \ref{onedimreplem}, we find that $H$ acts trivially on any $H$-invariant linear complement to $(\h^0)^G$ in $\h^0$. By compactness of $H$ we can always find such a complement, hence $H$ fixes all of $\h^0$. Therefore, $\h$ is a Lie algebra ideal in $\g$. Since $G$ is connected, this means that $\h^0$ is invariant under the coadjoint action of $G$. As for $H$, it now follows that $G$ must actually fix all of $\h^0$,  contradicting the fact that $(\h^0)^G$ has positive codimension in $\h^0$. 
\end{proof}

The situation for $\S_\textrm{Ham}(\underline{S})$ and $\S_\textrm{Ham}(S)$ is more subtle. Indeed, the regular parts of the canonical Hamiltonian stratification on both $S$ and $\underline{S}$ can be disconnected, even if both $S$, as well as the source-fibers and the base of $\G$ are connected. This is shown by the example below. 
\begin{ex}\label{prinpartdiscon} Consider the circle $\mathbb{S}^1$, the real line $\R$ and the $2$-dimensional torus $\mathbb{T}^2$ equipped with the $\Z_2$-actions given by:  
\begin{equation*} (\pm1)\cdot e^{i\theta}=e^{\pm i\theta}, \quad\quad (\pm1)\cdot x=\pm x, \quad\quad (\pm 1)\cdot (e^{i\theta_1},e^{i\theta_2})=(\pm e^{i\theta_1}, e^{i\theta_2}).
\end{equation*} Now, consider the proper Lie groupoid: 
\begin{equation}\label{gpoiddisconprincpartex} (\T^2\times \T^2)\times_{\Z_2} (\mathbb{S}^1\times \R)\rightrightarrows \T^2\times_{\Z_2}\R
\end{equation} with source, target and multiplication given by:
\begin{align*} s([e^{i\theta_1},e^{i\theta_2},e^{i\theta_3}, e^{i\theta_4},e^{i\theta},x])&=[e^{i\theta_3},e^{i\theta_4},x],\\
t([e^{i\theta_1},e^{i\theta_2},e^{i\theta_3}, e^{i\theta_4},e^{i\theta},x])&=[e^{i\theta_1},e^{i\theta_2},x], \\
m([e^{i\theta_1},e^{i\theta_2},e^{i\theta_3}, e^{i\theta_4},e^{i\theta},x],[e^{i\theta_3},e^{i\theta_4},e^{i\theta_5}, e^{i\theta_6},e^{i\phi},x])&=[e^{i\theta_1},e^{i\theta_2},e^{i\theta_5}, e^{i\theta_6}, e^{i(\theta+\phi)},x].
\end{align*}
This becomes a symplectic groupoid when equipped with the symplectic form induced by: 
\begin{equation*} \d\theta_1\wedge\d\theta_2-\d\theta_3\wedge\d\theta_4-\d\theta\wedge\d x\in \Omega^2(\T^2\times \T^2\times \mathbb{S}^1\times \R).
\end{equation*} Furthermore, this symplectic groupoid acts in a Hamiltonian fashion along:
\begin{equation*} J:(\T^2\times \mathbb{S}^1\times \R, \d\theta_1\wedge \d\theta_2-\d\theta\wedge \d x)\to \T^2\times_{\Z_2}\R, \quad (e^{i\theta_1},e^{i\theta_2},e^{i\theta},x)\mapsto [e^{i\theta_1},e^{i\theta_2},x],
\end{equation*} with the action given by:
\begin{equation*} [e^{i\theta_1},e^{i\theta_2},e^{i\theta_3}, e^{i\theta_4},e^{i\theta},x]\cdot (e^{i\theta_3},e^{i\theta_4},e^{i\phi},x)=(e^{i\theta_1},e^{i\theta_2},e^{i(\theta+\phi)},x).
\end{equation*}
This action is free and its orbit space is canonically diffeomorphic to $\R$. The canonical Hamiltonian stratification on the orbit space consists of three strata: $\{x>0\}$, $\{x<0\}$ and the origin $\{x=0\}$, because the isotropy groups of (\ref{gpoiddisconprincpartex}) at points in $\T^2\times_{\Z_2}\R$ with $x\neq 0$ are isomorphic to $\mathbb{S}^1$, whilst those at points with $x=0$ are isomorphic to $\Z_2\ltimes \mathbb{S}^1$. So, we see that its regular part is disconnected. \end{ex}
The following theorem provides a criterion that does ensure connectedness of the regular part.
\begin{thm}\label{printypthm} Let $(\G,\Omega)\rightrightarrows M$ be a proper symplectic groupoid and suppose that we are given a Hamiltonian $(\G,\Omega)$-action along $J:(S,\omega)\to M$. The following conditions are equivalent.
\begin{itemize}\item[a)] For every $p\in S$ that belongs to a codimension one stratum of the canonical Hamiltonian stratification $\S_\textrm{Ham}(S)$, the action of $\G_p$ on $\g_p^0$ is non-trivial.
\item[b)] The regular part $S^\textrm{princ}$ of $\S_\textrm{Gp}(S)$ (as in Subsection \ref{liedescrpregpartsec}) does not contain codimension one strata of $\S_\textrm{Ham}(S)$. 
\end{itemize} Furthermore, if $\underline{S}$ is connected and the above conditions hold, then $\underline{S}^\textrm{princ}_\textrm{Ham}$ is connected as well. If in addition the orbits of the action are connected, then $S^\textrm{princ}_\textrm{Ham}$ is also connected.
\end{thm}
\begin{proof} As in the proof of Theorem \ref{printypthminf} it follows that if $p\in S$ belongs to a codimension one stratum of $\S_\textrm{Ham}(S)$, then the action of $\G_p$ on $\S\No_p$ is trivial. So, by Proposition \ref{regpartalg}$a$, for such $p\in S$ the action of $\G_p$ on $\g_p^0$ is trivial if and only if $p\in S^\textrm{princ}$. From this it is clear that the two given conditions are equivalent. Furthermore, if $\underline{S}$ is connected, then by the principal type theorem for proper Lie groupoids (see Example \ref{princtypliegpoidex}), $\underline{S}^\textrm{princ}$ is connected. So, in light of Corollary \ref{prinJstrat} and Lemma \ref{cod1stratlem}, $\underline{S}^\textrm{princ}_\textrm{Ham}$ will be connected if in addition $\underline{S}^\textrm{princ}$ does not contain codimension one strata of $\S_\textrm{Ham}(\underline{S})$, or equivalently, if in addition condition $b$ holds.
\end{proof}

The proposition below gives a criterion for the conditions in the previous theorem to hold. 
\begin{prop} If $p\in S$ belongs to a codimension one stratum of $\S_\textrm{Ham}(S)$ and the coadjoint orbits of $\G_{J(p)}$ are connected, then the action of $\G_p$ on $\g_p^0$ is non-trivial.
\end{prop}  
\begin{proof} The same reasoning as in the proof of Theorem \ref{printypthminf} shows that if the action of $\G_p$ on $\g_p^0$ would be trivial, then the identity component of $\G_{J(p)}$ would fix all of $\g_p^0$. By connectedness of its coadjoint orbits, the entire group $\G_{J(p)}$ would then fix all of $\g_p^0$, which, as in the aforementioned proof, leads to a contradiction. 
\end{proof}
\begin{cor}\label{printypHamGspacethm} Let $G$ be a compact and connected Lie group and let $J:(S,\omega)\to \g^*$ be a connected Hamiltonian $G$-space. Then $\underline{S}^\textrm{princ}_\textrm{Ham}$ is connected.
\end{cor}
\begin{proof} For $G$ compact and connected, the isotropy groups of the coadjoint $G$-action are connected. So, the previous proposition ensures that condition $a$ in Theorem \ref{printypthm} is satisfied.
\end{proof} 

\begin{ex} Let $G$ be a compact and connected Lie group and let $J:(S,\omega)\to \g^*$ be a connected Hamiltonian $G$-space. We return to the partition in Example \ref{hamGspex}. This comes with a partial order, defined as follows. If $\underline{S}_p$ and $\underline{S}_q$ denote the members through the respective orbits $\O_p$ and $\O_q$, then by definition: 
\begin{equation*} \underline{S}_{p}\leq \underline{S}_{q} \iff (G_{J(q)},G_{q})\text{ is conjugate in $G$ to pair of subgroups of }(G_{J(p)},G_{p}). 
\end{equation*}
In analogy with the principal orbit type theorem (see Example \ref{prinorbthm}), this partial order has a greatest element, namely $\underline{S}^\textrm{princ}_\textrm{Ham}$. To see this, notice that from the normal form theorem as in Remark \ref{sharphamliegpactrem} it follows that every $\O_p\in \underline{S}$ admits an open neighbourhood $\underline{U}$ with the property that $\underline{S}_{p}\leq \underline{S}_{q}$ for all $\O_q\in \underline{U}$. From this and the fact that $\underline{S}^\textrm{princ}_\textrm{Ham}$ is connected and dense in $\underline{S}$, it follows that it is indeed a member of the partition in Example \ref{hamGspex}, and that it is the greatest element with respect to the above partial order.
\end{ex}

To end with, we note that the following generalization of \cite[Theorem 5.9, Remark 5.10]{LeSj} holds.
\begin{thm} Let $\L$ be a leaf of $\G$ and suppose that $\underline{S}_\L$ is connected. Then the regular part $\underline{S}^\textrm{princ}_\L$ of $\S_\textrm{Ham}(\underline{S}_\L)$ is connected as well. 
\end{thm}
\begin{proof} Since $\underline{S}^\textrm{princ}_\L$ is dense in $\underline{S}_\L$ and $\underline{S}_\L$ is connected, it is enough to show that every point in $\underline{S}_\L$ admits an open neighbourhood that intersects $\underline{S}_\L^\textrm{princ}$ in a connected subspace. To this end, let $\O_p\in \underline{S}_\L$, let $H=\G_p$ and $V=\S\No_p$. Consider a Hamiltonian Morita equivalence as in the proof of Proposition \ref{weakwhitstratthm}, so that the induced homeomorphism of orbit spaces identifies an open $\underline{U}$ around $\O_p$ in $\underline{S}$ with an open $\underline{B}_{\h^0\oplus V}$ around the origin in $(\h^0\oplus V)/H$. Let $B$ be the intersection of $B_{\h^0\oplus V}$ with $V$ and consider the Hamiltonian $H$-space:
\begin{equation*} J_B=J_V\vert_B:(B,\omega_V)\to \h^*.
\end{equation*} Then $\underline{U}\cap \underline{S}_\L$ is identified with $J_B^{-1}(0)/H$, and $\underline{U}\cap \underline{S}_\L^\textrm{princ}$ is identified with the principal part of $J_B^{-1}(0)/H$ (as follows from Morita invariance of the partitions by isomorphism types). Since $J_B^{-1}(0)$ is star-shaped with respect to the origin, $J_B^{-1}(0)/H$ is connected and hence, by \cite[Theorem 5.9, Remark 5.10]{LeSj}, so is its principal part. So, we have found the desired neighbourhood of $\O_p$. 
\end{proof}
\subsubsection{Relations amongst the regular parts}
In this last subsection we discuss another relationship between the regular parts of the various stratifications, starting with the following observation.
\begin{prop} Suppose that $J$ is a submersion on $S^\textrm{reg}$. Then the various regular and principal parts on $S$, $M$, $\underline{S}$ and $\underline{M}$ are related as:
\begin{equation*} S^\textrm{reg}_\textrm{Ham}=S^\textrm{reg}\cap J^{-1}(M^\textrm{reg}), \quad \quad S^\textrm{princ}_\textrm{Ham}=S^\textrm{princ}\cap J^{-1}(M^\textrm{princ}), \quad\quad \underline{S}^\textrm{princ}_\textrm{Ham}=\underline{S}^\textrm{princ}\cap (\underline{J})^{-1}(\underline{M}^\textrm{princ}).
\end{equation*}  
\end{prop}
\begin{proof} We prove the equality for $S^\textrm{reg}_\textrm{Ham}$; the others are proved similarly. Let $p\in S$, $x=J(p)$ and consider the strata 
$\Sigma^\textrm{Ham}_p\in \S_\textrm{Ham}^\textrm{inf}(S)$, $\Sigma^\textrm{Gp}_p\in \S_\textrm{Gp}^\textrm{inf}(S)$ and $\Sigma^\textrm{Gp}_x\in \S_\textrm{Gp}^\textrm{inf}(M)$ through $p$ and $x$. Then
\begin{equation}\label{eqregpartrel} \Sigma^\textrm{Ham}_p\subset \Sigma^\textrm{Gp}_p\cap J^{-1}\left(\Sigma^\textrm{Gp}_{x}\right)
\end{equation} is open in the right-hand space. This, combined with the fact that $J:S^\textrm{reg}\to M$ is open and continuous, implies that $\Sigma^\textrm{Ham}_p$ is open at $p$ in $S$ if and only if $\Sigma^\textrm{Gp}_p$ is open at $p$ in $S$ and $\Sigma^\textrm{Gp}_{x}$ is open at $x$ in $M$. In light of Proposition \ref{regpartopen} this means that:
\begin{equation*} S^\textrm{reg}_\textrm{Ham}=S^\textrm{reg}\cap J^{-1}(M^\textrm{reg}),
\end{equation*} as claimed.
\end{proof}
In general (that is, if $J$ is not submersive on $S^\textrm{reg}$) one would hope for a similar result. Since the image of $J$ need not intersect $M^\textrm{reg}$, one however needs an appropriate replacement for it. The proposition below gives a sufficient condition for the existence of such a replacement. 
\begin{prop}\label{printypdescr} Suppose that $S^\textrm{reg}_\textrm{Ham}$ is connected. Then there is a unique stratum $\Sigma\in \S_\textrm{Gp}^\textrm{inf}(M)$ with the property that $\Sigma\cap J(S)$ is open and dense in $J(S)$. Moreover, it holds that:
\begin{equation*} S^\textrm{reg}_\textrm{Ham}=S^\textrm{reg}\cap J^{-1}(\Sigma),
\end{equation*} and $J^{-1}(\Sigma)$ is connected, open and dense in $S$. Similar conclusions hold for the principal part on $S$ (resp. $\underline{S}$), under the assumption that $S^\textrm{princ}_\textrm{Ham}$ (resp. $\underline{S}^\textrm{princ}_\textrm{Ham}$) is connected. 
\end{prop}
\begin{proof} Again, we prove the result only for $S^\textrm{reg}_\textrm{Ham}$ since the other proofs are analogous. We use the notation introduced in the proof of the previous proposition. Consider $R\subset J(S)$ defined as:
\begin{equation*} R:=\{x\in J(S)\mid \Sigma^\textrm{Gp}_x\cap J(S)\text{ is open in }J(S)\}.
\end{equation*} We claim that $S^\textrm{reg}_\textrm{Ham}=S^\textrm{reg}\cap J^{-1}(R)$, that $R$ is connected, open and dense in $J(S)$ and that $J^{-1}(R)$ is connected, open and dense in $S$. The desired stratum $\Sigma$ is then the unique stratum containing $R$. To see that our claim holds, notice first $R$ is clearly open in $J(S)$, and so $J^{-1}(R)$ is open in $S$. Moreover, by continuity of $J$ and $(\ref{eqregpartrel})$ we find that $S^\textrm{reg}\cap J^{-1}(R)$ is a union of strata of $\S_\textrm{Ham}^\textrm{inf}(S)$ contained in $S^\textrm{reg}_\textrm{Ham}$. So, if $S^\textrm{reg}_\textrm{Ham}$ is connected, then  $S^\textrm{reg}\cap J^{-1}(R)$ must coincide with $S^\textrm{reg}_\textrm{Ham}$. Then since $S^\textrm{reg}_\textrm{Ham}$ is dense in $S$, so is $J^{-1}(R)$, and furthermore, $R$ must be dense in $J(S)$. Finally, because $S^\textrm{reg}_\textrm{Ham}$ is connected and dense in $J^{-1}(R)$, it follows that $J^{-1}(R)$ is connected and hence $R$ is connected as well. This proves our claim.
\end{proof}

\begin{ex} Let $G$ be a compact and connected Lie group and let $J:(S,\omega)\to \g^*$ be a connected Hamiltonian $G$-space. 
Let $T\subset G$ be a maximal torus, $\t^*_+$ a choice of closed Weyl chamber in $\t^*$ and $J_+(S):=J(S)\cap \t^*_+$, where $\t^*$ is canonically identified with the $T$-fixed point set $(\g^*)^T$ in $\g^*$. According to \cite[Theorem 3.1]{LeMeToWo}, there is a unique open face of the Weyl chamber (called the principal face) that intersects $J_+(S)$ in a dense subset of $J_+(S)$. Combining Corollary \ref{printypHamGspacethm} with Proposition \ref{printypdescr}, we recover the existence of the principal face. 
\end{ex}

\subsection{The Poisson structure on the orbit space}\label{poisstratthmsec}
\subsubsection{Poisson structures on reduced differentiable spaces and Poisson stratifications} In this section we discuss the Poisson structure on the orbit space of a Hamiltonian action and discuss basic Poisson geometric properties of the various stratifications associated to such an action. First, we give some more general background. 
\begin{defi} A \textbf{Poisson reduced ringed space} is a reduced ringed space $(X,\O_X)$ together with a Poisson bracket $\{\cdot,\cdot\}$ on the structure sheaf $\O_X$. A \textbf{morphism of Poisson reduced ringed spaces} is a morphism of reduced ringed spaces: \begin{equation*} \phi:(X,\O_X)\to (Y,\O_Y)\end{equation*} with the property that for every open $U$ in $Y$:
\begin{equation*} \phi^*:\left(\O_Y(U),\{\cdot,\cdot\}_U\right)\to \left(\O_X(\phi^{-1}(U)),\{\cdot,\cdot\}_{\phi^{-1}(U)}\right)
\end{equation*} is a Poisson algebra map. We will also call such $\phi$ simply a \textbf{Poisson map}. When $(X,\O_X)$ is a reduced differentiable space, we call $(X,\O_X,\{\cdot,\cdot\})$ a \textbf{Poisson reduced differentiable space}.
\end{defi} 
\begin{rem}\label{globpoisbrackrem} The Poisson reduced ringed spaces in this paper will all be Hausdorff and second countable reduced differentiable spaces. For such reduced ringed spaces $(X,\O_X)$ the data of a Poisson bracket on the sheaf $\O_X$ is the same as the data of a Poisson bracket on the $\R$-algebra $\O_X(X)$, so that when convenient we can restrict attention to the Poisson algebra of globally defined functions. This follows as for manifolds, using bump functions in $\O_X(X)$ (cf. Remark \ref{partofunityreddiffbsp}). 
\end{rem}
Next, we turn to subspaces and stratifications of Poisson reduced differentiable spaces. 
\begin{defi}\label{poissubmandef} Let $(X,\O_X,\{\cdot,\cdot\}_X)$ be a Poisson reduced differentiable space. A locally closed subspace $Y$ of $(X,\O_X)$ is a \textbf{Poisson reduced differentiable subspace} if the induced structure sheaf $\O_Y$ admits a (necessarily unique) Poisson bracket for which the inclusion of $Y$ into $X$ becomes a Poisson map. If $Y$ is also a submanifold of $(X,\O_X)$, then we call it a \textbf{Poisson submanifold}. 
\end{defi} 
As in \cite{FeOrRa}, we use the following definition.
\begin{defi}\label{poisstratdef} Let $(X,\O_X,\{\cdot,\cdot\}_X)$ be a Hausdorff and second countable Poisson reduced differentiable space. A \textbf{Poisson stratification} of $(X,\O_X,\{\cdot,\cdot\}_X)$ is a stratification $\S$ of $(X,\O_X)$ with the property that every stratum is a Poisson submanifold. We call $(X,\O_X,\{\cdot,\cdot\}_X,\S)$ a \textbf{Poisson stratified space}.  A \textbf{Symplectic stratified space} is a Poisson stratified space for which the strata are symplectic. A \textbf{morphism of Poisson stratified spaces} is a morphism of the underlying stratified spaces that is simultaneously a morphism of the underlying Poisson reduced ringed spaces.
\end{defi}
As for manifolds, we have the following useful characterization.
\begin{prop}\label{poisidealcharprop} Let $(X,\O_X,\{\cdot,\cdot\}_X)$ be a Hausdorff and second countable Poisson reduced differentiable space and let $Y$ be a locally closed subspace. Then $Y$ is a Poisson reduced differentiable subspace if and only if the vanishing ideal $\mathcal{I}_Y(X)$ in $\O_X(X)$ (consisting of $f\in \O_X(X)$ such that $f\vert_Y=0$) is a Poisson ideal (meaning that: if $f,h\in \O_X(X)$ and $h\vert_Y=0$, then $\{f,h\}_X\vert_Y=0$).
\end{prop}
\begin{proof} The forward implication is immediate. For the backward implication the same argument as for manifolds applies: given $f,h\in \O_Y(Y)$, by Proposition \ref{globchar} we can choose extensions $\widehat{f},\widehat{h}\in \O_X(U)$ of $f$ and $h$ defined on some open neighbourhood $U$ of $Y$ and set:
\begin{equation*} \{f,h\}_Y:=\{\widehat{f},\widehat{h}\}_U\vert_Y.
\end{equation*} This does not depend on the choice of extensions, because for any open $U$ in $X$ the ideal $\mathcal{I}_Y(U)$ in $\O_X(U)$, consisting of functions that vanish on $U\cap Y$, is a Poisson ideal. Indeed, this follows from the assumption that $\mathcal{I}_Y(X)$ is a Poisson ideal in $\O_X(X)$, using bump functions (cf. Remark \ref{partofunityreddiffbsp}). By construction, $\{\cdot,\cdot\}_Y$ defines a Poisson bracket on $\O_Y(Y)$ (and hence on $\O_Y$, by Remark \ref{globpoisbrackrem}) for which the inclusion of $Y$ into $X$ becomes a Poisson map.
\end{proof}
\subsubsection{The Poisson algebras of invariant functions} Next, we turn to the definition of the Poisson bracket on the orbit space of a Hamiltonian action, starting with the following observation. 
\begin{prop}\label{poisalginvfunprop} Let $(\G,\Omega)$ be a symplectic groupoid and suppose that we are given a Hamiltonian $(\G,\Omega)$-action along $J:(S,\omega)\to M$. The algebra of invariant smooth functions:
\begin{equation*} C^\infty(S)^\G=\{f\in C^\infty(S)\mid f(g\cdot p)=f(p),\text{ } \forall (g,p)\in \G\times_MS\}.
\end{equation*} is a Poisson subalgebra of $(C^\infty(S),\{\cdot,\cdot\}_\omega)$.  
\end{prop}
\begin{proof} Although this is surely known, let us give a proof. Let $f,h\in C^\infty(S)^\G$ and let $\Phi_f$ denote the Hamiltonian flow of $f$. Using the lemma below we find that for all $(g,p)\in \G\times_MS$:
\begin{align*} \{f,h\}_\omega(g\cdot p)&=\left.\frac{\d}{\d t}\right|_{t=0} h(\Phi_f^t(g\cdot p))\\
&=\left.\frac{\d}{\d t}\right|_{t=0} h(g\cdot \Phi_f^t(p))\\
&=\left.\frac{\d}{\d t}\right|_{t=0} h(\Phi_f^t(p))=\{f,h\}_\omega(p),\end{align*} so that $\{f,g\}_\omega\in C^\infty(S)^\G$, as required. 
\end{proof}
Here we used the following lemma, which will also be useful later.
\begin{lemma}\label{hamflowinvfunlem} Let $f\in C^\infty(S)^\G$ and let $\Phi_f$ denote its Hamiltonian flow. Then, for every $t\in \R$, the domain $U_t$ and the image $V_t$ of $\Phi_f^t$ are $\G$-invariant and $\Phi^t_f$ is an isomorphism of Hamiltonian $(\G,\Omega)$-spaces:  
\begin{center}\begin{tikzcd} (U_t,\omega)\arrow[rr,"\Phi_f^t"]\arrow[rd,"J"'] & & (V_t,\omega) \arrow[ld,"J"]\\
& M & 
\end{tikzcd}
\end{center} 
\end{lemma}
\begin{proof} Invariance of $f$ implies that $X_f(p)\in T_p\O^\omega$ for all $p\in S$. From this and Proposition \ref{infmomact}$a$ it follows that $J(\Phi_f^t(p))=J(p)$ for any $p\in S$ and any time $t$ at which the flow through $p$ is defined. So, for any $(g,p)\in \G\times_MS$ we can consider the curve: 
\begin{equation*} t\mapsto \left(g,\Phi_f^t(p)\right) \in \G\times_MS.
\end{equation*} Given such $(g,p)$, let $v\in T_{g\cdot p}S$ and take a tangent vector $\hat{v}$ to $\G\times_M S$ at $(g,p)$ such that $\d m(\widehat{v})=v$. Then we find:
\begin{equation*} \omega\left(\left.\frac{\d}{\d t}\right|_{t=0} g\cdot\Phi^t_f(p),v\right)=(m_S^*\omega)\left(\left.\frac{\d}{\d t}\right|_{t=0}\left(g,\Phi_f^t(p)\right),\widehat{v}\right).
\end{equation*} Using (\ref{hammultcond}) this is further seen to be equal to:
\begin{equation*} \omega\left(X_f(p),\d (\textrm{pr}_S)(\widehat{v})\right)=\d (f\circ \textrm{pr}_S)(\widehat{v})=\d f(v),
\end{equation*} where in the last step we used invariance of $f$. As this holds for all such $v$, we deduce that:
\begin{equation*} X_f(g\cdot p)=\left.\frac{\d}{\d t}\right|_{t=0} g\cdot \Phi_f^t(p).
\end{equation*} This being true for all $p$ in the fiber of $J$ over $s(g)$, and in particular for all points on a maximal integral curve of $X_f$ starting in this fiber, it follows that the maximal integral curve of $X_f$ through $g\cdot p$ is given by $t\mapsto g\cdot \Phi_f^t(p)$. The lemma readily follows from this.   
\end{proof}

Given a proper symplectic groupoid $(\G,\Omega)$ and a Hamiltonian $(\G,\Omega)$-action along $J:(S,\omega)\to M$, the Poisson bracket $\{\cdot,\cdot\}_\omega$ on the algebra $C^\infty(S)^\G$ in Proposition \ref{poisalginvfunprop} gives the orbit space $(\underline{S},\mathcal{C}^\infty_{\underline{S}})$ the structure of a Poisson reduced differentiable space, with Poisson bracket determined by the fact that the orbit projection becomes a Poisson map. Moreover, for each leaf $\L$ of $\G$ in $M$, the reduced space $\underline{S}_\L$ is a Poisson reduced differentiable subspace. Indeed, identifying the algebra of globally defined smooth functions on $\underline{S}$ with $C^\infty(S)^\G$, the vanishing ideal of $\underline{S}_\L$ is identified with the ideal $\mathcal{I}^{\G}_\L$ of invariant smooth functions that vanish on $J^{-1}(\L)$, which is a Poisson ideal by the proposition below. This observation is due to \cite{ArCuGot} in the setting of Hamiltonian group actions. 
\begin{prop}\label{poisidealpropredsp} The ideal $\mathcal{I}_\L^\G$ is a Poisson ideal of $(C^\infty(S)^\G,\{\cdot,\cdot\}_\omega)$. 
\end{prop} 
\begin{proof} If $f\in C^\infty(S)^\G$ and $h\in\mathcal{I}_\L^\G$ then by Lemma \ref{hamflowinvfunlem} the Hamiltonian flow of $f$ starting at $p\in J^{-1}(\L)$ is contained in a single fiber of $J$, and hence in $J^{-1}(\L)$, so that $\{f,h\}_\omega(p)=0$. 
\end{proof}
We will denote the respective Poisson structures on $(\underline{S},\mathcal{C}^\infty_{\underline{S}})$ and $(\underline{S}_\L,\mathcal{C}^\infty_{\underline{S}_\L})$ by $\{\cdot,\cdot\}_{\underline{S}}$ and $\{\cdot,\cdot\}_{\underline{S}_\L}$. 
\subsubsection{The Poisson stratification theorem} Now, we move to the main theorem of this section. 
\begin{thm}\label{poisstratthm} Let $(\G,\Omega)\rightrightarrows M$ be a proper symplectic groupoid and suppose that we are given a Hamiltonian $(\G,\Omega)$-action along $J:(S,\omega)\to M$. Then the following hold.
 \begin{itemize} \item[a)] The stratification $\S_\textrm{Gp}(\underline{S})$ is a Poisson stratification of the orbit space: \begin{equation*} (\underline{S},\mathcal{C}^\infty_{\underline{S}},\{\cdot,\cdot\}_{\underline{S}}).\end{equation*}
 \item[b)] The stratification $\S_\textrm{Ham}(\underline{S})$ is a Poisson stratification of the orbit space: \begin{equation*}(\underline{S},\mathcal{C}^\infty_{\underline{S}},\{\cdot,\cdot\}_{\underline{S}}),\end{equation*} the strata of which are regular Poisson submanifolds.
\item[c)] For each leaf $\L$ of $\G$ in $M$, the stratification $\S_\textrm{Ham}(\underline{S}_\L)$ is a symplectic stratification of the reduced space at $\L$: \begin{equation*} (\underline{S}_\L,\mathcal{C}^\infty_{\underline{S}_\L},\{\cdot,\cdot\}_{\underline{S}_\L}).
\end{equation*}
\end{itemize}
These are related as follows. First of all, the inclusions of smooth stratified spaces:
\begin{equation*} \left(\underline{S}_\L,\mathcal{C}^\infty_{\underline{S}_\L}, \{\cdot,\cdot\}_{\underline{S}_\L},\S_\textrm{Ham}(\underline{S}_\L)\right)\hookrightarrow \left(\underline{S},\mathcal{C}^\infty_{\underline{S}},\{\cdot,\cdot\}_{\underline{S}},\S_\textrm{Ham}(\underline{S})\right)\hookrightarrow \left(\underline{S},\mathcal{C}^\infty_{\underline{S}},\{\cdot,\cdot\}_{\underline{S}},\S_\textrm{Gp}(\underline{S})\right)
\end{equation*} are Poisson maps that map symplectic leaves onto symplectic leaves. 
Moreover, for each stratum $\underline{\Sigma}_S\in \S_\textrm{Ham}(\underline{S})$, the symplectic leaves in $\underline{\Sigma}_S$ are the connected components of the fibers of the constant rank map $\underline{J}:\underline{\Sigma}_S\to \underline{\Sigma}_M$, where $\underline{\Sigma}_M\in \S_\textrm{Gp}(\underline{M})$ is the stratum such that $\underline{J}(\underline{\Sigma}_S)\subset \underline{\Sigma}_M$.
\end{thm} 

\begin{proof}[Proof of Theorem \ref{poisstratthm}] Let $\underline{\Sigma}\in \S_\textrm{Ham}(\underline{S})$ be a stratum and let $\Sigma:=q^{-1}(\underline{\Sigma})$ where $q:S\to \underline{S}$ denotes the orbit projection. Identifying the algebra of globally defined smooth functions on $\underline{S}$ with $C^\infty(S)^\G$, the vanishing ideal of $\underline{\Sigma}$ is identified with the ideal: \begin{equation*} \mathcal{I}_\Sigma^\G=\{f\in C^\infty(S)^\G\mid f\vert_{\Sigma}=0\}.
\end{equation*} This is a Poisson ideal of $C^\infty(S)^\G$, for if $f\in C^\infty(S)^\G$ and $h\in \mathcal{I}_\Sigma^\G$, then as an immediate consequence of Lemma \ref{hamflowinvfunlem}, the Hamiltonian flow of $f$ leaves $\Sigma$ invariant and therefore: \begin{equation*}
\{f,h\}_\omega\vert_{\Sigma}=(\L_{X_f} h)\vert_{\Sigma}=0.
\end{equation*} By Proposition \ref{poisidealcharprop} this means that $\underline{\Sigma}$ is a Poisson submanifold (in the sense of Definition \ref{poissubmandef}). So, $\S_\textrm{Ham}(\underline{S})$ is a Poisson stratification of the orbit space. By the same reasoning it follows that the stratifications in statements $a$ and $c$ are Poisson stratifications. From the construction of the Poisson brackets on the orbit space and the reduced spaces, it is immediate that the inclusions given in the statement of the theorem are Poisson. Hence, each stratum of $\S_\textrm{Gp}(\underline{S})$ is partitioned into Poisson submanifolds by strata of $\S_\textrm{Ham}(\underline{S})$ and each stratum of $\S_\textrm{Ham}(\underline{S})$ is partitioned into Poisson submanifolds by strata of $\S_\textrm{Ham}(\underline{S}_\L)$, for varying $\L\in \underline{M}$. If $(N,\pi)$ is a Poisson manifold partitioned by Poisson submanifolds, then the symplectic leaves of each of the Poisson submanifolds in the partition are symplectic leaves of $(N,\pi)$. This follows from the fact that each symplectic leaf of a Poisson submanifold is an open inside a symplectic leaf of the ambient Poisson manifold. Therefore, each of the inclusions given in the statement of the theorem indeed maps symplectic leaves onto symplectic leaves. \\

 It remains to see that for each stratum $\underline{\Sigma}_S\in \S_\textrm{Ham}(\underline{S})$ the foliation by symplectic leaves of the Poisson structure $\pi_{\underline{\Sigma}_S}$ on $\underline{\Sigma}_S$ coincides with that by the connected components of the fibers of the constant rank map $\underline{J}:\underline{\Sigma}_S\to \underline{\Sigma}_M$, because the claims on regularity and non-degeneracy made in statements $b$ and $c$ follow from this as well. To this end, we have to show that for every orbit $\O\in \underline{\Sigma}_S$ the tangent space to the symplectic leaf at $\O$ coincides with $\ker(\d\underline{J}\vert_{\underline{\Sigma}_S})_\O$. Here the language of Dirac geometry comes in useful. We refer the reader to \cite{Cou,Bu} for background on this. Let $\Sigma_S=q^{-1}(\underline{\Sigma}_S)$ and consider the pre-symplectic form:\begin{equation*} \omega_{{\Sigma}_S}:=\omega\vert_{\Sigma_S}\in \Omega^2(\Sigma_S).
\end{equation*} We claim that the orbit projection: 
\begin{equation}\label{orbprojdirac} q:(\Sigma_S,\omega_{\Sigma_S})\to (\underline{\Sigma}_S,\pi_{\underline{\Sigma}_S})
\end{equation} is a forward Dirac map. To see this, we will use the fact that a map $\phi:(Y,\omega_Y)\to (N,\pi_N)$ from a pre-symplectic manifold into a Poisson manifold is forward Dirac if for every $f\in C^\infty(N)$ there is a vector field $X_{\phi^*f}\in \X(Y)$ such that:  
\begin{equation*} \iota_{X_{\phi^*f}}\omega_Y=\d (\phi^*f) \quad\&\quad \phi_*(X_{\phi^*f})=X_f. 
\end{equation*}
 Given an $f\in C^\infty(\underline{\Sigma}_S)$, choose a smooth extension $\widehat{f}$ defined an open $\underline{U}$ around $\underline{\Sigma}_S$ in $\underline{S}$. Because $q^*\widehat{f}$ is $\G$-invariant, its Hamiltonian flow leaves $\Sigma_S$ invariant (as before). Therefore, we can consider: \begin{equation*} X_{q^*f}:=(X_{q^*\widehat{f}})\vert_{\Sigma_S}\in \X(\Sigma_S)\end{equation*} and as is readily verified this satisfies: 
\begin{equation*} \iota_{(X_{q^*f})}\omega_{\Sigma_S}=\d (q^*f) \quad \& \quad q_*(X_{q^*f})=X_f. 
\end{equation*} So (\ref{orbprojdirac}) is indeed a forward Dirac map. From the equality of Dirac structures $L_{\pi_{\underline{\Sigma}_S}}=q_*(L_{\omega_{\Sigma_S}})$ we read off that the tangent space to the symplectic leaf at an orbit $\O$ through $p\in S$ is given by:
\begin{equation}\label{tansplfeq} \frac{T_p\O^{(\omega_{\Sigma_S})}}{T_p\O\cap T_p\O^{(\omega_{\Sigma_S})}}\subset \frac{T_p{\Sigma_S}}{T_p\O}=T_\O(\underline{\Sigma}_S).
\end{equation} It follows from Proposition \ref{infmomact}$a$ that $T_p\O^{(\omega_{\Sigma_S})}=\ker(\d J\vert_{\Sigma_S})_p$. This implies that (\ref{tansplfeq}) equals:
\begin{equation*} \frac{\ker(\d J\vert_{\Sigma_S})_p}{T_p\O\cap \ker(\d J\vert_{\Sigma_S})_p}=\ker(\d\underline{J}\vert_{\underline{\Sigma}_S})_\O\subset T_\O(\underline{\Sigma}_S),
\end{equation*} as we wished to show. 
\end{proof}
From the proof we also see:
\begin{cor}\label{orbprojstratdirac} For every stratum $\underline{\Sigma}_S\in \S_\textrm{Ham}(\underline{S})$, the orbit projection $(\ref{orbprojdirac})$ is forward Dirac. The same holds for the strata of $\S_\textrm{Gp}(\underline{S})$.
\end{cor}
\subsubsection{Dimension of the symplectic leaves} In the remainder of this section we make some further observations on the Poisson geometry of the orbit space, starting with:
\begin{prop}\label{dimsymplvslocnondecrprop} Let $(\G,\Omega)\rightrightarrows M$ be a proper symplectic groupoid and suppose that we are given a Hamiltonian $(\G,\Omega)$-action along $J:(S,\omega)\to M$. The dimension of the symplectic leaves in the orbit space $\underline{S}$ is locally non-decreasing. That is, every $\O\in \underline{S}$ admits an open neighbourhood $\underline{U}$ in $\underline{S}$ such that any symplectic leaf intersecting $\underline{U}$ has dimension greater than or equal to that of the symplectic leaf through $\O$. 
\end{prop}
\begin{proof} First, let us make a more general remark. Let $p\in S$, let $\underline{\Sigma}_S\in \S_\textrm{Ham}(\underline{S})$ be the stratum through $\O_p$ and let $\underline{\Sigma}_M\in \S_\textrm{Gp}(\underline{M})$ be such that $\underline{J}(\underline{\Sigma}_S)\subset \underline{\Sigma}_M$. From a Hamiltonian Morita equivalence as in the proof of Proposition \ref{eqcharhammortyp} we obtain (via Proposition \ref{transgeommap}$a$) an identification of smooth maps between $\underline{J}:\underline{\Sigma}_S\to \underline{\Sigma}_M$ near $\O_p$ and the map (\ref{mommaplocmodcentisotyp}) near the origin. Therefore, the dimension of the fibers of the former map is equal to that of the latter, which is $\dim(\S\No_p^{\G_p})$, or equivalently: $\dim(\ker(\underline{\d J}_p)^{\G_p})$ (see the proof of Proposition \ref{normrepham}$b$). In view of Theorem \ref{poisstratthm}, this is also the dimension of the symplectic leaf through $\O_p$. To prove the proposition, it is therefore enough to show that each $p\in S$ admits an invariant open neighbourhood $U$ with the property that $\ker(\underline{\d J}_p)^{\G_p}$ has dimension less than or equal to that of $\ker(\underline{\d J}_q)^{\G_q}$ for each $q\in U$. To this end, given $p\in S$, choose an invariant open neighbourhood $U$ for which there is a Hamiltonian Morita equivalence as in the proof of Proposition \ref{eqcharhammortyp}. Then $U$ has the desired property. Indeed, in light of Proposition \ref{transgeommap}$c$, it suffices to show (using the notation of the proof of Proposition \ref{eqcharhammortyp}) that for each $\alpha\in \h^0$ and $v\in V$:
\begin{equation*} \dim(V^H)\leq \dim(\ker(\underline{\d J}_\p)_{(\alpha,v)}^{H_{(\alpha,v)}}).
\end{equation*} To this end, consider the linear map:
\begin{equation}\label{incldimcountsymleavesmap} V^H\to \ker(\underline{\d J}_\p)_{(\alpha,v)}^{H_{(\alpha,v)}}, \quad w\mapsto \left[\left.\frac{\d}{\d t}\right\vert_{t=0} (\alpha,v+tw)\right]. 
\end{equation} Note here that this indeed takes values in $\ker(\underline{\d J}_\p)_{(\alpha,v)}$, because for all $w\in V^H$:
\begin{equation*} \left.\frac{\d}{\d t}\right\vert_{t=0} v+tw \in \ker(\d J_V)_{v},
\end{equation*} as follows from (\ref{quadmomfixset}). To complete the proof, we will now show that (\ref{incldimcountsymleavesmap}) is injective. Suppose that: 
\begin{equation*} \left.\frac{\d}{\d t}\right\vert_{t=0} (\alpha,v+tw) \in T_{(\alpha,v)}\O.  
\end{equation*} Then:
\begin{equation*} \left.\frac{\d}{\d t}\right\vert_{t=0} v+tw \in T_{v}\O\cap T_{v}(v+V^H).
\end{equation*} Because $H$ is compact, $V^H$ admits an $H$-invariant linear complement in $V$, which implies that: 
\begin{equation*} T_{v}\O\cap T_{v}(v+V^H)=0.
\end{equation*} Therefore $w=0$, proving that (\ref{incldimcountsymleavesmap}) is indeed injective.
\end{proof}
\begin{rem}\label{vectspsymplfisofxpt} In the above proof we have seen that the dimension of the symplectic leaf $(\L,\omega_\L)$ through $\O_p$ is $\dim(\S\No_p^{\G_p})$. In fact, there is a canonical isomorphism of symplectic vector spaces:
\begin{equation*} (T_{\O_p}\L,(\omega_{\L})_{\O_p})\cong(\S\No_p^{\G_p},\omega_p).
\end{equation*} 
\end{rem}
\subsubsection{Morita invariance of the Poisson stratifications} We end this section with:

\begin{prop}\label{hammorinvpoisbrack} Each of the stratifications in Theorem \ref{poisstratthm} is invariant under Hamiltonian Morita equivalence, as Poisson stratification.  
\end{prop}
\begin{proof} Suppose we are given a Morita equivalence between two Hamiltonian actions of two proper symplectic groupoids; we use the notation of Definition \ref{moreqdefLie} and \ref{moreqdefHam}. It is immediate that the induced homeomorphism $h_Q$ (see Proposition \ref{transgeommap}$a$) maps strata of $\S_\textrm{Ham}(\underline{S}_1)$ onto strata of $\S_\textrm{Ham}(\underline{S}_2)$, and the same goes for $\S_\textrm{Gp}(\underline{S}_1)$ and $\S_\textrm{Gp}(\underline{S}_2)$. So, in view of Proposition \ref{moreqisoredring} $h_Q$ is an isomorphism of smooth stratified spaces, for both of these stratifications. By Proposition \ref{transgeommap}$a$, $h_Q$ identifies the reduced space at a leaf $\L_1$ with the reduced space at the leaf $\L_2:=h_P(\L_1)$ (these being the fibers of $\underline{J}_1$ and $\underline{J}_2$) and it is clear that it maps strata of $\S_\textrm{Ham}(\underline{S}_{\L_1})$ onto strata of $\S_\textrm{Ham}(\underline{S}_{\L_2})$. So, by Remark \ref{morphsmstrspsimp} it restricts to an isomorphism of smooth stratified spaces between these reduced spaces. To prove the proposition we are left to show that $h_Q$ is a Poisson map, for it will then restrict to a Poisson map between the reduced spaces and between the strata as well. To this end, let $U_1$ and $U_2$ be $Q$-related invariant opens in $S_1$ and $S_2$. By the proof of Proposition \ref{moreqisoredring}, the Hamiltonian Morita equivalence induces isomorphisms:
\begin{center}
\begin{tikzpicture} \node (S_1) at (0,0) {$\mathcal{C}_{S_1}^\infty(U_1)^{\G_1}$};
\node (S_2) at (8,0) {$\mathcal{C}_{S_2}^\infty(U_2)^{\G_2}$};
\node (Q) at (4,1) {$\mathcal{C}_Q^\infty(\beta_1^{-1}(U_1))^{\G_1}\cap \mathcal{C}_Q^\infty(\beta_2^{-1}(U_2))^{\G_2}$};

\draw[->](S_1) to node[pos=0.45, below] {$\text{ }\text{ }\beta_1^*$} (Q);
\draw[->](S_2) to node[pos=0.45, below] {$\text{ }\text{ }\beta_2^*$} (Q);
\end{tikzpicture}
\end{center} and to prove that $h_Q$ is a Poisson map we have to show that $(\beta_2^*)^{-1}\circ \beta_1^*$ is an isomorphism of Poisson algebras. To see this, let $f_1,h_1\in \mathcal{C}_{S_1}^\infty(U_1)^{\G_1}$ and $f_2,h_2\in \mathcal{C}_{S_2}^\infty(U_2)^{\G_2}$ such that $\beta_1^*f_1=\beta_2^*f_2$ and $\beta_1^*h_1=\beta_2^*h_2$. Let $p_1\in U_1$, $p_2\in U_2$ and $q\in Q$ such that $p_1=\beta_1(q)$ and $p_2=\beta_2(q)$. As we have seen in Lemma \ref{hamflowinvfunlem} it holds that $X_{f_1}(p_1)\in \ker(\d J_1)$. So, as in the proof of Proposition \ref{transgeomham} we can find $\widehat{v}\in \ker(\d j_q)$ such that $\d \beta_1(\widehat{v})=X_{f_1}(p_1)$. It follows from (\ref{hameqmor1}) that:
\begin{align*} \omega_2(X_{f_2}(p_2),\d \beta_2(\cdot))&=\d(\beta_2^*f_2)_q\\
&=\d (\beta_1^*f_1)_q\\
&=(\beta_1^*\omega_1)(\widehat{v},\cdot)\\
&=(\beta_2^*\omega_2)(\widehat{v},\cdot)=\omega_2(\d \beta_2(\widehat{v}),\d \beta_2(\cdot)),
\end{align*} so that, since $\beta_2$ is a submersion, we find that $\d \beta_2(\widehat{v})=X_{f_2}(p_2)$. Using this we see that:
\begin{align*} \{f_1,h_1\}_{\omega_1}(p_1)&=\d h_1(X_{f_1}(p_1))\\
&=\d(\beta_1^*h_1)(\widehat{v})\\
&=\d(\beta_2^*h_2)(\widehat{v})\\
&=\d h_2(X_{f_2}(p_2))=\{f_2,h_2\}_{\omega_2}(p_2),
\end{align*} which proves that $(\beta_2^*)^{-1}\circ \beta_1^*$ is indeed an isomorphism of Poisson algebras. 
\end{proof} 

 \begin{rem} From the above proposition it follows that $\mathcal{P}_\textrm{Ham}(\underline{S})$ and $\mathcal{P}_\textrm{Ham}(\underline{S}_\L)$ are in fact Poisson homogeneous, meaning that they are smoothly homogeneous as in Definition \ref{homdefi}, with the extra requirement that the isomorphisms $h$ can be chosen to be Poisson maps. This gives another proof of the fact that the Poisson structures on the strata of $\S_\textrm{Ham}(\underline{S})$ must be regular.  
\end{rem}

\subsection{Symplectic integration of the canonical Hamiltonian strata}\label{sympintstratsec}
\subsubsection{The integration theorem}
The main theorem of this section is:
\begin{thm}\label{poisstratintgrthm} Let $(\G,\Omega)$ be a proper symplectic groupoid and suppose that we are given a Hamiltonian $(\G,\Omega)$-action along $J:(S,\omega)\to M$. Let $\underline{\Sigma}_S\in \S_\textrm{Ham}(\underline{S})$ and let $\pi_{\underline{\Sigma}_S}$ be the Poisson structure on $\underline{\Sigma}_S$ of Theorem \ref{poisstratthm}. There is a naturally associated proper symplectic groupoid (the symplectic leaves of which may be disconnected) that integrates $(\underline{\Sigma}_S,\pi_{\underline{\Sigma}_S})$.
\end{thm}

Our proof consists of two main steps: first we prove the theorem for Hamiltonian actions of principal type (defined below), and then we show how to reduce to actions of this type. 
\subsubsection{Hamiltonian actions of principal type} 
\begin{defi}\label{princtypdefi} We say that:
\begin{itemize}\item[i)] a proper Lie groupoid $\G\rightrightarrows M$ of \textbf{principal type} if $M^\textrm{princ}=M$ (see Example \ref{princtypliegpoidex}),
\item[ii)] a Hamiltonian action of a proper symplectic groupoid $(\G,\Omega)$ along $J:(S,\omega)\to M$ is of \textbf{principal type} if $S^\textrm{princ}_\textrm{Ham}=S$ and $M^\textrm{princ}=M$ (see Subsection \ref{liedescrpregpartsec}). 
\end{itemize}
\end{defi}
\begin{rem}\label{pringpoidrem} Notice that:
\begin{itemize} \item[i)] a proper Lie groupoid $\G\rightrightarrows M$ with connected leaf space $\underline{M}$ is of principal type if and only if $\G_x$ is isomorphic to $\G_y$ for all $x,y\in M$.
\item[ii)] a Hamiltonian action of a proper symplectic groupoid $(\G,\Omega)$ along $J:(S,\omega)\to M$ with connected orbit space $\underline{S}$ and connected leaf space $\underline{M}$ is of principal type if and only if $\G_p$ is isomorphic to $\G_q$ for all $p,q\in S$ and $\G_x$ is isomorphic to $\G_y$ for all $x,y\in M$.
\end{itemize}
\end{rem}
For the rest of this subsection, let $(\G,\Omega)\rightrightarrows M$ be a proper symplectic groupoid and suppose that we are given a Hamiltonian $(\G,\Omega)$-action of principal type along $J:(S,\omega)\to M$, for which both the orbit space $\underline{S}$ and the leaf space $\underline{M}$ are connected. Then both $\underline{S}$ and $\underline{M}$ are smooth manifolds and $J:S\to M$, as well as $\underline{J}:\underline{S}\to \underline{M}$, is of constant rank. If the action happens to be free, then $J$ is a submersion and the gauge construction (\cite[Theorem 3.2]{Xu}) yields a proper symplectic groupoid integrating $(\underline{S},\pi_{\underline{S}})$. This groupoid is obtained as quotient of the submersion groupoid: 
\begin{equation*}\label{submgpoid} S\times_MS\rightrightarrows S
\end{equation*} by the diagonal action of $\G$ on $S\times_MS$ along $J\circ \textrm{pr}_1$. As we will now show, this construction can be generalized to arbitrary Hamiltonian actions of principal type (for which the action need not be free). To this end, we consider to the subgroupoid:
\begin{equation*} \mathcal{R}=\{(p_1,p_2)\in S\times S\mid J(p_1)=J(p_2)\textrm{ and }\G_{p_1}=\G_{p_2}\}
\end{equation*} of the pair groupoid $S\times S$.
\begin{thm}\label{intgpoidpropprintyp} The groupoid $\mathcal{R}$ has the following properties.
\begin{itemize}\item[a)] It is a closed embedded Lie subgroupoid of the pair groupoid $S\times S$.
\item[b)] It is invariant under the diagonal action of $\G$ on $S\times_MS$, the restriction of the action to $\mathcal{R}$ is smooth, $\underline{\mathcal{R}}:=\mathcal{R}/\G$ is a smooth manifold and the orbit projection $\mathcal{R}\to \underline{\mathcal{R}}$ is a submersion. 
\item[c)] The symplectic pair groupoid $(S\times S,\omega\oplus -\omega)$ descends to give a proper symplectic groupoid: 
\begin{equation*} (\underline{\mathcal{R}},\Omega_{\underline{\mathcal{R}}})\rightrightarrows \underline{S},
\end{equation*} that integrates $(\underline{S},\pi_{\underline{S}})$.
\end{itemize}
\end{thm}
\begin{proof}[Proof of Theorem \ref{intgpoidpropprintyp}; part $a$] We will first use the normal form to study the subspace $S\times_MS$. To this end, let $(p_1,p_2)\in S\times_MS$ and let $x:=J(p_1)=J(p_2)$. Then, as in the proof of Theorem \ref{righamthm}, we can find two neighbourhood equivalences $(\Phi,\Psi_1)$ and $(\Phi,\Psi_2)$ between the given Hamiltonian action and the two local models for it around the respective orbits $\O_{p_1}$ and $\O_{p_2}$ through $p_1$ and $p_2$, using one and the same isomorphism of symplectic groupoids $\Phi$ for both neighbourhood equivalences. Using this, the subset $S\times _MS$ of $S\times S$ is identified near $(p_1,p_2)$ with the subset $S_{\theta,1}\times_{M_\theta} S_{\theta,2}$ of the product $S_{\theta,1}\times S_{\theta,2}$ of the local models around $\O_{p_1}$ and $\O_{p_2}$ (using the notation of Subsection \ref{hamlocmodconsec}) near $(\Psi_1(p_1),\Psi_2(p_2))$. Since we assume the Hamiltonian action to be of principal type, the coadjoint $\G_x$-action and the actions underlying the symplectic normal representations at $p_1$ and $p_2$ are trivial (cf. Proposition \ref{normrepsymp}$b$, Example \ref{princtypliegpoidex} and Proposition \ref{regpartalg}). So, denoting by $P$ the source-fiber of $\G$ over $x$, the momentum maps $J_{\theta,i}:S_{\theta,i}\to M_\theta$ in the local model become:
\begin{equation}\label{mommaplocmoddescpprinctyp} P/\G_{p_i}\times (\g_{p_i}^0\oplus \S\No_{p_i})\to P/\G_x\times \g_x^*, \quad ([q],\alpha,v)\mapsto ([q],\alpha), 
\end{equation} (or rather, a restriction of this to an open neighbourhood of the central orbit $P/\G_{p_i}$) for $i\in \{1,2\}$. From this we see that $S_{\theta,1}\times_{M_\theta} S_{\theta,2}$ is a submanifold of $S_{\theta,1}\times S_{\theta,2}$ with tangent space given by all pairs of tangent vectors $(v_1,v_2)$ satisfying $\d J_{\theta,1}(v_1)=\d J_{\theta,2}(v_2)$. Passing back to $S\times S$ via $(\Psi_1,\Psi_2)$, we find that $S\times_MS$ is an embedded submanifold of $S\times S$ at $(p_1,p_2)$ with tangent space:
\begin{equation}\label{destngtspeq} \{(v_1,v_2)\in T_{p_1}S\times T_{p_2}S\mid \d J_{p_1}(v_1)=\d J_{p_2}(v_2)\}.
\end{equation} We now turn to $\mathcal{R}$. As we will show in a moment, $\mathcal{R}$ is both open and closed in $S\times_MS$. Together with the above, this would show that $\mathcal{R}$ is a closed embedded submanifold of $S\times S$ (with connected components of possibly varying dimension), the tangent space of which is given by (\ref{destngtspeq}). To then show that $\mathcal{R}$ is an embedded Lie subgroupoid (with connected components of one and the same dimension), it would be enough to show the two projections $\mathcal{R}\to S$ are submersions. In view of the description (\ref{destngtspeq}) of the tangent space of $\mathcal{R}$ this is equivalent to the requirement that $\im(\d J_{p_1})=\im(\d J_{p_2})$ for all $(p_1,p_2)\in \mathcal{R}$, which is indeed satisfied, as follows from Proposition \ref{infmomact}$b$. So, to prove part $a$ it remains to show that $\mathcal{R}$ is both open and closed in $S\times_MS$.\\

To prove that $\mathcal{R}$ is closed in $S\times_MS$, we will show that every $(p_1,p_2)\in S\times_MS$ admits an open neighbourhood that intersects $\mathcal{R}$ in a closed subset of this neighbourhood. Given such $(p_1,p_2)$, as before, we pass to the local models around $\O_{p_1}$ and $\O_{p_2}$ using $(\Phi,\Psi_1)$ and $(\Phi,\Psi_2)$. From the description (\ref{mommaplocmoddescpprinctyp}) we find that $S_{\theta,1}\times_{M_\theta}S_{\theta,2}$ is the subset of $S_{\theta,1}\times S_{\theta,2}$ consisting of pairs:
\begin{equation*} (([q_1],\alpha_1,v_1),([q_2],\alpha_2,v_2))
\end{equation*} satisfying:
\begin{equation*} [q_1]=[q_2]\in P/\G_x \quad \& \quad \alpha_1=\alpha_2\in \g_x^*.
\end{equation*} 
Furthermore, a straightforward verification shows that $(\Psi_1,\Psi_2)$ identifies $\mathcal{R}$ near $(p_1,p_2)$ with the subset of those pairs that in addition satisfy:
\begin{equation}\label{condRset}  [q_1:q_2]\in N_{\G_x}(\G_{p_1},\G_{p_2}):=\{g\in \G_x\mid g\G_{p_1}g^{-1}=\G_{p_2}\}.
\end{equation} Notice that $N_{\G_x}(\G_{p_1},\G_{p_2})$ is closed in $\G_x$ and invariant under left multiplication by elements of $\G_{p_2}$ and under right multiplication by elements of $\G_{p_1}$, so that it corresponds to a closed subset of: 
\begin{equation*} \G_{p_2}\backslash \G_x /\textrm{ }\G_{p_1}.
\end{equation*} Hence, by continuity of the map:
\begin{equation*} (P/\G_{p_1})\times_{P/\G_x} (P/\G_{p_2})\to \G_{p_2}\backslash \G_x /\textrm{ }\G_{p_1},\quad ([q_1],[q_2])\mapsto [q_1:q_2] \mod \G_{p_2}\times \G_{p_1}
\end{equation*} it follows that (\ref{condRset}) is a closed condition in $S_{\theta,1}\times_{M_\theta}S_{\theta,2}$. So, $\mathcal{R}$ is indeed closed in $S\times_MS$.\\

To show that $\mathcal{R}$ is open in $S\times_MS$ we can argue in exactly the same way, now restricting attention to pairs $(p_1,p_2)\in \mathcal{R}$, so that the condition (\ref{condRset}) becomes:
\begin{equation*} [q_1:q_2]\in N_{\G_x}(\G_{p_1}),
\end{equation*} 
where $N_{\G_x}(\G_{p_1})$ denotes the normalizer of $\G_{p_1}$ in $\G_x$, and we are left to show that $N_{\G_x}(\G_{p_1})$ is open in $\G_x$. To this end, recall from before that the coadjoint action of $\G_x$ on $\g_x^*$ is trivial. So, the action by conjugation of $\G_x$ on its identity component $\G^0_x$ is trivial. This can be rephrased as saying that the action by conjugation of $\G^0_x$ on $\G_x$ is trivial. In particular, $\G^0_x$ is contained in $N_{\G_x}(\G_{p_1})$ and therefore the Lie subgroup $N_{\G_x}(\G_{p_1})$ is indeed open in $\G_x$. This concludes the proof of part $a$.
\end{proof}
For the proof of part $b$ we recall the lemma below, which follows from the linearization theorem for proper Lie groupoids (see e.g. the proof of \cite[Proposition 23]{CrMe} for details).
\begin{lemma}\label{orbspsmoothlem} Let $\G\rightrightarrows M$ be a proper Lie groupoid with a single isomorphism type, meaning that $\G_x$ is isomorphic to $\G_y$ for all $x,y\in M$ (see Example \ref{exisotyp}). Then the leaf space $(\underline{M},\mathcal{C}^\infty_{\underline{M}})$ is a smooth manifold and the projection $M\to \underline{M}$ is a submersion.
\end{lemma}

\begin{proof}[Proof of Theorem \ref{intgpoidpropprintyp}; parts $b$ and $c$] It is readily verified that $\mathcal{R}$ is invariant under the diagonal $\G$-action along $J\circ \textrm{pr}_1:S\times_MS\to M$ and that the restricted action is smooth. Since $\G$ is proper, so is the action groupoid $\G\ltimes \mathcal{R}$. Furthermore, the isotropy group of the $\G$-action at $(p,q)\in \mathcal{R}$ is the isotropy of the $\G$-action on $S$ at $p$. So, since the isotropy groups of the action on $S$ are all isomorphic (by Remark \ref{pringpoidrem}), the same holds for the isotropy groups of the action on $\mathcal{R}$. In view of Lemma \ref{orbspsmoothlem}, we conclude that part $b$ holds. \\

We turn to part $c$. One readily verifies that $\underline{\mathcal{R}}$ inherits the structure of a Lie groupoid over $\underline{S}$ from the Lie groupoid $\mathcal{R}\rightrightarrows S$. To see that the Lie groupoid $\underline{\mathcal{R}}$ is proper, suppose that we are given a sequence of $[p_n,q_n]\in \underline{\mathcal{R}}$ with the property that $t_{\underline{\mathcal{R}}}([p_n,q_n])=[p_n]$ and $s_{\underline{\mathcal{R}}}([p_n,q_n])=[q_n]$ converge in $\underline{S}$ as $n\to \infty$. We have to show that the given sequence in $\underline{\mathcal{R}}$ admits a convergent subsequence. Since the orbit projection $S\to \underline{S}$ is a surjective submersion, it admits local sections around all points in $\underline{S}$. Using this, we can (for $n$ large enough) find $g_n, h_n\in \G$ in the source fiber over $J(p_n)=J(q_n)$ such that $g_n\cdot p_n$ and $h_n\cdot q_n$ converge in $S$ as $n\to \infty$. Then $t_\G(g_nh_n^{-1})=J(g_n\cdot p_n)$ and $s_\G(g_nh_n^{-1})=J(h_n\cdot q_n)$ both converge in $M$ as $n\to \infty$. By properness of $\G$, it follows that there is a subsequence $g_{n_k}h_{n_k}^{-1}$ that converges in $\G$ as $k\to \infty$. Together with convergence of $h_{n_k}\cdot q_{n_k}$, this implies that $g_{n_k}\cdot q_{n_k}$ converges in $S$ as well. So, since $\mathcal{R}$ is closed in $S\times S$, it follows that $g_{n_k}\cdot (p_{n_k},q_{n_k})$ converges in $\mathcal{R}$. Therefore, $[p_{n_k},q_{n_k}]$ converges in $\underline{\mathcal{R}}$. This shows that the required subsequence exists and hence proves properness of the Lie groupoid $\underline{\mathcal{R}}$. \\

To complete the proof of $c$, we are left to show that the symplectic structure on the pair groupoid $S\times S$ descends to a symplectic structure $\Omega_{\underline{\mathcal{R}}}$ on $\underline{\mathcal{R}}$, and that $(\underline{\mathcal{R}},\Omega_{\underline{\mathcal{R}}})$ integrates $(\underline{S},\pi_{\underline{S}})$. To see that the restriction $\Omega_\mathcal{R}\in \Omega^2(\mathcal{R})$ of $\omega\oplus -\omega$ to $\mathcal{R}$ descends to a $2$-form on $\underline{\mathcal{R}}$, recall that this is equivalent to asking that $\Omega_\mathcal{R}$ is basic with respect to the $\G$-action on $\mathcal{R}$ (in the sense of \cite{PfPoTa,Wa,Yu}), which means that: $m_\mathcal{R}^* \Omega_\mathcal{R}=\textrm{pr}_\mathcal{R}^*\Omega_\mathcal{R}$, where $m_\mathcal{R},\textrm{pr}_\mathcal{R}:\G\ltimes \mathcal{R}\to \mathcal{R}$ denote the target and source map of the action groupoid. This equality is readily verified. So, $\Omega_\mathcal{R}$ indeed descends to a $2$-form $\Omega_{\underline{\mathcal{R}}}$ on $\underline{\mathcal{R}}$. Further notice that $\Omega_{\underline{\mathcal{R}}}$ is closed (because $\omega$ is closed) and it inherits multiplicativity from the multiplicative form $\omega\oplus -\omega$ on the pair groupoid $S\times S$. Moreover, using the momentum map condition (\ref{mommapcond}), Proposition \ref{infmomact} and the description (\ref{destngtspeq}) of the tangent space to $\mathcal{R}$, it is straightforward to check that $\Omega_{\underline{\mathcal{R}}}$ is non-degenerate. So, $(\underline{\mathcal{R}},\Omega_{\underline{\mathcal{R}}})$ is a symplectic groupoid. We leave it to the reader to verify that $(\underline{\mathcal{R}},\Omega_{\underline{\mathcal{R}}})$ integrates $(\underline{S},\pi_{\underline{S}})$. \end{proof}

\subsubsection{Reduction to Hamiltonian actions of principal type}
The aim of this subsection is to show that the restriction of a given Hamiltonian action (by a proper symplectic groupoid) to any stratum of $\S_\textrm{Ham}(\underline{S})$ can be reduced to a  Hamiltonian action of principal type. More precisely, we prove:
\begin{thm}\label{redstratthm} Let $(\G,\Omega)$ be a proper symplectic groupoid and suppose that we are given a Hamiltonian $(\G,\Omega)$-action along $J:(S,\omega)\to M$. Let $\underline{\Sigma}_S\in \S_\textrm{Ham}(\underline{S})$ and let $\underline{\Sigma}_M\in \S_\textrm{Gp}(\underline{M})$ be such that $\underline{J}(\underline{\Sigma}_S)\subset\underline{\Sigma}_M$. Finally, let $q_S:S\to \underline{S}$ and $q_M:M\to \underline{M}$ be the orbit and leaf space projections, and consider $\Sigma_S=q_S^{-1}(\underline{\Sigma}_S)$ and $\Sigma_M=q_M^{-1}(\underline{\Sigma}_M)$. Then the following hold.
\begin{itemize}\item[a)] The restriction $\omega_{\Sigma_S}\in \Omega^2(\Sigma_S)$ of the symplectic form $\omega$ to $\Sigma_S$ has constant rank. Moreover, the null foliation integrating $\ker(\omega_{\Sigma_S})$ is simple, meaning that its leaf space admits a smooth manifold structure with respect to which the leaf space projection is a submersion. 
\end{itemize} Let $S_{\Sigma}$ denote this leaf space and let $\omega_{S_\Sigma}$ denote the induced symplectic form on $S_{\Sigma}$. 
\begin{itemize}
\item[b)] The restriction of $\Omega$ to $\G\vert_{\Sigma_M}$ has constant rank and the leaf space of its null foliation is naturally a proper symplectic groupoid $(\G_{\Sigma_M},\Omega_{\Sigma_M})$ over $\Sigma_M$. 
\item[c)] The map $J$ descends to a map: 
\begin{equation*} J_{S_\Sigma}:(S_{\Sigma},\omega_{S_\Sigma})\to \Sigma_M
\end{equation*} and the Hamiltonian $(\G,\Omega)$-action along $J$ descends to a Hamiltonian $(\G_{\Sigma_M},\Omega_{\Sigma_M})$-action along $J_{S_\Sigma}$, which is of principal type.
\item[d)] There is a canonical Poisson diffeomorphism: 
\begin{equation*} (\underline{\Sigma}_S,\pi_{\underline{\Sigma}_S})\cong (\underline{S_{\Sigma}},\pi_{\underline{S_{\Sigma}}}).
\end{equation*} 
\end{itemize}
\end{thm}
Together with Theorem \ref{intgpoidpropprintyp}, this would prove Theorem \ref{poisstratintgrthm}. To prove Theorem \ref{redstratthm}, we use the following Lie theoretic description of the null foliation of $\omega_{\Sigma_S}$. Recall that, given a real finite-dimensional representation $V$ of a compact Lie group $G$, the fixed-point set $V^G$ has a canonical $G$-invariant linear complement $\c_V$ in $V$, given by the linear span of the collection:
\begin{equation*} \{v-g\cdot v\mid v\in V, g\in G\}. 
\end{equation*} To see that $\c_V$ is indeed a linear complement to $V^G$, note that for any choice of $G$-invariant inner product on $V$, $\c_V$ coincides with the orthogonal complement to $V^G$ in $V$. We will call $\c_V$ the \textbf{fixed-point complement} of $V$. For the dual representation $V^*$, it holds that:
\begin{equation}\label{coinv} (V^*)^G=(\c_V)^0,
\end{equation} the annihilator of $\c_V$ in $V$. Of particular interest will be the adjoint representation. 
\begin{prop}\label{fixptcompliealg} Let $G$ be a compact Lie group. The fixed-point complement $\c_\g$ of the adjoint representation is a Lie subalgebra of $\g$, given by:
\begin{equation*} \c_\g=\c_{Z(\g)}\oplus \g^{\textrm{ss}},
\end{equation*} where $Z(\g)$ is the center (viewed as $G$-representation) and $\g^\textrm{ss}=[\g,\g]$ is the semi-simple part of $\g$. 
\end{prop}
\begin{proof} This follows from the observation that $Z(\g)$ is the fixed-point set for the adjoint action of the identity component of $G$ and $[\g,\g]$ is the orthogonal complement to $Z(\g)$ in $\g$ with respect to any invariant inner product. 
\end{proof}
We now give the aforementioned description of the null foliation. 
\begin{lemma}\label{liedescrisofol} Let $p\in \Sigma_S$ and $x=J(p)\in \Sigma_M$, with notation as in Theorem \ref{redstratthm}. Let \begin{equation*} a_J:J^*(T^*M)\to TS
\end{equation*} be the bundle map underlying the infinitesimal action (\ref{assliealgact}) associated to the Hamiltonian action. Further, let $\c_{\g_x}$ denote the fixed-point complement of the adjoint representation of $\G_x$. Then:
\begin{equation*} \ker(\omega_{\Sigma_S})_p=(a_J)_p(\c_{\g_x}),
\end{equation*} where we view $\g_x\subset T_x^*M$ via (\ref{imsymp}). 
\end{lemma} 
\begin{proof} Because (by Corollary \ref{orbprojstratdirac}) the orbit projection (\ref{orbprojdirac}) is a forward Dirac map from the pre-symplectic manifold $(\Sigma_S,\omega_{\Sigma_S})$ into a Poisson manifold, it must hold that:
\begin{equation*} \ker(\omega_{\Sigma_S})_p\subset T_p\O.
\end{equation*} Since $T_p\O\subset T_p\Sigma_S$, it also holds that: 
\begin{equation*} \ker(\omega_{\Sigma_S})_p\subset T_p\O^\omega.
\end{equation*} For any Hamiltonian action, we have the equality:
\begin{equation*} T_p\O\cap T_p\O^\omega=(a_J)_p(\g_x),
\end{equation*} as is readily derived from the momentum map condition (\ref{mommapcond}). So, we conclude that:
\begin{equation*} \ker(\omega_{\Sigma_S})_p\subset (a_J)_p(\g_x).
\end{equation*} Now consider the composition of maps: 
\begin{equation}\label{compmapsliedescrisofol} \frac{T_p\Sigma_S}{T_p\O}\hookrightarrow \No_p\xrightarrow{\underline{\d J}_p}\No_x\xrightarrow{\sim}\g_x^*,
\end{equation} where the third map is dual to the canonical isomorphism between $\g_x$ (which via (\ref{imsymp}) we view as the annihilator of $T_x\L$ in $T_x^*M$) and $\No_x^*$. Using a Hamiltonian Morita equivalence as in the proof of Proposition \ref{eqcharhammortyp}, together with Proposition \ref{normrepham}$b$, Proposition \ref{transgeommap}$c$, Lemma \ref{techlemisotype} and Morita invariance of the $J$-isomorphism types, it is readily verified that the image of (\ref{compmapsliedescrisofol}) is $(\g_p^0)^{\G_x}$. From this and the momentum map condition (\ref{mommapcond}) it follows that, given $\alpha\in \g_x$, the tangent vector $(a_J)_p(\alpha)$ belongs to $\ker(\omega_{\Sigma_S})_p$ if and only if $\alpha$ belongs to the annihilator of $(\g_p^0)^{\G_x}$. This annihilator equals $\g_p+\c_{\g_x}$, as (\ref{coinv}) implies that:
\begin{equation*} (\g_p^0)^{\G_x}:=\g_p^0\cap (\g_x^*)^{\G_x}=(\g_p+\c_{\g_x})^0.
\end{equation*} So, all together it follows that:
\begin{equation*} \ker(\omega_{\Sigma_S})_p=(a_J)_p(\g_p+\c_{\g_x})=(a_J)_p(\c_{\g_x}),
\end{equation*} which proves the lemma. 
\end{proof}
We can interpret this as follows: the $T_\pi^*M$-action associated to the momentum map $J$ restricts to an infinitesimal action of the bundle of Lie algebras:
\begin{equation*} \bigsqcup_{x\in \Sigma_M}\c_{\g_x}\subset T^*M\vert_{\Sigma_M}
\end{equation*} and the orbit distribution of this infinitesimal action coincides with the distribution $\ker(\omega_{\Sigma_S})$. The proof of Theorem \ref{redstratthm} therefore boils down to showing that this infinitesimal action integrates to an action of a bundle of Lie groups in $\G$, the orbit space of which is smooth. 

\begin{prop}[\cite{CrFeTo2}]\label{gpoidredstrat} Let $(\G,\Omega)\rightrightarrows M$ be a proper symplectic groupoid, let $\underline{\Sigma}\in \S_\textrm{Gp}(\underline{M})$ and let $\Sigma=q^{-1}(\underline{\Sigma})$, for $q:M\to \underline{M}$ the leaf space projection. Consider the family of Lie groups:
\begin{equation*} \bigsqcup_{x\in \Sigma} C_{\g_x}\subset \G\vert_{\Sigma}
\end{equation*} where $C_{\g_x}$ is the unique connected Lie subgroup of $\G_x$ that integrates $\c_{\g_x}$. This defines a closed, embedded and normal Lie subgroupoid of $\G\vert_{\Sigma}$ and the quotient of $\G\vert_{\Sigma}$ by this bundle of Lie groups is naturally a proper symplectic groupoid over $\Sigma$ of principal type.
\end{prop}
\begin{proof} First, observe that for any compact Lie group $G$, the connected Lie subgroup $C_{\g}$ of $G$ with Lie algebra $\c_\g$ is compact. To see this, let $G^\textrm{ss}$ be the connected Lie subgroup of $G$ with Lie algebra the compact and semisimple Lie subalgebra $\g^\textrm{ss}=[\g,\g]$ of $\g$, let $G^0$ be the identity component of $G$ and let $Z(G^0)^0$ denote the identity component of the center of $G^0$. Fix $g_1,...,g_n\in G$ such that $G/G^0=\{[g_1],...,[g_n]\}$. It follows from Proposition \ref{fixptcompliealg} that $C_\g$ is the image of the morphism of Lie groups:
\begin{equation*} \left(Z(G^0)^0\right)^n\times G^\textrm{ss}\to G, \quad (h_1,...,h_n,g)\mapsto [h_1,g_1]\cdot ...\cdot [h_n,g_n]\cdot g,
\end{equation*} where $[h_i,g_i]=h_ig_ih_i^{-1}g_i^{-1}$ is the commutator (which again belongs to $Z(G^0)^0$). So, since both $G^\textrm{ss}$ and $Z(G^0)^0$ are compact, $C_\g$ is compact as well. Using this and the linearization theorem for proper Lie groupoids (see Subsection \ref{locmodconstsec} for the local model) one sees that the family of the Lie groups $C_{\g_x}$ is a closed embedded Lie subgroupoid of $\G\vert_{\Sigma}$ over $\Sigma$. Furthermore, for every $g\in \G\vert_\Sigma$ starting at $x$ and ending at $y$, it holds that:
\begin{equation*} gC_{\g_x}g^{-1}=C_{\g_y}.
\end{equation*} This follows from the observation that an isomorphism of compact Lie groups $G_1\to G_2$ maps $C_{\g_1}$ onto $C_{\g_2}$. So, the family of Lie groups is also a normal subgroupoid of $\G\vert_{\Sigma}$. Therefore, the quotient of the proper Lie groupoid $\G\vert_{\Sigma}$ by this bundle of Lie groups is again a proper Lie groupoid. It follows from Lemma \ref{liedescrisofol}, applied to the Hamiltonian action of Example \ref{identityactionex}, that the pre-symplectic form $\Omega\vert_{(\G\vert_{\Sigma})}$ on $\G\vert_{\Sigma}$ has constant rank and its null foliation coincides with the foliation by orbits of the action on $\G\vert_{\Sigma}$ of this bundle of Lie groups. So, the quotient groupoid inherits a symplectic form. This symplectic form inherits multiplicativity from $\Omega$. Hence, the quotient is a symplectic groupoid. Finally, in light of Remark \ref{pringpoidrem} the quotient groupoid is of principal type, because for any $x,y\in \Sigma$ there is an isomorphism between $\G_x$ and $\G_y$, and any such isomorphism descends to one between the isotropy groups $\G_x/C_{\g_x}$ and $\G_y/C_{\g_y}$ of the quotient groupoid. 
\end{proof}
We are now ready to complete the proof of the reduction theorem.
\begin{proof}[Proof of Theorem \ref{redstratthm}] Consider the family of Lie groups:
\begin{equation*} \H_{\Sigma_M}:=\bigsqcup_{x\in \Sigma_M} C_{\g_x}\subset \G\vert_{\Sigma_M}
\end{equation*} of Proposition \ref{gpoidredstrat}. Being a closed embedded Lie subgroupoid of the proper Lie groupoid $\G\vert_{\Sigma_M}$, the Lie groupoid $\H_{\Sigma_M}$ is proper as well. Hence, so is any smooth action of $\H_{\Sigma_M}$. It acts along $J:\Sigma_S\to \Sigma_M$ via the action of $\G$. Proposition \ref{eqcharhammortyp} implies that for any $p,q\in \Sigma_S$, writing $x=J(p)$ and $y=J(q)$, there is an isomorphism of pairs of Lie groups:
\begin{equation}\label{isopairliegp} (\G_{x},\G_p)\cong (\G_{y},\G_q).
\end{equation} Such an isomorphism restricts to an isomorphism between the isotropy groups of the $\H_{\Sigma_M}$-action:
\begin{equation*} (\H_{\Sigma_M})_p=C_{\g_{x}}\cap \G_p\cong C_{\g_{y}}\cap \G_q=(\H_{\Sigma_M})_q.
\end{equation*} So, appealing to Lemma \ref{orbspsmoothlem}, we find that the orbit space admits a smooth manifold structure for which the orbit projection is a submersion. It follows from Lemma \ref{liedescrisofol} that the orbits of this action are the leaves of the null foliation of $\omega_{\Sigma_S}$, so this proves part $a$ of the theorem. Part $b$ of the theorem is proved in Proposition \ref{gpoidredstrat}. For part $c$, notice that $J$ factors through to a map $J_{S_\Sigma}$ (since the source and target of any element in $\H_{\Sigma_M}$ coincide) and the action of $\G$ along $J$ descends to an action of $\G_{\Sigma_M}$ along $J_{S_\Sigma}$. As the action of $(\G,\Omega)$ along $J$ is Hamiltonian, the same follows for the action of $(\G_{\Sigma_M},\Omega_{\Sigma_M})$ along $J_{S_\Sigma}$. By the previous proposition, $\G_{\Sigma_M}$ is of principal type. Furthermore, for any two $[p],[q]\in S_{\Sigma}$ there is, as before, an isomorphism of pairs (\ref{isopairliegp}), and this descends to an isomorphism between the Lie groups:
\begin{equation*} \G_p/(C_{\g_x}\cap \G_p)\cong \G_q/(C_{\g_y}\cap \G_q),
\end{equation*} which are canonically isomorphic to the respective isotropy groups of the $\G_{\Sigma_M}$-action at $[p]$ and $[q]$. In view of Remark \ref{pringpoidrem} we conclude that the Hamiltonian $(\G_{\Sigma_M},\Omega_{\Sigma_M})$-action is of principal type. This completes the proof of parts $a-c$. For the final statement, consider the diagram:
\begin{center} 
\begin{tikzcd} (\Sigma_S, \omega_{\Sigma_S})\arrow[r] \arrow[d] & (S_{\Sigma},\omega_{S_\Sigma}) \arrow[d] \\
(\underline{\Sigma}_S,\pi_{\underline{\Sigma}_S}) & (\underline{S_{\Sigma}},\pi_{\underline{S_{\Sigma}}})
\end{tikzcd}
\end{center} All arrows are surjective submersions and by construction (in particular, Corollary \ref{orbprojstratdirac}) each is forward Dirac. Evidently, the left vertical map factors through the composition of the other two, and vice versa. Hence, by functoriality of the push-forward construction for Dirac structures, the diagram completes to give the desired Poisson diffeomorphism.   
\end{proof}
\bibliographystyle{plain}

\bibliography{ref}
\Addresses
\end{document}